\RequirePackage{ifthen}
\RequirePackage{ifpdf}
\newcommand{\driverOption}{}
\ifthenelse{\boolean{pdf}}{
  \renewcommand{\driverOption}{pdftex}
} { % else
  \renewcommand{\driverOption}{dvips}
}

\documentclass[reqno, 11pt, letterpaper, commented, oneside, \driverOption]{amsart}

\newboolean{isCommented}
\usepackage{geometry}
\usepackage[final]{graphicx}
\usepackage{caption}
\usepackage{color}
\usepackage{upref}
\usepackage{enumitem}
\usepackage{latexsym}
\usepackage{amssymb}
\usepackage[ansinew]{inputenc}
\usepackage[T1]{fontenc}
\usepackage{mathtools}
\usepackage[colorinlistoftodos]{todonotes}
\usepackage{extarrows}
\usepackage{multirow}
\usepackage{xcolor,colortbl}
\usepackage{pifont}
\usepackage{arydshln}
\usepackage{subcaption}
\usepackage{mycommands}

\newcommand{\hyperrefDriverOption}{}
\ifthenelse{\boolean{pdf}}{
	\renewcommand{\hyperrefDriverOption}{pdftex}
} { % else
	\renewcommand{\hyperrefDriverOption}{hypertex}
}
\usepackage[\hyperrefDriverOption,
  colorlinks = false, 
  pdfauthor={Stefan Felsner, Linda Kleist, Torsten M\"utze, Leon Sering},
  pdftitle={Rainbow cycles in flip graphs}]
  {hyperref}

\usepackage{bookmark} % using this package fixes Arxiv problems with ?? in the table of contents

\usepackage{tikz}
\usepackage{pgfplots}
\pgfdeclarelayer{background}
\pgfdeclarelayer{foreground}
\pgfsetlayers{background,main,foreground}

\tikzset{node_white/.style={circle, draw, fill=white, inner sep=0pt, text width=0pt, text height=0pt, text depth=0pt, minimum size = 3pt}}
\tikzset{node_black/.style={circle, draw, fill=black, inner sep=0pt, text width=0pt, text height=0pt, text depth=0pt, minimum size = 3pt}}
\tikzset{node_circle/.style={circle, draw, fill=white, minimum size = 13pt}}

\tikzset{line_solid/.style={draw, thick}}
\tikzset{line_dashed/.style={draw, thick, dash pattern={on 3pt off 1pt}}}
\tikzset{line_dashdotted/.style={draw, thick, dash pattern={on 3pt off 3pt on 1pt off 3pt}}}
\tikzset{line_dotted/.style={draw, thick, dash pattern={on 1pt off 1pt}}}

\ifthenelse{\boolean{isCommented}} {
	\newcommand{\TM}[1]{\marginpar{\parbox{4cm}{{\small {\bf TM:} #1}}}}
	\newcommand{\SF}[1]{\marginpar{\parbox{4cm}{{\small {\bf SF:} #1}}}}
	\newcommand{\LK}[1]{\marginpar{\parbox{4cm}{{\small {\bf LK:} #1}}}}
	\newcommand{\LS}[1]{\marginpar{\parbox{4cm}{{\small {\bf LS:} #1}}}}
} { % else
	\newcommand{\TM}[1]{}
	\newcommand{\SF}[1]{}
	\newcommand{\LK}[1]{}
	\newcommand{\LS}[1]{}
}
\newtheorem{theorem}{Theorem}
\newtheorem{lemma}[theorem]{Lemma}
\newtheorem{proposition}[theorem]{Proposition}

\theoremstyle{definition}

\theoremstyle{remark}

\newtheorem*{claim}{Claim}

\long\def\symbolfootnote[#1]#2{\begingroup
\def\thefootnote{\fnsymbol{footnote}}\footnote[#1]{#2}\endgroup}

\ifthenelse{\boolean{isCommented}} {
	\geometry{%showframe,
	  hmargin={25mm, 50mm},
	  marginparwidth=40mm, 
	  vmargin={25mm, 25mm},
	  headsep=10mm,
	  headheight=5mm,
	  footskip=10mm
	}
} { % else
	\geometry{%showframe,
	  hmargin={25mm, 25mm}, 
	  vmargin={25mm, 25mm},
	  headsep=10mm,
	  headheight=5mm,
	  footskip=10mm
	}
}

\graphicspath{{./graphics/}}

\begin{document}

\begin{center}

\LARGE Rainbow cycles in flip graphs
\vspace{2mm}

\Large Stefan Felsner, Linda Kleist, Torsten M\"utze, Leon Sering
\vspace{2mm}

\large
  Institut f\"ur Mathematik \\
  TU Berlin, 10623 Berlin, Germany \\
  {\small\tt \{felsner,kleist,muetze,sering\}@math.tu-berlin.de}
\vspace{5mm}

\small

\begin{minipage}{0.8\linewidth}
\textsc{Abstract.}
The flip graph of triangulations has as vertices all triangulations of a convex $n$-gon, and an edge between any two triangulations that differ in exactly one edge.
An $r$-rainbow cycle in this graph is a cycle in which every inner edge of the triangulation appears exactly $r$ times.
This notion of a rainbow cycle extends in a natural way to other flip graphs.
In this paper we investigate the existence of $r$-rainbow cycles for three different flip graphs on classes of geometric objects:
the aforementioned flip graph of triangulations of a convex $n$-gon, the flip graph of plane trees on an arbitrary set of $n$ points, and the flip graph of non-crossing perfect matchings on a set of $n$ points in convex position.
In addition, we consider two flip graphs on classes of non-geometric objects: the flip graph of permutations of $\{1,2,\dots,n\}$ and the flip graph of $k$-element subsets of $\{1,2,\dots,n\}$.
In each of the five settings, we prove the existence and non-existence of rainbow cycles for different values of $r$, $n$ and~$k$.
\end{minipage}
\end{center}

\vspace{5mm}

\section{Introduction}

Flip graphs are fundamental structures associated with families of geometric objects such as triangulations, plane spanning trees, non-crossing matchings, partitions or dissections.
A classical example is the flip graph of triangulations.
The vertices of this graph $G_n^\tT$ are the triangulations of a convex $n$-gon, and two triangulations are adjacent whenever they differ by exactly one edge.
In other words, moving along an edge of $G_n^\tT$ corresponds to flipping the diagonal of a convex quadrilateral formed by two triangles.
Figure~\ref{fig:triang} shows the graph $G_6^\tT$.

\begin{figure}[htb]
\includegraphics{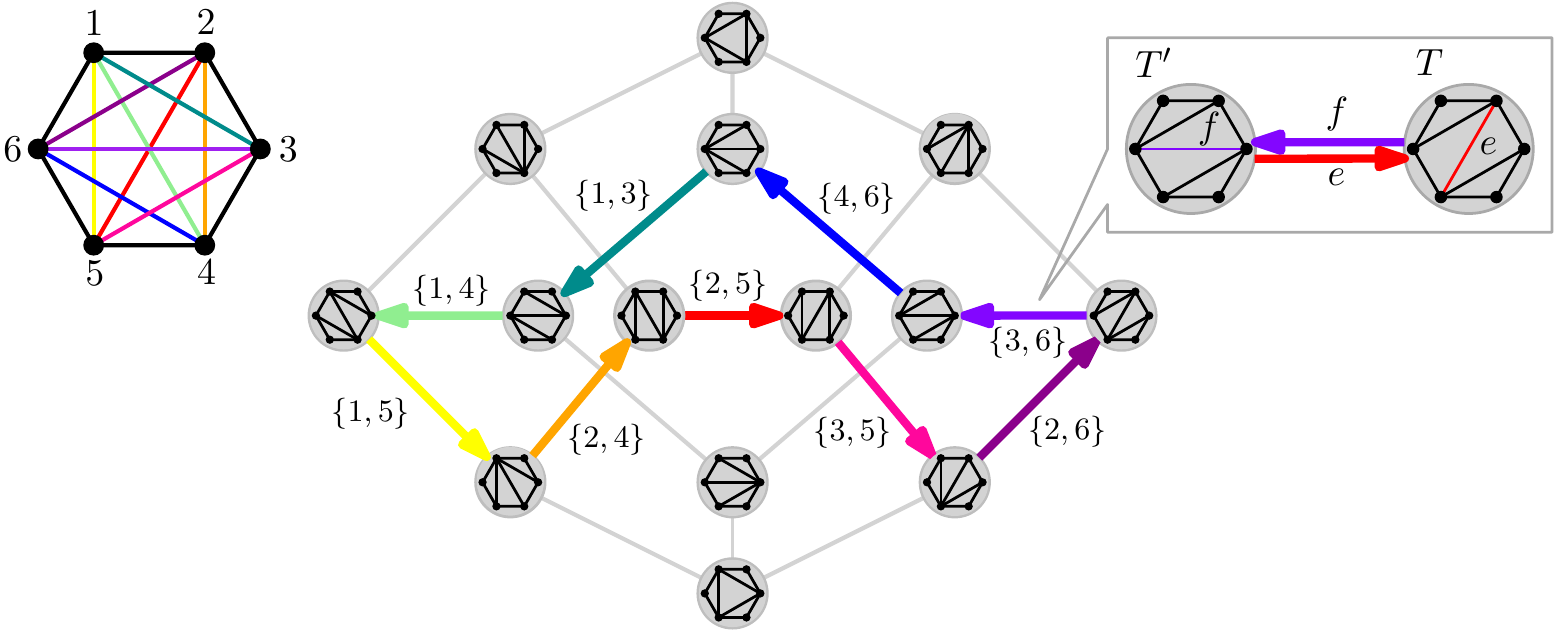}
\caption{The flip graph of triangulations $G_n^\tT$ of a convex $n$-gon for $n=6$, and a rainbow cycle in this graph.}
\label{fig:triang}
\end{figure}

A question that has received considerable attention is to determine the diameter of $G_n^\tT$, i.e., the number of flips that is necessary and sufficient to transform any triangulation into any other; see the survey \cite{MR2455502}.
In a landmark paper \cite{MR928904}, Sleator, Tarjan and Thurston proved that the diameter of $G_n^\tT$ is $2n-10$ for sufficiently large $n$.
Recently, Pournin \cite{MR3197650} gave a combinatorial proof that the diameter is $2n-10$ for all $n>12$.
A challenging algorithmic problem in this direction is to efficiently compute a minimal sequence of flips that transforms two given triangulations into each other; see \cite{MR1744193,DBLP:conf/cocoon/LiZ98}.
These questions involving the diameter of the flip graph become even harder when the $n$ points are not in convex, but in general position; see e.g.~\cite{MR1640808,MR1706610,MR2770955}.
Moreover, apart from the diameter, many other properties of the flip graph $G_n^\tT$ have been investigated, e.g., its realizability as a convex polytope \cite{MR3437894}, its automorphism group \cite{MR1022776}, the vertex-connectivity \cite{MR1723053}, and the chromatic number \cite{MR2535071}.

Another property of major interest is the existence of a Hamilton cycle in $G_n^\tT$.
This was first established by Lucas \cite{MR920505} and a very nice and concise proof was given by Hurtado and Noy \cite{MR1723053}.
The reason for the interest in Hamilton cycles is that a Hamilton cycle in $G_n^\tT$ corresponds to a so-called \emph{Gray code}, i.e., an algorithm that allows to generate each triangulation exactly once, by performing only a single flip operation when moving to the next triangulation.
In general, the task of a Gray code algorithm is to generate all objects in a particular combinatorial class, each object exactly once, by applying only a small transformation in each step, such as a flip in a triangulation.
Combinatorial classes of interest include geometric configurations such as triangulations, plane spanning trees or non-crossing perfect matchings, but also classes without geometric information such as permutations, combinations, bitstrings etc.
This fundamental topic is covered in depth in the most recent volume of Knuth's seminal series \emph{The Art of Computer Programming} \cite{MR3444818}, and in the classical books by Nijenhuis and Wilf \cite{MR510047,MR993775}.
Here are some important Gray code results in the geometric realm:
Hernando, Hurtado and Noy \cite{MR1939072} proved the existence of a Hamilton cycle in the flip graph of non-crossing perfect matchings on a set of $2m$ points in convex position for every even $m\geq 4$.
Aichholzer et al.\ \cite{MR2346418} described Hamilton cycles in the flip graphs of plane graphs on a general point set, for plane and connected graphs and for plane spanning trees on a general point set.
Huemer et al.\ \cite{MR2510231} constructed Hamilton cycles in the flip graphs of non-crossing partitions of a point set in convex position, and for the dissections of a convex polygon by a fixed number of non-crossing diagonals.

As mentioned before, a Hamilton cycle in a flip graph corresponds to a cyclic listing of all objects in some combinatorial class, such that each object is encountered exactly once, by performing a single flip in each step.
In this work we consider the \emph{dual} problem: we are interested in a cyclic enumeration of some of the combinatorial objects, such that each flip operation is encountered exactly once.
For instance, in the flip graph of triangulations $G_n^\tT$, we ask for the existence of a cycle with the property that each inner edge of the triangulation appears (and disappears) exactly once.
An example of such a cycle is shown in Figure~\ref{fig:triang}.
This idea can be formalized as follows.
Consider two triangulations $T$ and $T'$ that differ in flipping the diagonal of a convex quadrilateral, i.e., $T'$ is obtained from $T$ by removing the diagonal $e$ and inserting the other diagonal $f$.
We view the edge between $T$ and $T'$ in the flip graph $G_n^\tT$ as two arcs in opposite directions, where the arc from $T$ to $T'$ receives the label $f$, and the arc from $T'$ to $T$ receives the label $e$, so the label corresponds to the edge of the triangulation that enters in this flip; see the right hand side of Figure~\ref{fig:triang}.
Interpreting the labels as colors, we are thus interested in a directed cycle in the flip graph in which each color appears exactly once, and we refer to such a cycle as a \emph{rainbow cycle}.
More generally, for any integer $r\geq 1$, an \emph{$r$-rainbow cycle} in $G_n^\tT$ is a cycle in which each edge of the triangulation appears (and disappears) exactly $r$ times.
Note that a rainbow cycle does not need to visit all vertices of the flip graph.
Clearly, this notion of rainbow cycles extends in a natural way to all the other flip graphs discussed before; see Figure~\ref{fig:flip}.

\begin{figure}[ht]
\centering
\begin{subfigure}[t]{.45\textwidth}
\centering
 \includegraphics{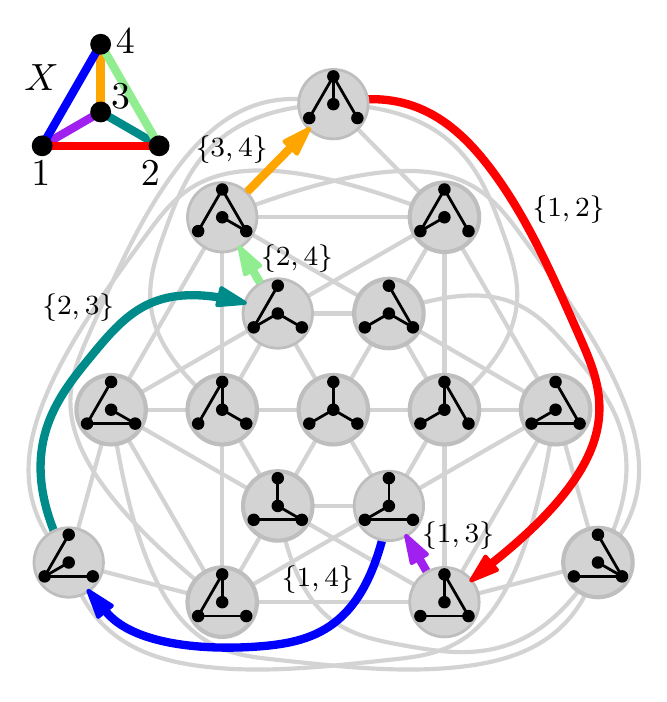}
 \subcaption{Flip graph of plane spanning trees~$G_X^\tS$.}
\end{subfigure}\qquad
\begin{subfigure}[t]{.45\textwidth}
\centering
 \includegraphics{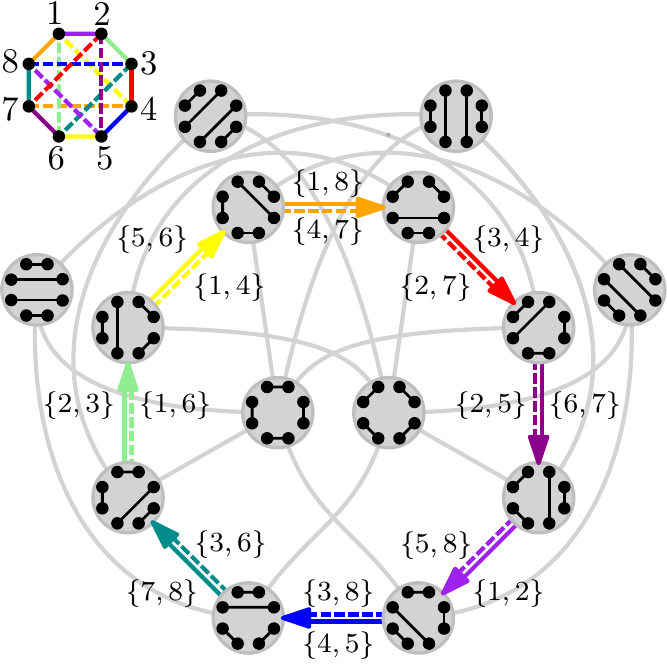}
 \subcaption{Flip graph of non-crossing perfect matchings~$G_4^\tM$.}
\end{subfigure}
\begin{subfigure}[t]{.45\textwidth}
\centering
 \includegraphics{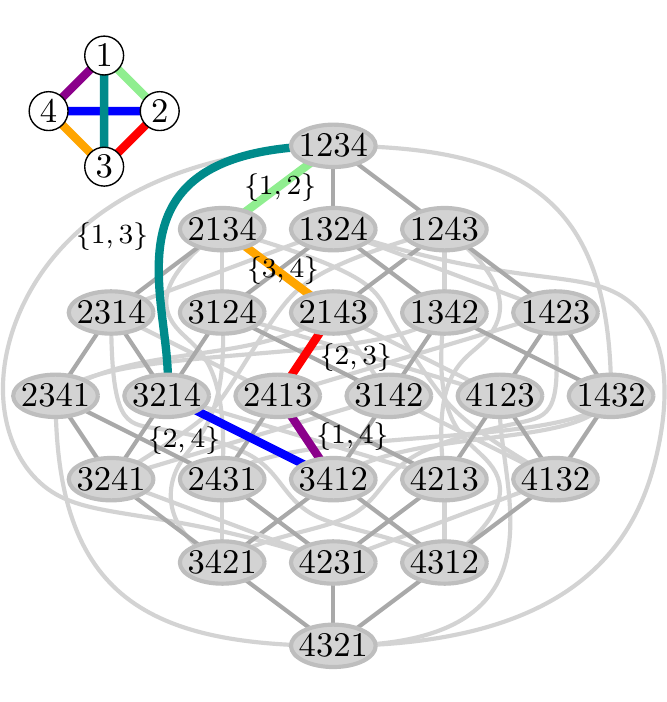}
 \caption{Flip graph of permutations~$G_4^\tP$.}
 \label{fig:permFlip}
\end{subfigure}\qquad
\begin{subfigure}[t]{.45\textwidth}
\centering
 \includegraphics{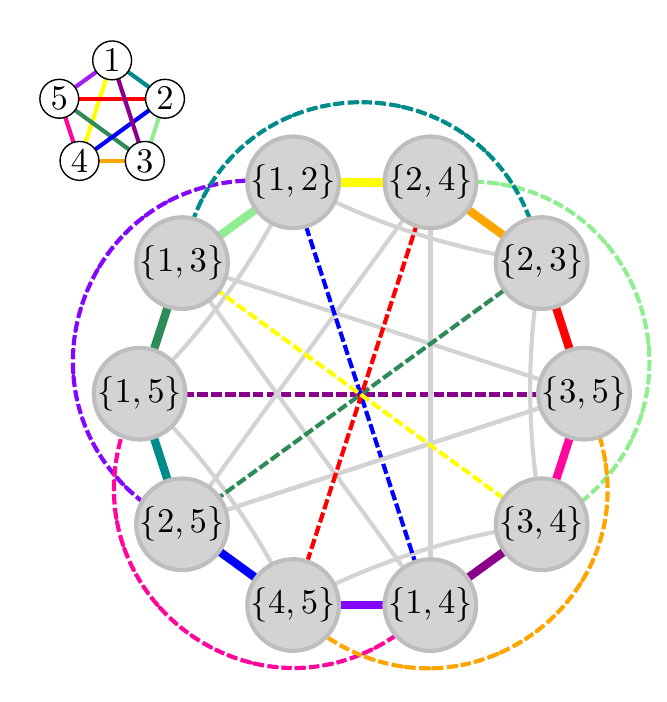}
 \caption{Flip graph  of subsets~$G_{5,2}^\tC$.}
\end{subfigure}
\caption{Examples of flip graphs with 1-rainbow cycles.
In (d), two edge-disjoint rainbow Hamilton cycles in $G_{5,2}^\tC$ are highlighted, one with bold edges and one with dashed edges.}
\label{fig:flip}
\end{figure}

\subsection{Our results}

In this work we initiate the investigation of rainbow cycles in flip graphs for five popular classes of combinatorial objects.
We consider three geometric classes: triangulations of a convex polygon, plane spanning trees on point sets in general position, and non-crossing perfect matchings on point sets in convex position.
In addition, we consider two classes without geometric information: permutations of the set $[n]:=\{1,2,\dots,n\}$, and $k$-element subsets of $[n]$.
We proceed to present our results in these five settings in the order they were just mentioned.
For the reader's convenience, all results are summarized in Table~\ref{tab:results}.

\begin{table}[ht]
\def\arraystretch{1.2}
\setlength{\tabcolsep}{3pt}
\centering
\caption{Overview of results.}
\label{tab:results}
\begin{tabular}{lllll G R l}
\hline
\multirow{12}{*}{\rotatebox[origin=c]{90}{\textsc{geometric}}}
  & \multicolumn{3}{c}{flip graph} &  \multicolumn{3}{c}{existence of $r$-rainbow cycle} & \\ 
  & & vertices & arcs/edges & $r$ & yes & no & \\ \hline
  & $G_n^\tT$ & \multirow[t]{2}{4cm}{triangulations of convex $n$-gon} & edge flip & $1$ & $n\geq 4$  & & Thm.~\ref{thm:triang} \\ 
  & & & & $2$&$n\geq 7$& & \\[10pt]
  & $G_X^\tS$ & \multirow[t]{2}{4cm}{plane spanning trees on point set $X$ in general position} & edge flip & $1,\hspace{-1.5pt}...,|X|-2$ & $|X|\geq 3$ & & Thm.~\ref{thm:trees} \\
  & & & & & & & \\ [10pt]
  & $G_m^\tM$ & \multirow[t]{3}{4cm}{non-crossing perfect matchings on $2m$ points in convex position} & \multirow[t]{2}{2cm}{two edge flip} & $1$ & $m\in\{2,4\}$ &  & Thm.~\ref{thm:match} \\
  & & & & & & \multirow[t]{-2}{2.4cm}{odd $m$, \\ $m\in\{6,8,10\}$} & \\
  & & & & $2$ & $m\in\{6,8\}$ & & \\
  \hdashline
\multirow{7}{*}{\rotatebox[origin=c]{90}{\textsc{abstract}}}
  & $G_n^\tP$ & \multirow[t]{1}{3.2cm}{permutations of $[n]$} & transposition & 1 & $\floor{n/2}$ even & $\floor{n/2}$ odd & Thm.~\ref{thm:perm} \\ [5pt]
  & $G_{n,k}^\tC$ & \multirow[t]{2}{3cm}{$k$-subsets of $[n]$, $2\leq k\leq \lfloor n/2\rfloor$} & \multirow[t]{2}{2cm}{element exchange} &1&  & even $n$ & Thm.~\ref{thm:comb} \\
  & & & & & \multirow[t]{-2}{2cm}{odd $n$ and $k<n/3$} & & \\
  & & \multirow[t]{2}{3cm}{$2$-subsets of $[n]$ for odd $n$} & & 1 & & & \\
  & & & & & & & \\
  & & & & & \multirow[t]{-3}{2.cm}{two edge-disjoint 1-rainbow Ham.\ cycles}& & \\ [5pt] \hline
\end{tabular}
\end{table}

Our first result is that the flip graph of triangulations $G_n^\tT$ defined in the introduction has a 1-rainbow cycle for $n\geq 4$ and a 2-rainbow cycle for $n\geq 7$ (Theorem~\ref{thm:triang} in Section~\ref{sec:triang}).

Next, we consider the flip graph $G_X^\tS$ of plane spanning trees on a point set $X$ in general position; see Figure~\ref{fig:flip}~(a).
We prove that $G_X^\tS$ has an $r$-rainbow cycle for any point set $X$ with at least three points for any $r=1,2,\dots,|X|-2$ (Theorem~\ref{thm:trees} in Section~\ref{sec:trees}).

We then consider the flip graph $G_m^\tM$ of non-crossing perfect matchings on $2m$ points in convex position; see Figure~\ref{fig:flip}~(b).
We exhibit 1-rainbow cycles for $m=2$ and $m=4$ matching edges, and 2-rainbow cycles for $m=6$ and $m=8$.
We also argue that there is no 1-rainbow cycle for $m\in\{6,8,10\}$, and none for any odd $m$.
In fact, we believe that there are no 1-rainbow cycles in $G_m^\tM$ for any $m\geq 5$.
Our results for this setting are summarized in Theorem~\ref{thm:match} in Section~\ref{sec:match}.

Next, we consider the flip graph $G_n^\tP$ of permutations of $[n]$, where an edge connects any two permutations that differ in a transposition, i.e., in exchanging two entries at positions $i$ and~$j$; see Figure~\ref{fig:flip}~(c).
The edges of this graph are colored with the corresponding pairs $\{i,j\}$, and in a 1-rainbow cycle each of the $\binom{n}{2}$ possible pairs appears exactly once.
We prove that $G_n^\tP$ has a 1-rainbow cycle if $\lfloor n/2\rfloor$ is even, and no 1-rainbow cycle if $\lfloor n/2\rfloor$ is odd (Theorem~\ref{thm:perm} in Section~\ref{sec:perm}).

Finally, we consider the flip graph $G_{n,k}^\tC$ of $k$-element subsets of $[n]$, also known as $(n,k)$-combinations, where an edge connects any two subsets that differ in exchanging one element $i$ for another element~$j$, i.e., the symmetric difference of the subsets has cardinality two; see Figure~\ref{fig:flip}~(d).
The edges of this graph are colored with the corresponding pairs $\{i,j\}$, and in a 1-rainbow cycle each of the $\binom{n}{2}$ possible pairs appears exactly once.
As $G_{n,k}^\tC$ is isomorphic to $G_{n,n-k}^\tC$, including the edge-coloring, we assume without loss of generality that $2\leq k\leq \lfloor n/2\rfloor$.
We prove that $G_n^\tC$ has a 1-rainbow cycle for every odd $n$ and $k<n/3$, and we prove that it has no 1-rainbow cycle for any even $n$.
The case $k=2$ is of particular interest, as a 1-rainbow cycle in the flip graph $G_{n,2}^\tC$ is a Hamilton cycle (both the number of subsets and the number of exchanges equal $\binom{n}{2}$).
Moreover, we show that $G_{n,2}^\tC$ even has two edge-disjoint 1-rainbow Hamilton cycles (for odd $n$).
Our results in this setting are summarized in Theorem~\ref{thm:comb} in Section~\ref{sec:comb}.

We conclude in Section~\ref{sec:open} with some open problems.

\subsection{Related work}

Gray codes are named after Frank Gray, a physicist at Bell Labs, who in 1953 patented a simple scheme to generate all $2^n$ bitstrings of length $n$ by flipping a single bit in each step.
This classical inductive construction is now called the \emph{binary reflected Gray code}; see \cite{MR993775} or \cite{MR3444818}.
Since its invention, there has been continued interest in developing binary Gray codes that satisfy various additional constraints, cf.\ the survey by Savage~\cite{MR1491049}.
The existence of a binary Gray code with the property that the bitflip counts in each of the $n$ coordinates are balanced, i.e., they differ by at most 2, was first established by Tootill \cite{tootill:56} (see also \cite{MR1410880}).
When $n$ is a power of two, every bit appears (and disappears) exactly $1/2\cdot 2^n/n=:r$ many times.
This balanced Gray code therefore corresponds to an $r$-rainbow cycle in the corresponding flip graph.
In this light, our results are a first step towards balanced Gray codes for other combinatorial classes.
For 2-element subsets, we indeed construct perfectly balanced Gray codes.

% bitstrings
% \cite{MR1410880} Savage/Bhat: balanced Gray codes
% \cite{MR1329390} Savage/Winkler: monotone Gray codes
% \cite{MR2014514} Goddyn/Gvozdjak: Gray codes with run length constraints
% \cite{MR1377601} Bultena/Ruskey: which graphs can be induced by Gray codes
% \cite{MR1935751} Wilmer/Ernst: which graphs can be induced by Gray codes (trees with arbitrarily large diameter can)
% \cite{MR2427745} Suparta/van Zanten: Gray codes inducing complete graph
% \cite{MR2974271} Dimitrov et al: Gray codes inducing a hypercube

The Steinhaus-Johnson-Trotter algorithm \cite{MR0159764,DBLP:journals/cacm/Trotter62}, also known as `plain changes', is a method to generate all permutations of $[n]$ by adjacent transpositions $i\leftrightarrow i+1$.
More generally, it was shown in \cite{kompelmakher-liskovets} that all permutations of $[n]$ can be generated by any set of transpositions that form a spanning tree on the set of positions $[n]$.
This is even possible under the additional constraint that in every second step the same transposition is applied \cite{MR1201997}.

% permutations
% \cite{MR0159764,DBLP:journals/cacm/Trotter62} Johnson/Trotter algorithm to generate permutations by adjacent transpositions
% \cite{kompelmakher-liskovets} Kompel'makher/Liskovets: generate permutations with restricted transpositions
% \cite{MR1201997} Ruskey/Savage: generalization of Kompel'makher/Liskovets result

The generation of $(n,k)$-combinations subject to certain restrictions on admissible exchanges $i\leftrightarrow j$ has been studied widely.
Specifically, it was shown that all $(n,k)$-combinations can be generated with only allowing exchanges of the form $i\leftrightarrow i+1$ \cite{MR737262,MR821383,MR936104}, provided that $n$ is even and $k$ is odd, or $k\in \{0,1,n-1,n\}$.
The infamous \emph{middle levels conjecture} asserts that all $(2k,k)$-combinations can be generated with only exchanges of the form $1\leftrightarrow i$, and this conjecture has recently been proved in \cite{MR3483129,mlc-short:17}.

% combinations
% \cite{MR737262} Buck/Wiedemann: path transposition graph, existence results
% \cite{MR821383} Eades/Hickey/Read: path transposition graph, existence result
% \cite{MR936104} Ruskey: path transposition graph, efficient algorithm
% \cite{MR1246674} Enns: cyclic adjacency graph
% \cite{MR3483129} M\"utze: star adjacency graph

Rainbow cycles and paths have also been studied in graphs other than flip graphs.
A well-known conjecture in this context due to Andersen \cite{andersen:89} asserts that every properly edge-coloured complete graph on $n$ vertices has a rainbow path of length $n-2$, i.e., a path that has distinct colors along each of its edges.
Progress towards resolving this conjecture was recently made by Alon, Pokrovskiy and Sudakov \cite{alon-rainbow:16}, and Balogh and Molla \cite{balogh-rainbow:17}.

\section{Triangulations}
\label{sec:triang}

In this section we consider a convex $n$-gon on points labeled clockwise by $1,2,\dots,n$, and we denote by $\cT_n$ the set of all triangulations on these points.
The graph $G_n^\tT$ has $\cT_n$ as its vertex set, and an arc $(T,T')$ between any two triangulations $T$ and $T'$ that differ in exchanging the diagonal $e\in T$ of a convex quadrilateral formed by two triangles for the other diagonal $f\in T'$; see Figure~\ref{fig:triang}.
We refer to this operation as a \emph{flip}, and we denote it by $(e,f)$.
Furthermore, we label the arc $(T,T')$ with the edge $f$, so an arc is labelled with the edge that enters the triangulation in this flip.
The set of arc labels of $G_n^\tT$ is clearly $E_n:=\{\{i,j\}\mid j-i>1\}\setminus\{1,n\}$, and we think of these labels as colors.
An $r$-rainbow cycle in $G_n^\tT$ is a directed cycle along which every label from $E_n$ appears exactly $r$ times.
Clearly, the length of an $r$-rainbow cycle equals $r|E_n|=r(\binom{n}{2}-n)$.
For comparison, the number of vertices of $G_n^\tT$ is the $(n-2)$-th Catalan number $\frac{1}{n-1}\binom{2n-4}{n-2}$.
Given an $r$-rainbow cycle, the cycle obtained by reversing the orientation of all arcs is also an $r$-rainbow cycle, as every edge that appears $r$ times also disappears $r$ times.
Here is an interesting interpretation of an $r$-rainbow cycle using the language of polytopes:
The secondary polytope of the triangulations of a convex $n$-gon, called the \emph{associahedron}, has the graph $G_n^\tT$ as its skeleton, and the facets of this polytope are the triangulations with a fixed edge.
Consequently, an $r$-rainbow cycle enters (and leaves) each facet of the associahedron exactly $r$ times.

The following theorem summarizes the results of this section.

\begin{theorem}
\label{thm:triang}
The flip graph of triangulations $G_n^\tT$ has the following properties:
\begin{enumerate}[label=(\roman*)]
\item If $n\geq 4$, then $G_n^\tT$ has a 1-rainbow cycle.
\item If $n\geq 7$, then $G_n^\tT$ has a 2-rainbow cycle.
\end{enumerate}
\end{theorem}

\begin{proof}
Let $S_i$ be the star triangulation with respect to the point~$i$, i.e., the triangulation where the point~$i$ has degree $n-1$.
To transform $S_1$ into $S_2$ we can use the flip sequence
\begin{equation}
\label{eq:Fn1}
F_{1,n}:=\big((\{1,3\},\{2,4\}),(\{1,4\},\{2,5\}), (\{1,5\},\{2,6\}),\dots,(\{1,n-1\},\{2,n\})\big).
\end{equation}
For any $i=1,2,\dots,n$, let $F_{i,n}$ denote the flip sequence obtained from $F_{1,n}$ by adding $i-1$ to all points on the right-hand side of \eqref{eq:Fn1}.
Here and throughout this proof addition is to be understood modulo $n$ with $\{1,2,\dots,n\}$ as representatives for the residue classes.
Note that $F_{i,n}$ transforms $S_i$ into $S_{i+1}$ for any $i\in[n]$, and all the edges from $E_n$ that are incident with the point~$i+1$ appear exactly once during that flip sequence.
Note also that $F_{i,n}$ has length $n-3$.

We begin proving (ii).
The concatenation $(F_{1,n},F_{2,n},\dots,F_{n,n})$ is a flip sequence which applied to $S_1$ leads back to $S_1$.
Along the corresponding cycle $C$ in $G_n^\tT$, every edge from $E_n$ appears exactly twice.
Specifically, every edge $\{i,j\}\in E_n$ appears in the flip sequences $F_{i-1,n}$ and $F_{j-1,n}$.
It remains to show that $C$ is indeed a cycle, i.e., every triangulation appears at most once.
For this observe that when applying $F_{i,n}$ to $S_i$, then for every $j=1,2,\dots,n-4$, in the $j$-th triangulation we encounter after $S_i$, the point $i$ is incident with exactly $n-3-j$ diagonals, the point $i+1$ is incident with exactly $j$ diagonals, while all other points are incident with at most two diagonals.
We call these triangulations \emph{bi-centered} with the two centers $i$ and $i+1$.
For $n\geq 8$, we have $\max_{1\leq j\leq n-4}\{n-3-j,j\}\geq 3$ and therefore we can determine at least one center $k$.
The other center is either the point $k-1$ or the point $k+1$.
Since only one of these two points is incident to some diagonal, we can identify it as the other center.
Hence for any bi-centered triangulation encountered along $C$, we can uniquely reconstruct in which flip sequence $F_{i,n}$ it occurs.
For $n=7$ it can be verified directly that $C$ is a 2-rainbow cycle. 

\begin{figure}
\centering
\includegraphics{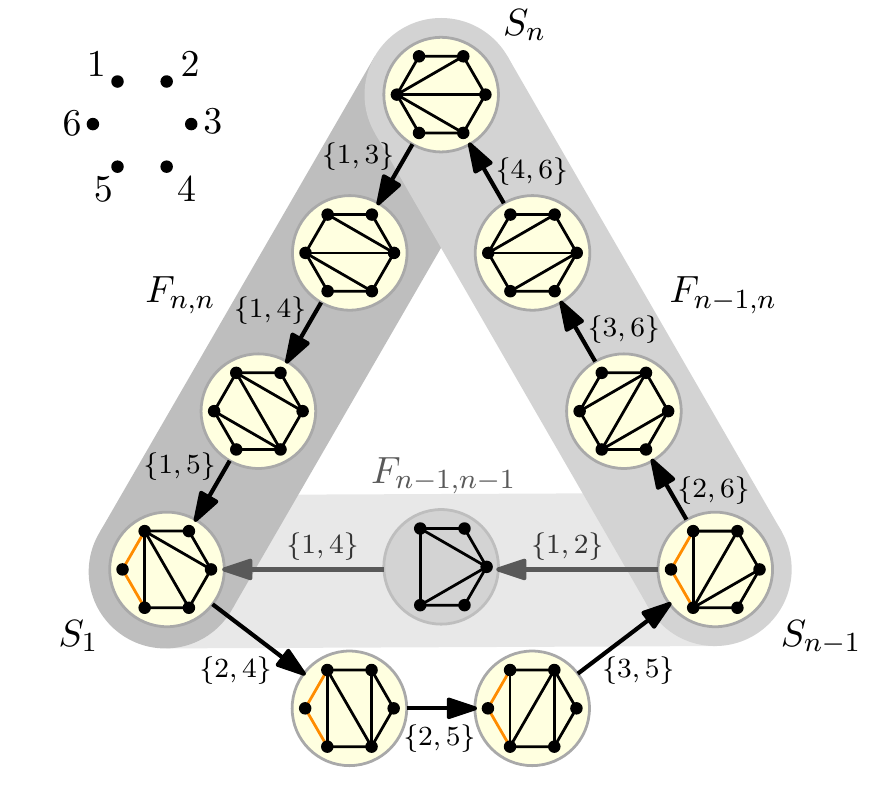} 
\caption{Illustration of the proof of Theorem~\ref{thm:triang}~(i).
The figure shows the inductive construction of a rainbow cycle for triangulations on 6~points from a rainbow cycle for 5~points.}
\label{fig:triang1}
\end{figure}

It remains to prove (i).
For $n\geq 4$, we obtain a 1-rainbow cycle in $G_n^\tT$ by applying the flip sequence $X_n:=(F_{3,4},F_{4,5},F_{5,6},\dots,F_{n-2,n-1},F_{n-1,n},F_{n,n})$ to the triangulation $S_1$. 
Note that $X_n$ differs from $X_{n-1}$ by replacing the terminal subsequence $F_{n-1,n-1}$ by $F_{n-1,n}$ and $F_{n,n}$; see Figure~\ref{fig:triang1}.
By induction, this yields a sequence of length $\big(\binom{n-1}{2}-(n-1)\big)-(n-4)+2(n-3)=\binom{n}{2}-n$.
The fact that $X_n$ produces a rainbow cycle follows by induction, by observing that applying $X_{n-1}$ to $S_1$ in $G_n^\tT$ yields a cycle along which every edge from $E_{n-1}$ appears exactly once.
Moreover, along this cycle the point $n$ is not incident with any diagonals.
The modifications described before to construct $X_n$ from $X_{n-1}$ shorten this cycle in $G_n^\tT$ and extend it by a detour through triangulations where the point $n$ is incident with at least one diagonal, yielding a cycle along which every edge from the following set appears exactly once:
\begin{equation*}
\begin{split}
E_{n-1} \;&\setminus\; \big\{\{1,3\},\{1,4\},\dots,\{1,n-2\}\big\} \\
{}&\cup\; \big\{\{n,2\},\{n,3\},\dots,\{n,n-2\}\big\} \\
{}&\cup\; \big\{\{1,3\},\{1,4\},\dots,\{1,n-2\},\{1,n-1\}\big\} \\
{}&= E_{n-1} \;\cup\; \big\{\{n,2\},\{n,3\},\dots,\{n,n-2\}\big\} \;\cup\;\{1,n-1\} \;=\; E_n.
\end{split}
\end{equation*}

This shows that applying $X_n$ to $S_1$ yields a 1-rainbow cycle in $G_n^\tT$.
\end{proof}

\section{Spanning trees}
\label{sec:trees}
In this section we consider plane spanning trees on a set $X$ of $n$ points in general position, i.e., no three points are collinear.
We use $\cS_X$ to denote the set of all plane spanning trees on~$X$.
The graph $G_X^\tS$ has $\cS_X$ as its vertex set, and an arc $(T,T')$ between any two spanning trees $T$ and $T'$ that differ in replacing an edge $e\in T$ by another edge $f\in T'$; see Figure~\ref{fig:flip}~(a).
We refer to this operation as a \emph{flip}, and we denote it by $(e,f)$.
Furthermore, we label the arc $(T, T')$ with the edge $f$, so an arc is labeled with the edge that enters the tree in this flip.
Note that the entering edge $f$ alone does not determine the flip uniquely (unlike for triangulations).
Clearly, none of the two edges $e$ and $f$ can cross any of the edges in $T \cap T'$, but they may cross each other.
The set of arc labels of $G_X^\tS$ is clearly $E_X:=\binom{X}{2}$, and we think of these labels as colors.
An $r$-rainbow cycle in $G_X^\tS$ is a directed cycle along which every label from $E_X$ appears exactly $r$ times, so it has length $r\binom{n}{2}$.

The following theorem summarizes the results of this setting.

\begin{theorem}
\label{thm:trees}
The flip graph of plane spanning trees $G_X^\tS$ has the following properties:
\begin{enumerate}[label=(\roman*)]
\item For any point set $X$ with $|X|\geq 3$ in general position, $G_X^\tS$ has a 1-rainbow cycle.
\item For any point set $X$ with $|X|\geq 4$ in general position and any $r=2,3,\dots,m$, where $m:=|X|-1$ if $|X|$ is odd and $m:=|X|-2$ if $|X|$ is even, $G_X^\tS$ has an $r$-rainbow cycle.
\end{enumerate}
\end{theorem}

\subsection{Proof of Theorem~\ref{thm:trees}~(i)}

We label the $n$ points of $X$ with integers $1, 2, \dots, n$ as follows; see Figure~\ref{fig:tree-labels}~(a).
We first label an arbitrary point on the convex hull of $X$ as point $1$, and we then label the points from $2$ to $n$ in counter-clockwise order around $1$ such that $\{1, 2\}$ and $\{1, n\}$ are edges on the convex hull of $X$.

\begin{figure}[htb]
\centering
\begin{tabular}{cc}
\includegraphics{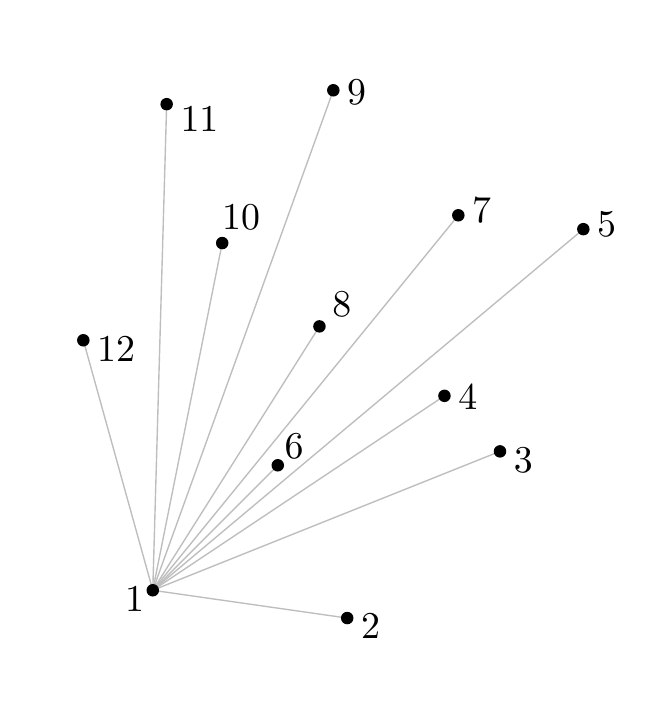}
&
\includegraphics{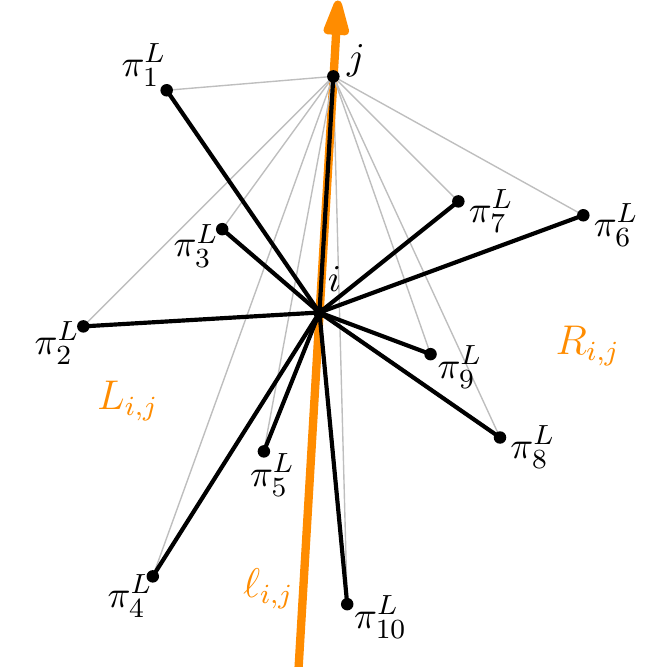}
\\
(a) & (b)
\end{tabular}
\caption{(a) Ordering of points $1,2,\dots,n$.
(b) The point ordering $\pi^L$ for the path from $S_i$ to $S_j$.}
\label{fig:tree-labels}
\end{figure}

Given a graph $G$ that has an edge $e$ but that does not have an edge $f$, we write $G - e$ for the graph obtained from $G$ by removing $e$, and we write $G+f$ for the graph obtained from $G$ by adding~$f$.
Furthermore, for any subset $Y \subseteq X$ and point $i \in Y$ we write $S_i(Y)$ for the tree on $Y$ that forms a star with center vertex $i$.
We write $S_i$ for the star $S_i(X)$; see Figure~\ref{fig:tree-labels}~(b).
For two distinct points $i,j\in[n]$, the directed line from $i$ to $j$ is denoted $\ell_{i,j}$.
The direction allows us to distinguish the left and right half-plane.
Let $L_{i,j}$ be the points of $X$ strictly on the left and $R_{i,j}$ the points of $X$ strictly on the right of the line $\ell_{i,j}$.

We will define two specific flip sequences $F^L_{i,j}$ and $F^R_{i,j}$ that transform the star $S_i$ into the star $S_j$; see Figure~\ref{fig:tree-labels}~(b).
Let $\tau^L$ be the sequence of all points $k$ in $L_{i,j}$ ordered by decreasing clockwise angles $(i,j,k)$.
Similarly, let $\tau^R$ be the sequence of all points $k$ in $R_{i,j}$ ordered by decreasing counter-clockwise angles $(i,j,k)$.
Let $\pi^L :=(\tau^L,\tau^R)$ and $\pi^R:=(\tau^R,\tau^L)$ be the concatenations of these two sequences.

The flip sequence $F^L_{i,j}$ is defined as
\begin{equation*}
F^L_{i,j}:=(f_1,f_2,\ldots,f_{n-1}) \quad \text{were} \quad
f_k:=\begin{cases}
      \big(\{i,j\},\{j,\pi^L_1\}\big) & \text{if } k=1, \\
      \big(\{i,\pi^L_{k-1}\},\{j,\pi^L_k\}\big) & \text{if } 2\leq k\leq n-2, \\
      \big(\{i,\pi^L_{n-2}\},\{j,i\}\big) & \text{if } k=n-1.
     \end{cases}
\end{equation*}
The flip sequence $F^R_{i,j}$ is defined analogously, by using $\pi^R$ instead of $\pi^L$.

Note that in both flip sequences, every edge from $E_X$ that contains the point $j$ appears exactly once.
Furthermore, if $\{i, j\}$ is an edge of the convex hull of $X$, then either $L_{i,j}$ or $R_{i,j}$ is empty and therefore $F^L_{i,j} = F^R_{i,j}$.
Otherwise, these flip sequences differ, as the first flip of $F^L_{i,j}$ adds an edge on the left of $\ell_{i,j}$, while the first flip of $F^R_{i,j}$ adds an edge on the right of $\ell_{i,j}$.

Clearly, each of the flip sequences $F^L_{i,j}$ and $F^R_{i,j}$ yields a path from $S_i$ to $S_j$ in the graph $G_X^\tS$, i.e., every flip adds an edge which is not in the tree and removes one which is in the tree, and the trees along the path are distinct plane spanning trees.
We denote the paths from $S_i$ to $S_j$ in the graph $G_X^\tS$ obtained from the flip sequences $F^L_{i,j}$ and $F^R_{i,j}$ by $P^L_{i,j}$ and $P^R_{i,j}$, respectively.
We refer to the trees along these paths other than $S_i$ and $S_j$ as \emph{intermediate trees}.
Note that there are $n-2$ intermediate trees along each of the paths $P^L_{i,j}$ and $P^R_{i,j}$.

\begin{proof}[Proof of Theorem~\ref{thm:trees}~(i)]
We prove the following stronger statement by induction on $n$:
For any point set $X=[n]$ of $n \geq 3$ points in general position, there is a $1$-rainbow cycle in $G_X^\tS$ that contains the subpath $P^L_{n,1}$.
Recall that the edge $\{1,n\}$ lies on the convex hull of~$X$ and therefore $P^L_{n,1}=P^R_{n,1}$.

\begin{figure}
\centering
\includegraphics{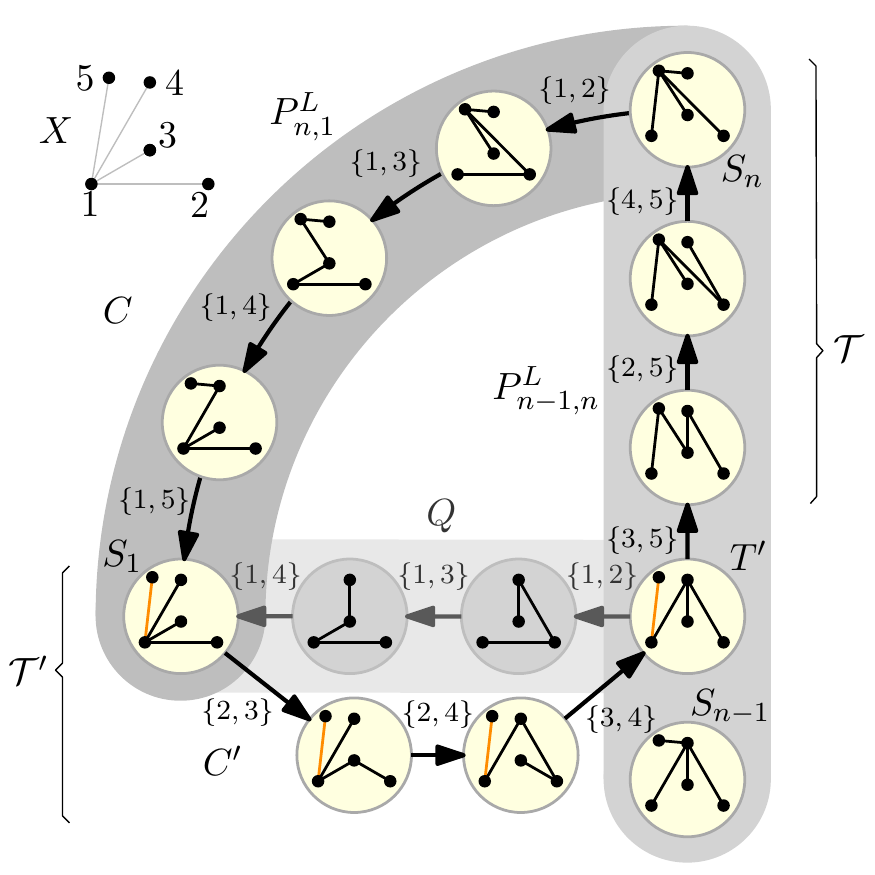} 
\caption{Illustration of the proof of Theorem~\ref{thm:trees}~(i).
In the induction step, the rainbow cycle from Figure~\ref{fig:flip}~(a) on the point set $X\setminus\{5\}=[4]$ is extended to a rainbow cycle on the point set $X=[5]$.}
\label{fig:trees1}
\end{figure}

To settle the base case $n=3$ we take the cycle $(S_1, S_3, S_2)$.

The following inductive construction is illustrated in Figure~\ref{fig:trees1}.
For the induction step let $C'$ be the rainbow cycle for the point set $X':=X\setminus\{n\}=[n-1]$, and let $Q$ be the subpath $P^L_{n-1,1}=P^R_{n-1,1}$ of $C'$.
We obtain the desired rainbow cycle $C$ for the point set $X=[n]$ as follows:
We remove all intermediate trees on the path $Q$ from $C'$, and we add the edge $\{1,n\}$ to all remaining trees, so that all these trees are spanning trees on the point set $X$.
The resulting path in $G_X^\tS$ starts at $S_1(X)=S_1(X')+\{1,n\}$ and ends at $T':=S_{n-1}(X')+\{1,n\}$.
Note that $T'$ is the first intermediate tree on the path $P^L_{n-1,n}$, so we continue the cycle $C$ from $T'$ along this path until we reach the star $S_n(X)$, and from there we complete the cycle $C$ along the path $P^L_{n,1}=P^R_{n,1}$ back to the star $S_1(X)$.
By construction, $C$ contains the required subpath.

We now argue that $C$ does not visit any spanning tree twice.
Let $\cT'$ denote the set of trees on $C'$ except the intermediate trees on $Q$, and let $\cT$ denote the set of trees on the paths $P^L_{n-1,n}$ and $P^L_{n,1}$ except the trees $S_{n-1}(X)$, $T'$ and $S_1(X)$.
Note that all trees in $\cT'$ contain the edge $\{1,n\}$ and the point $n$ has degree~1, whereas all trees in $\cT$ do not contain the edge $\{1,n\}$ or the point $n$ has degree at least~2.
It follows that $\cT'\cap\cT=\emptyset$.
The intermediate trees on the path $P^L_{n-1,n}$ do not contain the edge $\{1,2\}$, whereas the intermediate trees on the path $P^L_{n,1}$ do contain this edge, so no two intermediate trees of these paths are the same.
We conclude that $C$ does not visit any tree twice.

By construction, along the cycle $C$ every edge from the following set appears exactly once:
\begin{equation*}
\begin{split}
 E_{X'} \;&\setminus\; \big\{\{1,2\},\{1,3\},\dots,\{1,n-1\}\big\} \\
 {}&\cup\; \big\{\{n,2\},\{n,3\},\dots,\{n,n-1\}\big\} \\
 {}&\cup\; \big\{\{1,2\},\{1,3\},\dots,\{1,n-1\},\{1,n\}\big\} \\
 {}&= E_{X'} \;\cup\; \big\{\{n,2\},\{n,3\},\dots,\{n,n-1\}\big\} \;\cup\;\{1,n\} \;=\; E_{X}.
\end{split}
\end{equation*}
This shows that $C$ is a 1-rainbow cycle in $G_X^\tS$.
\end{proof}

\subsection{Proof of Theorem~\ref{thm:trees}~(ii)}

The next lemma explicitly describes all intermediate trees along the paths $P^L_{i,j}$ and $P^R_{i,j}$.
It is an immediate consequence of the definition of the flip sequences given in the previous section.

A \emph{caterpillar} is a tree that has the property that when removing all leafs, the remaining graph is a path.
We refer to any path that is obtained from a caterpillar by removing a number of leafs as a \emph{central path} of the caterpillar.
Note that all other vertices not on a central path are leafs of degree~1.
Consequently, specifying the degree sequence of the vertices on a central path of a caterpillar describes the caterpillar uniquely.
Note that a caterpillar may have several different central paths, e.g.\ the caterpillar with central path $(a,b,c,d)$ and degree sequence $(1,4,4,1)$ can also be described via the central path $(b,c)$ and the degree sequence $(4,4)$.

\begin{lemma}
\label{lem:intermediate}
Let $n\geq 3$.
For any two points $i,j\in[n]$ and any $1\leq t\leq n-2$, the $t$-th intermediate tree on the path $P^L_{i,j}$ from $S_i$ to $S_j$ is a caterpillar and $(i,\pi^L_t,j)$ is a central path of the caterpillar with degree sequence $(n-1-t,2,t)$.
An analogous statement holds for all intermediate trees on the path $P^R_{i,j}$.
\end{lemma}

The next lemma asserts that the intermediate trees along any two paths obtained from our flip sequences are all distinct.

\begin{lemma}
\label{lem:Pij-disjoint}
Let $n \geq 6$.
For any two paths $P \in \{P^L_{i,j}, P^R_{i,j}\}$ and $P' \in \{P^L_{i',j'}, P^R_{i',j'}\}$ with $\{i, j\} \not = \{i', j'\}$, all intermediate trees on $P$ and $P'$ are distinct.
Equivalently, $P$ and $P'$ are internally vertex-disjoint paths in the graph $G_X^\tS$.
\end{lemma}

\begin{proof}
Let $\pi \in \{\pi^L_{i,j}, \pi^R_{i,j}\}$ be the ordering of points corresponding to the path $P$.
We argue that for any intermediate tree on $P$, it is possible to uniqely reconstruct the points $i$ and $j$ which form the center of the stars $S_i$ and $S_j$ that are the end vertices of the path $P$.
Let $T$ be the $t$-th intermediate tree on $P$, $1\leq t\leq n-2$.
By Lemma~\ref{lem:intermediate}, the point $\pi_t$ has degree~2 in $T$, whereas all other points $\pi_a$, $a\neq t$, are leafs of degree~1 in $T$, and $\deg(i)+\deg(j)=n-1\geq 5$.
If $T$ has two points of degree at least 3, those must be $i$ and $j$ and we are done.
Otherwise $T$ has only one point with degree at least 3, we assume w.l.o.g.\ that it is point~$i$.
Then we can determine point $j$ as the unique point with distance exactly $2$ from $i$ in $T$.
Specifically, by Lemma~\ref{lem:intermediate}, $i$ and $j$ are connected via $\pi_t$ in $T$, and all other points are neighbors of either $i$ or $j$, so they have  distance 1 or 3 from $i$.
This completes the proof.
\end{proof}

The proof of Theorem~\ref{thm:trees}~(ii) is split into three parts.
We first construct $r$-rainbow cycles for even values of $r$ (Proposition~\ref{prop:trees2r}), then for odd values of $r$ (Proposition~\ref{prop:trees2r-1}), and we finally settle some remaining small cases (Proposition~\ref{prop:trees-small}).

\begin{proposition}
\label{prop:trees2r}
Let $X$ be a set of $n \geq 6$ points in general position.
For any $r=1,2,\dots,\lfloor (n-1)/2\rfloor$, there is a $2r$-rainbow cycle in $G_X^\tS$.
\end{proposition}

In the proof we will use a decomposition of the complete graph on $n$~vertices into $\lfloor (n-1)/2\rfloor$ Hamilton cycles (and a perfect matching for even~$n$, which will not be used in the proof, though).
Such a decomposition exists by Walecki's theorem; see \cite{MR2394738}.

\begin{proof}
We apply Walecki's theorem to obtain a set $\cH$ of $r$ edge-disjoint Hamilton cycles in $K_n$, the complete graph on $n$ vertices.
We now consider the complete graph $K_X$ on the point set $X=[n]$, and we map the Hamilton cycles in $\cH$ onto $K_X$ such that one Hamilton cycle $H_0\in \cH$ visits all points on the convex hull of $X$ successively.
We orient $H_0$ so that it visits the points on the convex hull in counter-clockwise order.
If $X$ is not in convex position, then there is a unique edge $e$ on the convex hull of $X$ that is not covered by $H_0$.
If this edge $e$ is contained in some other Hamilton cycle $H_1 \in \cH$, then we orient $H_1$ so that the edge $e$ is also traversed in counter-clockwise direction on the convex hull.
Each of the remaining Hamilton cycles in $\cH$ is oriented arbitrarily in one of the two directions.
The union of these $r$ oriented Hamilton cycles in $K_X$ yields a directed graph with in-degree and out-degree equal to $r$ at each point.
We fix an arbitrary Eulerian cycle $\cE$ in this graph.
Note that $\cE$ visits each point exactly $r$ times, and it traverses all edges on the convex hull of $X$ in counter-clockwise direction.

We define a directed closed tour $C'$ in $G_X^\tS$, which possibly contains certain trees multiple times, by considering every triple of points $(i,j,k)$ along $\cE$.
If $i\in R_{j,k}$, then we say that the triple \emph{$(i,j,k)$ takes a right-turn}, and if $i\in L_{j,k}$, then we say that the triple \emph{$(i,j,k)$ takes a left-turn}.
If $(i,j,k)$ takes a right-turn, then we add the path $P^L_{j,k}$ to $C'$, and if $(i,j,k)$ takes a left-turn, then we add the path $P^R_{j,k}$ to $C'$.
Figure~\ref{fig:detour} shows an example of two concatenated paths.
From Lemma~\ref{lem:Pij-disjoint} we know that the stars $S_i$, $i\in[n]$, are the only trees that are visited multiple times by $C'$.
Specifically, each $S_i$ is visited exactly $r$ times by $C'$.
Furthermore, every edge $\{i,j\}\in E_X$ appears on exactly $2r$ arcs of $C'$, once on every path to $S_i$ and once on every path to $S_j$.
It follows that the tour $C'$ has the $2r$-rainbow property.

\begin{figure}
\begin{center}
\includegraphics{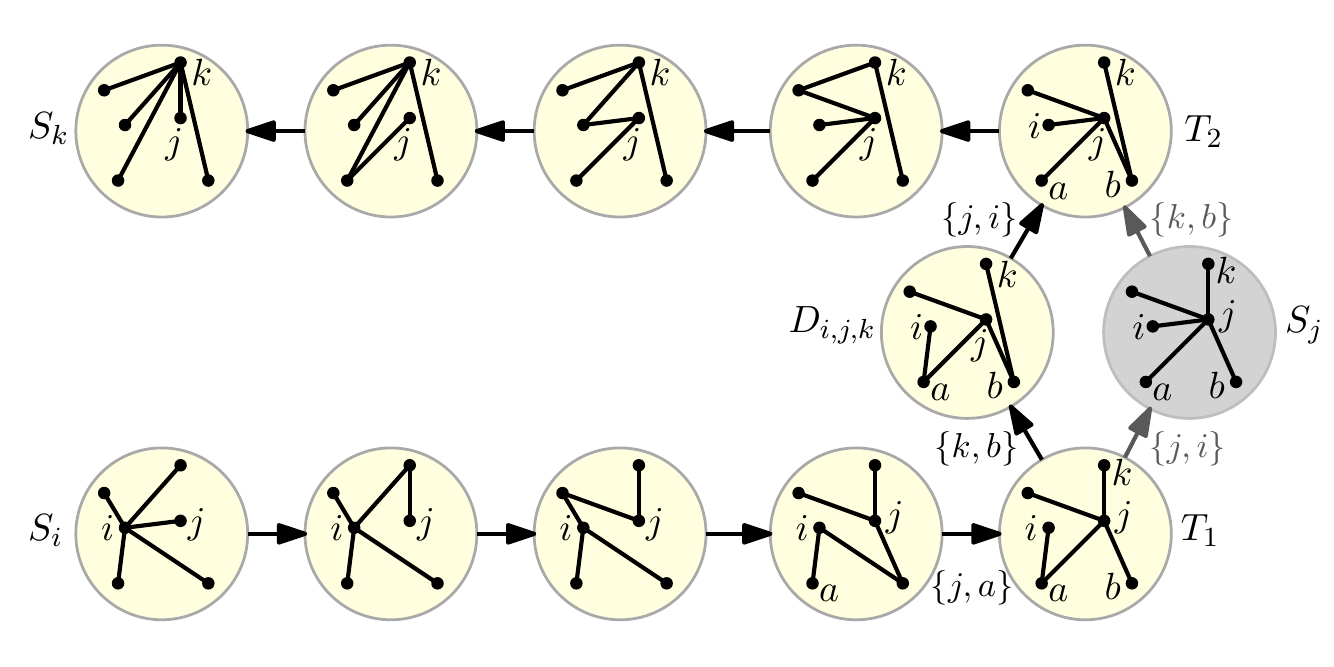} 
\caption{The path $P_{i,j}=P^L_{i,j}$ from $S_i$ to $S_j$ (bottom) and the path $P_{j,k}=P^R_{j,k}$ from $S_j$ to $S_k$ (top) and a detour around $S_j$ via the detour tree $D_{i,j,k}$, in the case where $a\notin \{b,k\}$.}
\label{fig:detour}
\end{center}
\end{figure}

To modify $C'$ into a $2r$-rainbow cycle $C$ we will use detours around most of the stars; see Figures~\ref{fig:detour} and \ref{fig:detour2}.
In the following we define the \emph{detour tree $D_{i,j,k}$} for those triples of points $(i,j,k)$ on $\cE$ where $\{j,k\}$ is not an edge of the convex hull of $X$.
These detour trees are then used to replace all the stars of $C'$ except a single occurence of each $S_j$ where $j$ is on the convex hull.
These replacements yield $C$.

\begin{figure}[htb]
\begin{center}
\includegraphics{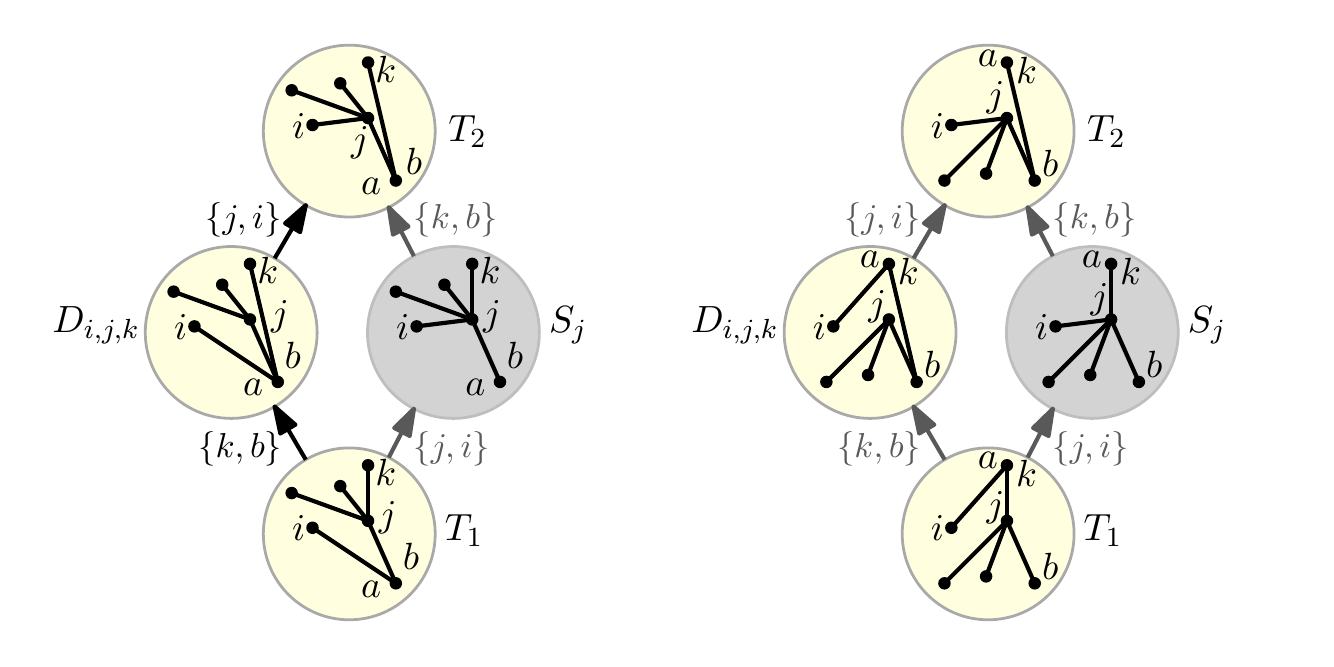} 
\caption{The detour trees $D_{i,j,k}$ for the cases $a = b$ (left) and $a = k$ (right).} 
\label{fig:detour2}
\end{center}
\end{figure}

To define $D_{i,j,k}$ consider the path $P_{i,j}\in \{P^L_{i,j}, P^R_{i,j}\}$ from $S_i$ to $S_j$ and the path $P_{j,k} \in \{P^L_{j,k}, P^R_{j,k}\}$ from $S_j$ to $S_k$ in $C'$.
Let $T_1$ be the predecessor of $S_j$ on $P_{i,j}$, and let $T_2$ be the successor of $S_j$ on $P_{j,k}$.
There are points $a,b\in[n]$ such that $S_j = T_1 - \{i,a\} + \{j,i\}$ and $T_2 = S_j - \{j,k\} +  \{k,b\}$.
The detour tree is defined as
\begin{equation}
\label{eq:Dijk}
  D_{i,j,k} = T_1 - \{j, k\} + \{k, b\} = T_2  - \{j,i\} + \{i,a\}.
\end{equation}
We denote the two relevant flips in this definition by $f_1:=(\{i,a\},\{j,i\})$ and $f_2:=(\{j,k\},\{k,b\})$.
As $\{j,k\}$ does not lie on the convex hull of $X$, both half-planes $L_{j,k}$ and $R_{j,k}$ contain points from $X$.
By our choice of the path from $S_j$ to $S_k$ based on the orientation of the triple $(i,j,k)$, the points~$b$ and $i$ lie in opposite half-planes.
It follows that $b \neq i$ and that $i$, $j$, $k$, and $b$ are four different points.
Note also that the point $a$ is different from $i$ and $j$, but it may happen that $a=b$ or that $a=k$; see Figure~\ref{fig:detour2}.

\begin{claim}
The flip $f_2=(\{j, k\},\{k, b\})$ is a legal flip applicable to $T_1$ and the resulting graph $D_{i,j,k}$ is a plane spanning tree.
\end{claim}

To establish this claim we check the following four properties:

\begin{itemize}

\item \textit{$\{j, k\}$ is an edge of $T_1$.}
$T_1$ contains all edges incident with $j$ except the edge $\{i,j\}$ by Lemma~\ref{lem:intermediate}.
As $i \neq k$, the tree $T_1$ contains in particular the edge $\{j,k\}$.

\item \textit{$\{k, b\}$ is not an edge of $T_1$.}
The only edge in $T_1$ that is not incident to $j$ is the edge $\{i, a\}$.
As $i \neq k$ and $i \neq b$, it follows that $\{k, b\}$ is not an edge of $T_1$.

\item \textit{$D_{i,j,k}$ is a spanning tree.}
We only need to check that $D_{i,j,k}$ is connected, which is true as the two endpoints $j$ and $k$ of the removed edge $\{j,k\}$ are connected via the path $(j,b,k)$ in $D_{i,j,k}$.

\item \textit{$D_{i,j,k}$ is plane.}
The added edge $\{k, b\}$ does not cross any edges incident with $j$, otherwise it would not be the first added edge along $P_{j,k}$.
It remains to show that $\{k, b\}$ does not cross the edge $\{i, a\}$ either.
If $a=b$ or $a=k$, then there is no crossing and we are done, so we can assume that $a$ is different from the four other points $i,j,k,b$.
In the following we assume that the triple $(i,j,k)$ takes a left-turn, the other case follows by symmetry.
Recall from before that $b\in R_{j,k}$ and $i\in L_{j,k}$.
Hence, the counter-clockwise order around $j$ is $(b,k,i,\bar{k})$, where $\bar{k}$ denotes the antipodal direction of $k$.
From the definition of $P_{i,j}^L$ and $P_{i,j}^R$ it follows that $a$ is either the immediate predecessor or the immediate successor of $i$ in the cyclic order of the points around $j$.
If the counter-clockwise order around $j$ is $(b,k,a,i,\bar{k})$ or $(b,k,i,a,\bar{k})$ then there is a line through $j$ which separates $\{k, b\}$ and $\{i, a\}$.
In the remaining case $a$ and $b$ belong to $R_{i,j}\cap R_{j,k}$.
By the definition of $T_2$ we also have $a\in R_{k,b}$ in this case, so the line through $k$ and $a$ separates $\{k, b\}$ and $\{i, a\}$.

\end{itemize}

This completes the proof of the claim.

Note that executing the flip operations $f_1$ and $f_2$ in the opposite order just changes the order in which the edges $\{j,i\}$ and $\{k,b\}$ appear, so the resulting tour $C$ still has the $2r$-rainbow property.
To show that $C$ is a cycle, it remains to show that each detour tree $D_{i,j,k}$ is used only once in $C$.
For this we first give an explicit description of the intermediate trees, which is an immediate consequence of the previous definitions.

\begin{lemma}
\label{lem:detour}
Let $n\geq 6$.
For any triple $(i,j,k)\in \cE$, the intermediate tree $D_{i,j,k}$ defined in \eqref{eq:Dijk} has the following properties:
\begin{enumerate}[label=(\arabic*)]
\item If $a\notin\{b,k\}$, then $D_{i,j,k}$ is a caterpillar and $(a,j,b)$ is a central path of the caterpillar with degree sequence $(2,n-3,2)$, where the unique leafs in distance~2 from $j$ are the points $i$ and $k$; see Figure~\ref{fig:detour}.
\item If $a=b$, then $D_{i,j,k}$ is a caterpillar and $(j,a)=(j,b)$ is a central path of the caterpillar with degree sequence $(n-3,3)$, where the unique leafs in distance~2 from $j$ are the points $i$ and $k$; see the left hand side of Figure~\ref{fig:detour2}. \\
Moreover, if $(i,j,k)$ takes a left-turn, then the point $a=b$ lies within the sector $R_{i,j}\cap R_{j,k}$ and there is no point in the sector $R_{i,j}\cap R_{j,a}$.
% or in the sector $R_{j,k}\cap R_{a,k}$.
If $(i,j,k)$ takes a right-turn, on the other hand, then an analogous statement holds with right and left half-planes interchanged.
\item If $a=k$, then $D_{i,j,k}$ is a caterpillar and $(j,b,a)=(j,b,k)$ is a central path of the caterpillar with degree sequence $(n-3,2,2)$, where $k$ is the unique point in distance~2 from $j$, and $i$ is the unique leaf incident with $k$; see the right hand side of Figure~\ref{fig:detour2}.
\end{enumerate}
\end{lemma}

Since at most one of the triples $(i,j,k)$ or $(k,j,i)$ appears along $\cE$, this lemma allows us to reconstruct the triple $(i,j,k)$ from any given detour tree $D_{i,j,k}$.
Only in the case $n=6$ when the degree sequence of the central path of the caterpillar is $(3,3)$ (case~(2) of the lemma), there is an ambiguity which of the two points of the central path is $j$.
This ambiguity can be resolved by using the additional property mentioned in (2) involving half-planes.
It can be easily checked that if this condition holds for $(i,j,k)$ as in the lemma, then it does not hold for $(x,a,y)=(x,b,y)$, where $x$ and $y$ are the unique leafs incident with $j$.
Consequently, all detour trees $D_{i,j,k}$ included in $C$ are distinct.

It remains to argue that each detour tree $D_{i,j,k}$ is distinct from all intermediate trees along any path $P_{i',j'}\in\{P^L_{i',j'},P^R_{i',j'}\}$ from which $C$ is built.
By Lemma~\ref{lem:intermediate}, the $t$-th intermediate tree $T$ on $P_{i',j'}$ has a central path with degree sequence $(n-1-t,2,t)$ for each $1\leq t\leq n-2$.
In cases~(1) and (2) of Lemma~\ref{lem:detour}, comparing the degree sequences shows that $D_{i,j,k}$ must be different from $T$.
In case~(3) there can only be a conflict if $t=2$, as then we have $(n-1-t,2,t)=(n-3,2,2)$.
Matching the degree sequences in this case, we must have $(i',j')=(j,a)=(j,k)$.
However, in the second intermediate spanning ($t=2$) on $P_{i',j'}=P_{j,k}$ the point $b$ has degree 1, whereas the point $b$ has degree 2 in $D_{i,j,k}$.
It follows that all detour trees $D_{i,j,k}$ are distinct from all intermediate trees.

This shows that $C$ is indeed a $2r$-rainbow cycle in $G_X^\tS$, completing the proof of Proposition~\ref{prop:trees2r}.
\end{proof}

To construct a $(2r - 1)$-rainbow cycle in $G_X^\tS$ we slightly modify the construction from the previous proof.
Specifically, we remove one Hamilton cycle from $K_X$ before building the Eulerian cycle, which decreases the rainbow count by 2 for each edge from $E_X$.
Instead we substitute the $1$-rainbow cycle constructed in the proof of Theorem~\ref{thm:trees}~(i), yielding a $(2r-1)$-rainbow cycle.

\begin{proposition}
\label{prop:trees2r-1}
Let $X$ be a set of $n \geq 6$ points in general position.
For any $r=2,3,\dots,\lfloor (n-1)/2\rfloor$, there is a $(2r-1)$-rainbow cycle in $G_X^\tS$.
\end{proposition}

\begin{proof}
The construction starts as in the proof of Proposition~\ref{prop:trees2r}.
We consider a set $\cH$ of $r$ edge-disjoint Hamilton cycles in the complete graph $K_n$, and we map the Hamilton cycles in $\cH$ onto the complete graph $K_X$ on the point set $X=[n]$ such that one Hamilton cycle $H_0$ visits all points on the convex hull of $X$ successively.
In addition, we perform the mapping so that $H_0$ contains the edges $\{n-1, n\}$, $\{n, 1\}$, and $\{1, 2\}$ (the first of these three edges does not necessarily lie on the convex hull, but the latter two edges do by our ordering of the points).
We orient each of the Hamilton cycles as in the previous proof, so that all edges on the convex hull are oriented counter-clockwise.
We now remove the Hamilton cycle $H_0$, yielding a set of directed Hamilton cycles $\cH^-:=\cH\setminus \{H_0\}$, and build a Eulerian cycle $\cE$ in this graph, which has in-degree and out-degree equal to $r-1\geq 1$ at each point.
We fix one triple of the form $(i',1,k')$ in $\cE$.
From the Eulerian cycle $\cE$ we build a directed closed tour $C'$ in $G_X^\tS$ as in the previous proof, where for the special triple $(i',1,k')$ we use the path $P^R_{1,k'}$, regardless of the orientation of this triple.
Along the edges of this tour, every edge from $E_X$ appears exactly $2r-2$ times.
We then modify $C'$ into a cycle $C$ by considering every triple $(i,j,k)\in\cE$ except the special triple $(i',1,k')$ and by replacing $S_j$ by the corresponding detour tree $D_{i,j,k}$ if the edge $\{j,k\}$ does not lie on the convex hull of $X$ (as before).
For the special triple $(i',1,k')$, we do not replace $S_1$.
As in $\cH^-$ no directed edge starting at 1 proceeds along an edge of the convex hull of $X$ ($H_0$ uses the edge $\{1,2\}$, and the edge $\{n,1\}$ is oriented towards 1), all occurences of $S_1$ in $C'$ except for the single occurence corresponding to the special triple $(i',1,k')$ are replaced in $C$.
Now let $C^1$ be the 1-rainbow cycle constructed as in the proof of Theorem~\ref{thm:trees}~(i) starting with $S_1$.
We replace the unique occurence of $S_1$ in $C$ by $C^1$ followed by $D_{n,1,k'}$, yielding a tour $C^+$; see Figure~\ref{fig:odd-rainbow}.
We claim that $C^+$ is a $(2r-1)$-rainbow cycle in $G_X^\tS$.
Clearly, $C^+$ has the $(2r-1)$-rainbow property, so we only need to show that $C^+$ is a cycle, i.e., no tree is visited more than once.
By the arguments given in the proof of Proposition~\ref{prop:trees2r}, it suffices to show that all trees on $C^1$ are distinct from the ones in $C\setminus\{S_1\}$.

We divide the cycle $C^1$ into segments according to its inductive construction in the proof of Theorem~\ref{thm:trees}~(i).
Specifically, for $c=3,4,\dots,n$ we define
\begin{equation*}
T_c:=S_c([c])+\{1,c+1\}+\{1,c+2\}+\dots+\{1,n\};
\end{equation*}
see Figure~\ref{fig:odd-rainbow}.
These are the spanning trees along which the cycle in the inductive construction described in Theorem~\ref{thm:trees}~(i) is split in each step.
We follow $C^1$ starting at $S_1$ and argue that each of the trees along the cycle is distinct from $C\setminus\{S_1\}$.
The arguments are divided into cases (1)--(9), which are illustrated in Figure~\ref{fig:odd-rainbow}.

\begin{figure}
\centering
\includegraphics{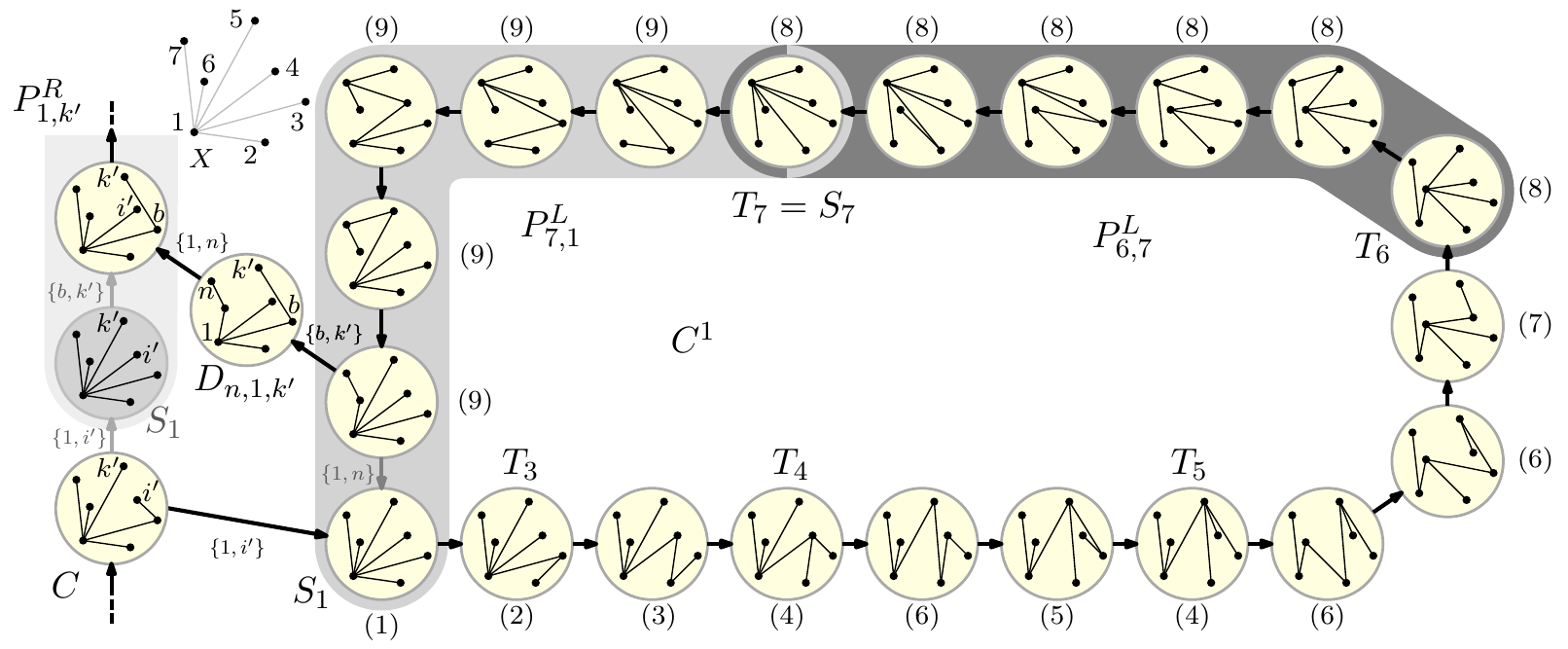}
\caption{Modification of the $(2r-2)$-rainbow cycle $C$ by including the $1$-rainbow cycle $C^1$ to obtain a $(2r-1)$-rainbow cycle $C^+$ for a point set with $n=7$ points.}
\label{fig:odd-rainbow}
\end{figure}

\begin{enumerate}
\item
The first tree on $C^1$ is $S_1$.
As argued before, $C$ contains only a single occurence of $S_1$ which was replaced, so $S_1$ is unique in $C^+$.

\item
The tree $T_3$ is a caterpillar with central path $(1,3,2)$ with degree sequence $(n-2,2,1)$.
This tree is different from any detour tree by Lemma~\ref{lem:detour}.
By Lemma~\ref{lem:intermediate}, it can only be equal to the first intermediate tree on the path $P^L_{1,2}=P^R_{1,2}$.
However, as the edge $\{1,2\}$ is contained in $H_0$ and not in $\cE$, the path $P^L_{1,2}$ is not part of $C'$, and therefore $T_3$ does not occur in $C$.

\item
The successor of $T_3$ on $C^1$ is a caterpillar with central path $(1,4,2)$ with degree sequence $(n-3,2,2)$.
By Lemma~\ref{lem:detour} this tree can only be equal to a detour tree $D_{i,j,k}$ as captured in case~(3) of the lemma, which would imply $j=1$ and $k=2$.
However, no triple of the form $(i,j,k)=(i,1,2)$ is contained in $\cE$.
Moreover, by Lemma~\ref{lem:intermediate} this tree can only be equal to the second intermediate tree on the path $P_{1,2}$, which is not part of $C$ as argued before.

\item
The tree $T_c$, $4 \leq c \leq n-2$, is a caterpillar with a central path $(1,c)$ with degree sequence $(n-c+1,c-1)$.
It follows from Lemmas~\ref{lem:intermediate} and \ref{lem:detour} that for $4<c<n-2$ this tree is different from any tree on $C$.
For $c=4$, this tree can only be equal to a detour tree $D_{i,j,k}$ as captured in case~(2) of Lemma~\ref{lem:detour} where $a=b$, which would imply $(j,a)=(1,4)$ and $\{i,k\}=\{2,3\}$, or $(j,a)=(c,1)$ and $\{i,k\}=\{5,6\}$ for $n=6$.
In both cases we obtain a contradiction to the property that the point $a=b$ lies within the sector $R_{i,j}\cap R_{j,k}$ if the triple $(i,j,k)$ takes a left-turn, or in the sector $L_{i,j}\cap L_{j,k}$ if the triple takes a right-turn.
For $c=n-2$, a similar reasoning shows that $T_c$ is different from any trees on $C$.

\item
The predecessor of $T_c$, $5 \leq c \leq n-2$, is a caterpillar with a central path of the form $(1,c,x)$ with degree sequence $(n-c+1,c-2,2)$, which by Lemmas~\ref{lem:intermediate} and \ref{lem:detour} is distinct from any tree on $C$.

\item
For $4 \leq c \leq n-2$, all trees strictly between $T_i$ and the predecessor of $T_{c+1}$ on $C^1$, have diameter at least~5.
Specifically, the $k$-th successor of $T_i$ contains a path of the form $(n, 1, c+1, x, c, y)$.
By Lemmas~\ref{lem:intermediate} and \ref{lem:detour} all trees on $C$ have diameter at most~4.

\item
The predecessor of $T_{n-1}$ on $C^1$ has a central path of the form $(1,n-1,x)$ with degree sequence $(2,n-3,2)$, where $n$ is the unique leaf incident with $1$ and $n-2$ is the unique leaf incident with~$x$.
By Lemma~\ref{lem:intermediate} this tree can only be equal to a detour spanning $D_{i,j,k}$ as captured in case~(1) of the lemma, which would imply $(i,j,k)=(n-2,n-1,n)$ or $(i,j,k)=(n,n-1,n-2)$.
However, as the edge $\{n-1,n\}$ is not in $\cE$, this triple does not occur in $\cE$ either, so this detour tree is not in $C$.

\item
The path from $T_{n-1}$ and $T_n=S_n(X)$ is by definition the path $P^L_{n-1, n}$ with the first tree $S_{n-1}(X\setminus\{n\})$ removed, and as the edge $\{n-1, n\}$ is not in $\cE$, all trees on this path are distinct from the ones on $C$.
The spaning tree $T_n=S_n$ is not in $C$, as the only edge on the convex hull of $X$ incident with $n$, if such an edge is present in $\cH^-$ at all, is oriented towards $n$, so the next edge along $\cE$ is not a convex hull edge.

\item
The path from $T_n = S_n$ to the predecessor of $S_1$ is exactly $P^L_{n, 1}$ with the last tree $S_1$ removed, and as the edge $\{n, 1\}$ is not in $\cE$, all trees on this path are distinct from the ones on $C$.
The detour tree $D_{n,1,k'}$ is distinct from all trees on $C$ as argued in the proof of Proposition~\ref{prop:trees2r}.
\end{enumerate}

This completes the proof of Proposition~\ref{prop:trees2r-1}.
\end{proof}

\begin{proposition}
\label{prop:trees-small}
For any point set $X$ with $n=4$ points in general position, there is a $2$-rainbow cycle in $G_X^\tS$.
For any point set $X$ with $n=5$ points in general position and any $r\in\{2,3,4\}$, there is an $r$-rainbow cycle in $G_X^\tS$.
\end{proposition}

The rainbow cycles for proving Proposition~\ref{prop:trees-small} can be constructed explicitly using slight variants of the methods described in the preceding proofs.
This proof is deferred to the appendix.

\begin{proof}[Proof of Theorem~\ref{thm:trees}~(ii)]
Combine Proposition~\ref{prop:trees2r}, Proposition~\ref{prop:trees2r-1} and Proposition~\ref{prop:trees-small}. 
\end{proof}

\section{Matchings}
\label{sec:match}

In this section we consider a set of $n=2m$ points in convex position labeled clockwise by $1,2,\dots,n$. 
Without loss of generality we assume that the points are distributed equidistantly on a unit circle centered at the origin.
We use $\cM_m$ to denote the set of all non-crossing perfect matchings with $m$ edges on these points.
The graph $G_m^\tM$ has $\cM_m$ as its vertex set, and an arc $(M,M')$ between any two matchings $M$ and $M'$ that differ in exchanging two edges $e=\{a,b\}\in M$ and $f=\{c,d\}\in M$ for the edges $e'=\{a,c\}$ and $f'=\{b,d\}\in M'$; see Figure~\ref{fig:flip}~(b).
We refer to this operation as a \emph{flip}.
Furthermore, we label the arc $(M,M')$ with the edges $e'$ and $f'$, so an arc is labeled with the edges that enter the matching in this flip.
The set of arc labels of $G_m^\tM$ is $E_m:=\{\{i,j\}\mid i,j\in [n]\text{ and $j-i$ is odd}\}$.
In this definition, the difference $j-i$ must be odd so that an even number of points lies on either side of the edge $\{i,j\}$.
Every arc of $G_m^\tM$ carries two such labels, and we think of these labels as colors.
An $r$-rainbow cycle in $G_m^\tM$ is a directed cycle along which every label in $E_m$ appears exactly $r$ times.
As every arc is labeled with two edges, an $r$-rainbow cycle has length $r|E_m|/2=rm^2/2$.
The number of vertices of $G_m^\tM$ is the $m$-th Catalan number $\frac{1}{m + 1}\binom{2m}{m}$.

The following theorem summarizes the results of this section.

\begin{theorem}
\label{thm:match}
The flip graph of non-crossing perfect matchings $G_m^\tM$, $m\geq 2$, has the following properties:
\begin{enumerate}[label=(\roman*)]
\item If $m$ is odd, then $G_m^\tM$ has no 1-rainbow cycle.
\item If $m\in\{6,8,10\}$, then $G_m^\tM$ has no 1-rainbow cycle.
\item If $m\in\{2,4\}$, then $G_m^\tM$ has a 1-rainbow cycle, and if $m\in\{6,8\}$, then $G_m^\tM$ has a 2-rainbow cycle.
\end{enumerate}
\end{theorem}

\subsection{Proof of Theorem~\ref{thm:match}~(i)}

\begin{proof}
A 1-rainbow cycle must have length $m^2/2$.
For odd $m$, this number is not integral, so there can be no such cycle.
\end{proof}

\subsection{Proof of Theorem~\ref{thm:match}~(ii)}
\label{sec:centered}

In view of part (i) of Theorem~\ref{thm:match}, we assume for the rest of this section that the number of matching edges $m$ is even.

The following definitions are illustrated in Figure~\ref{fig:centered}.
The \emph{length} of a matching edge $e\in M$, denoted by $\ell(e)$, is the minimum number of other edges from $M$ that lie on either of its two sides.
Consequently, a matching edge on the convex hull has length 0, whereas the maximum possible length is $(m-2)/2$, so there are $m/2$ different edge lengths.
Note that $E_m$ contains exactly $n=2m$ edges of each length.

\begin{figure}[h]
\begin{center}
% draw the given number of points equidistantly on a circle
% and call them (1),...,(n)
\newcommand{\matching}[1]{
\foreach \i in {1,...,#1} {
  \node[node_black] (\i) at (-\i*360/#1-90+180/#1:1.5) {};
  % \draw (-\i*360/#1-90+180/#1:1.8) node {$\i$};
}
}

\begin{tikzpicture}[scale=0.8]

\matching{16}

\begin{pgfonlayer}{background}
\draw[draw=none, fill=black!25!white] (1.center) -- (4.center) -- (9.center) -- (16.center) -- cycle;
\end{pgfonlayer}

\path[line_solid] (-0.1,-0.1) -- (0.1,0.1);
\path[line_solid] (-0.1,0.1) -- (0.1,-0.1);

\path[line_solid] (1) to (4);
\path[line_solid] (2) to (3);
\path[line_solid] (5) to (6);
\path[line_solid] (7) to (8);
\path[line_solid] (9) to (16);
\path[line_solid] (10) to (13);
\path[line_solid] (11) to (12);
\path[line_solid] (14) to (15);

\draw (0,-1.2) node {0};
\draw (-0.7,-0.7) node {1};
\draw (-0.5,0.3) node {2};
\draw (0.05,0.5) node {3};
\draw (-5,1) node {$0+1+2+3=6$};
\draw (-3.06,.4) node {$=m-2$};
\end{tikzpicture}\hspace{2cm}
\begin{tikzpicture}[scale=0.8]

\matching{16}

\begin{pgfonlayer}{background}
\draw[draw=none, fill=black!25!white] (1.center) -- (4.center) -- (7.center) -- (8.center) -- cycle;
\end{pgfonlayer}

\path[line_solid] (-0.1,-0.1) -- (0.1,0.1);
\path[line_solid] (-0.1,0.1) -- (0.1,-0.1);

\path[line_solid] (1) to (4);
\path[line_solid] (2) to (3);
\path[line_solid] (5) to (6);
\path[line_solid] (7) to (8);
\path[line_solid] (9) to (16);
\path[line_solid] (10) to (13);
\path[line_solid] (11) to (12);
\path[line_solid] (14) to (15);

\draw (-0.7,-0.7) node {1};
\draw (-0.95,0.5) node {1};
\draw (-0.6,1) node {0};
\draw (-0.5,0) node {3};
\draw (3.5,1) node {$1+1+0+3=5$};
\draw (5.4,.4) node {$<m-2$};
\end{tikzpicture}
\caption{Examples of a centered 4-gon (left) and a non-centered 4-gon (right) for $m=8$ matching edges.
The numbers are the edge lengths.} 
\label{fig:centered}
\end{center}
\end{figure}
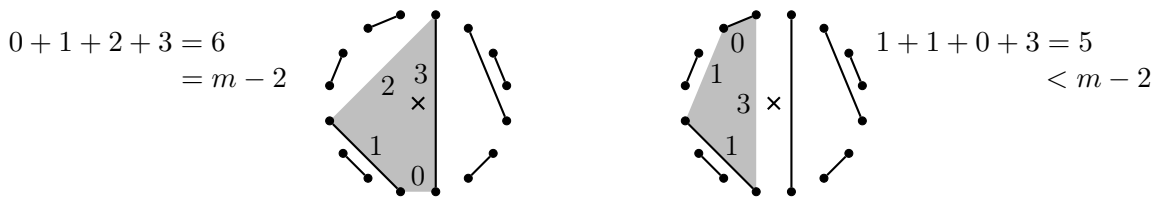

We call a convex quadrilateral formed by four edges from $E_m$ a \emph{centered 4-gon}, if the sum of the edge lengths of the quadrilateral is $m-2$.
Note that this is the maximum possible value.
We refer to a flip involving a centered 4-gon as a \emph{centered flip}.
Equivalently, a flip is centered if and only if the corresponding 4-gon contains the origin.
Note that all flips in the rainbow cycle in Figure~\ref{fig:flip}~(b) are centered flips.
This is in fact not a coincidence, as shown by the following lemma.

\begin{lemma}
\label{lem:centered}
All flips along an $r$-rainbow cycle in $G_m^\tM$ must be centered flips.
\end{lemma}

\begin{proof}
$E_m$ contains exactly $n=2m$ edges of each length $0,1,\dots,(m-2)/2$.
Along an $r$-rainbow cycle $C$, exactly $rn$ edges of each length appear and exactly $rn$ edges of each length disappear.
Consequently, the average length of all edges that appear or disappear along $C$ is $(m-2)/4$.
By definition, in a centered flip the average length of the four edges involved in the flip is exactly the same number $(m-2)/4$; whereas for a non-centered flip, it is strictly smaller.
Therefore, $C$ must not contain any non-centered flips.
\end{proof}

\begin{figure}[p]
\centering
\includegraphics{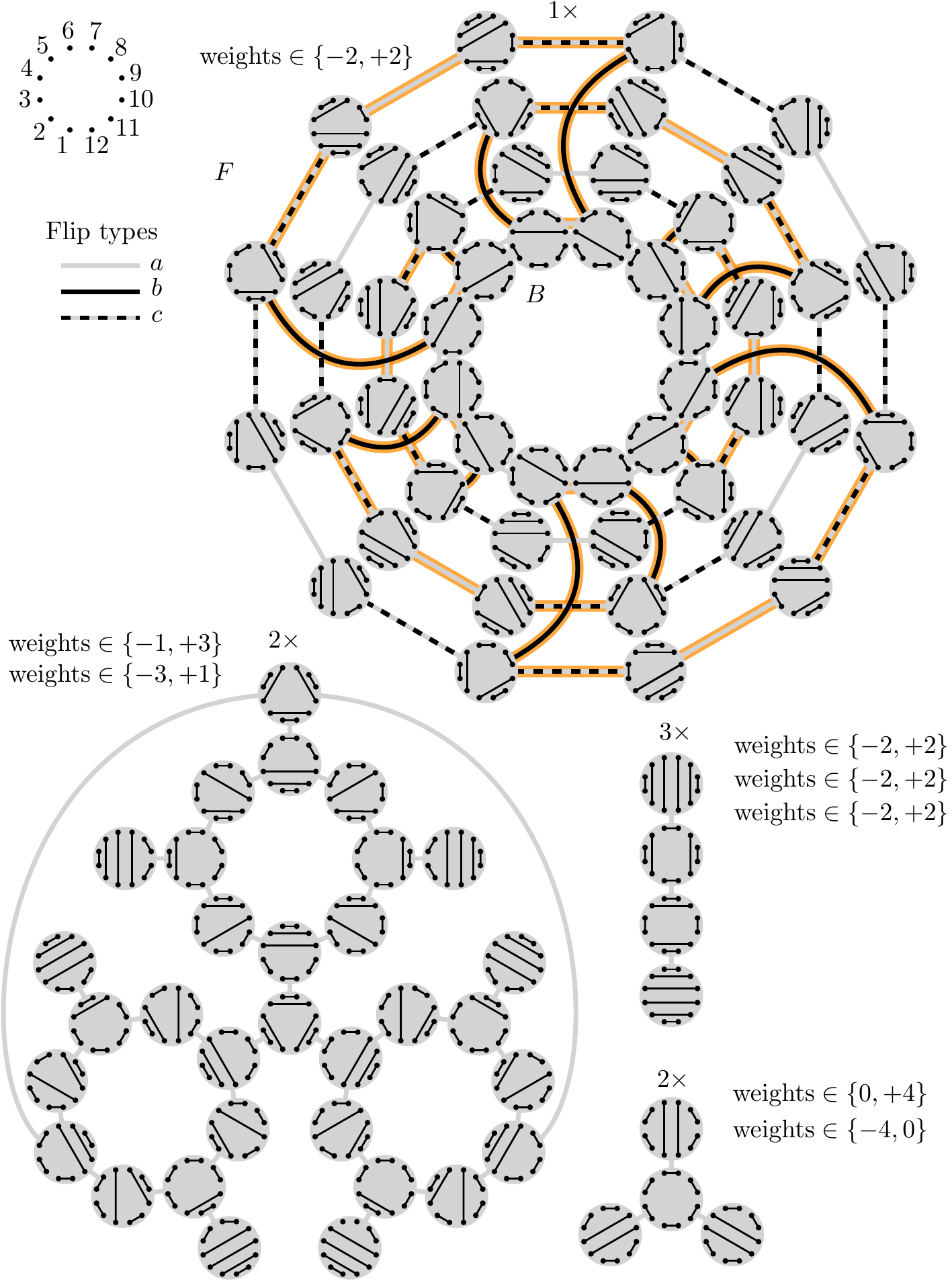}
\caption{Illustration of the graph $H_6\subseteq G_6^\tM$.
Some components of this graph are isomorphic to each other and differ only by rotation of the matchings by multiples of $\pi/6$. Only one representative for each component is shown, together with its multiplicity.
The total number of matchings is the 6th Catalan number $132$.
%$132=48+2\times 32+3\times 4+2\times 4$.
The 2-rainbow cycle constructed in the proof of Theorem~\ref{thm:match}~(iii) is highlighted in the component $F$.}
\label{fig:HM6}
\end{figure}

Lemma~\ref{lem:centered} suggests to restrict our search for rainbow cycles to the subgraph of $G_m^\tM$ obtained by considering only arcs that represent centered flips.
We denote this subgraph of $G_m^\tM$ on the same vertex set $\cM_m$ by $H_m$.
This graph is shown in Figure~\ref{fig:HM6} for $m=6$ matching edges.

% \begin{lemma}
% \label{lem:flip-twice}
% Let $m\geq 6$.
% No 1-rainbow cycle can use two flips determined by the exact same centered 4-gon.
% \end{lemma}

% \begin{proof}
% Suppose for the sake of contradiction that such a rainbow cycle $C$ with two flips using the exact same 4-gon $P$ exists.
% Let $\{e,f\}$ and $\{g,h\}$ be the two pairs of opposite edges of $P$.
% In one of the flips involving $P$ along $C$ the edges $\{e,f\}$ appear and $\{g,h\}$ disappear from the matching, and in the second flip along $C$ involving $P$, the edges $\{e,f\}$ disappear and $\{g,h\}$ appear.
% It follows that in every matching along the rainbow cycle, either both $e$ and $f$ are present, or both $g$ and $h$ are present.
% W.l.o.g.\ we assume that $\ell(e)\geq \ell(f)$ and $\ell(g)\geq \ell(h)$.
% As $\ell(e)+\ell(f)+\ell(g)+\ell(h)=m-2$ and $m\geq 6$, we obtain $\ell(e)\geq 1$ and $\ell(g)\geq 1$.
% Consequently, there is an edge from $E_m$ crossing $e$ and $g$ which never enters the matching along the cycle $C$, a contradiction to the rainbow property.
% \end{proof}

\begin{proof}[Proof of Theorem~\ref{thm:match}~(ii)]

The proof for the cases $m=8$ and $m=10$ is computer-based, and uses exhaustive search for a 1-rainbow cycle in each connected component of $H_m$.
However, we are unable to prove this by hand.
We proceed to show that there is no 1-rainbow cycle in $H_m$ for $m=6$.
Unlike all our other non-existence proofs, this one is not a simple parity argument, but involves structural considerations.

Suppose there exists a 1-rainbow cycle $C$ in $H_m$ for $m=6$.
Clearly, $C$ has length $m^2 / 2 = 18$.
The graph $H_6$ has eight connected components; see Figure~\ref{fig:HM6}.
Five of these components are trees, and do not contain any cycles, and three of them contain cycles.
From those three cyclic components, two are isomorphic to each other and differ only by rotation of the matchings by $\pi/6$, so only one of the two is shown at the bottom left of Figure~\ref{fig:HM6}.
The only cycles in these two components are of length 8 or 12, so they do not contain the desired rainbow cycle.
It therefore remains to show that the third cyclic component $F$ shown at the top right of Figure~\ref{fig:HM6} has no rainbow cycle.
This component has a cycle $B$ of length 12 containing all matchings with a single edge of length 2.
We refer to this cycle as the \emph{base cycle} of $F$; it is drawn in the center of $F$ in the figure.
Removing the base cycle $B$ from $F$ leaves three other cycles of length 12, that differ only by rotation of the matchings in them.
We call these three cycles \emph{satellite cycles}, and we refer to the edges between the base cycle and the satellite cycles as \emph{spokes}.
Each satellite cycle is attached with 4 spokes to the base cycle, and the spokes are spaced equidistantly along both cycles, and no two spokes share any vertices.

There are three different types of centered 4-gons involved in the flips in $F$.
Each type is characterized by the cyclic sequence of edge lengths of the 4-gon, where mirroring counts as the same 4-gon (this corresponds to reversing the sequence of edge lengths).
These types are shown in Figure~\ref{fig:4gon-types}, and they denoted by $a$, $b$, and $c$.

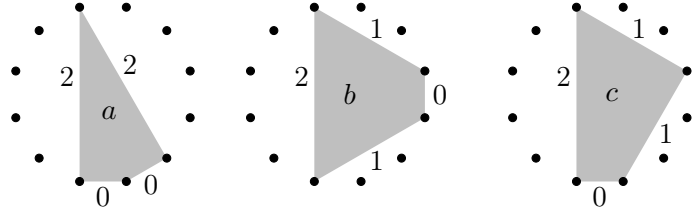
\begin{figure}[h]
\begin{center}
% draw the given number of points equidistantly on a circle
% and call them (1),...,(n)
\newcommand{\matching}[1]{
\foreach \i in {1,...,#1} {
  \node[node_black] (\i) at (-\i*360/#1-90+180/#1:1.5) {};
  %\draw (-\i*360/#1-90+180/#1:1.8) node {\scriptsize $\i$};
}
}

\begin{tikzpicture}[scale=0.8]

\matching{12}
    
\begin{pgfonlayer}{background}
\draw[draw=none, fill=black!25!white] (1.center) -- (6.center) -- (11.center) -- (12.center) -- cycle;
\end{pgfonlayer}

\draw (0.1,-0.3) node {$a$};

\draw (-0.6,0.3) node {2};
\draw (0.45,0.5) node {2};
\draw (0.8,-1.5) node {0};
\draw (0,-1.7) node {0};
\draw (0,-1.9) node {};  % dummy label for alignment

\end{tikzpicture} \hspace{3mm}
\begin{tikzpicture}[scale=0.8]

\matching{12}
    
\begin{pgfonlayer}{background}
\draw[draw=none, fill=black!25!white] (1.center) -- (6.center) -- (9.center) -- (10.center) -- cycle;
\end{pgfonlayer}

\draw (0.2,0) node {$b$};

\draw (-0.6,0.3) node {2};
\draw (0.65,1.1) node {1};
\draw (1.7,0) node {0};
\draw (0.65,-1.1) node {1};
\draw (0,-1.9) node {};  % dummy label for alignment

\end{tikzpicture} \hspace{3mm}
\begin{tikzpicture}[scale=0.8]

\matching{12}
    
\begin{pgfonlayer}{background}
\draw[draw=none, fill=black!25!white] (1.center) -- (6.center) -- (9.center) -- (12.center) -- cycle;
\end{pgfonlayer}

\draw (0.2,0) node {$c$};

\draw (-0.6,0.3) node {2};
\draw (0.65,1.1) node {1};
\draw (1.1,-0.65) node {1};
\draw (0,-1.7) node {0};
\draw (0,-1.9) node {};  % dummy label for alignment

\end{tikzpicture}
\vspace{-18pt}
\end{center}
\caption{The three different types of centered 4-gons involved in the flips in $F$.}
\label{fig:4gon-types}
\end{figure}

In Figure~\ref{fig:HM6}, the edges corresponding to each flip type $a$, $b$, and $c$ are drawn solid gray, solid black and dashed black, respectively.
Note that all flips along the base cycle are of type $a$, all flips along spokes are of type $b$, and along each satellite cycle, one flip of type $a$ alternates with two flips of type $c$.

% The following table shows the number of occurences of each edge length $\{0,1,2\}$ for each of the three flip types:
% \begin{center}
% \begin{tabular}{l|cc}
% Flip type & edge length sequence & freq.\ of lengths (0,1,2) \\ \hline
% $a$ & (2,2,0,0) & (2,0,2) \\
% $b$ & (2,1,0,1) & (1,2,1) \\
% $c$ & (2,1,1,0) & (1,2,1)
% \end{tabular}.
% \end{center}

Note that the set $E_m$ of 36 edges that appear along $C$ contains exactly 12 edges of each length from $\{0,1,2\}$.
Since a flip of type $a$ does not involve any edge of length 1, and flips $b$ and $c$ involve exactly two edges of length 1, we must perform in total 12 flips of types $b$ or $c$ along $C$ (as we must have $24=2\cdot 12$ appearances or disappearances of an edge of length 1), and hence exactly 6 flips of type~$a$.
Since the edges corresponding to flips of types $b$ and $c$ form stars in $F$, it follows that every sequence of three consecutive flips along $C$ must contain at least one flip of type $a$.
Combining these two observations shows that from the 18 flips along $C$, exactly every third must be of type $a$.
Therefore, $C$ can be partitioned into six paths of length three such that each path either has the form $\beta:=(b,a,b)$ or $\sigma:=(c,a,c)$.
A path of type $\beta$ moves from a satellite cycle along a spoke to the base cycle, then uses one edge of the base cycle, and then returns via a spoke to another satellite cycle.
A path of type $\sigma$ moves along three consecutive edges of a satellite cycle, between the end vertices of two spokes.
The rainbow cycle $C$ can therefore be described by a cyclic sequence of 6 symbols from $\{\beta,\sigma\}$.
In this sequence, any $\beta$ must be followed by $\sigma$, as we must not traverse the same spoke twice.
Moreover, at most three $\sigma$ symbols can appear consecutively, otherwise we would close a cycle of length 12 along a satellite cycle.
Without loss of generality, we may assume that the maximum length substring of consecutive $\sigma$ symbols comes first in this sequence.
This leaves the following possible patterns: $(\sigma,\beta,\sigma,\beta,\sigma,\beta)$, $(\sigma,\sigma,\beta,\sigma,\sigma,\beta)$ and $(\sigma,\sigma,\sigma,\beta,\sigma,\beta)$.
A straightforward case analysis shows that for any cycle $C$ following one of those patterns, one of the matching edges of length 2 appears and disappears twice, rather than only once.
Consequently, $F$ and $H_6$ contain no 1-rainbow cycle.
\end{proof}

\subsection{Proof of Theorem~\ref{thm:match}~(iii)}

\begin{proof}
There are two non-crossing matchings with $m=2$ edges, connected by two arcs in $G_m^\tM$ that form 1-rainbow cycle.
For $m=4$, a 1-rainbow cycle in $G_m^\tM$ is shown in Figure~\ref{fig:flip}~(b).

For $m=6$, a 2-rainbow cycle in $G_m^\tM$ can be constructed using the path $P$ of length 6 between matchings $M$ and $M'$ in $G_m^\tM$ depicted in Figure~\ref{fig:Ppath6}.

\begin{figure}[h]
 \begin{center}
  % draw the given number of points equidistantly on a circle
% and call them (1),...,(n)
\newcommand{\matching}[1]{
\foreach \i in {1,...,#1} {
  \node[node_black] (\i) at (-\i*360/#1-90+180/#1:1.5) {};
  %\draw (-\i*360/#1-90+180/#1:1.8) node {\scriptsize $\i$};
}
}

\begin{tikzpicture}[scale=0.5]

\matching{12}

\begin{pgfonlayer}{background}
\draw[draw=none, fill=black!25!white] (5.center) -- (6.center) -- (9.center) -- (12.center) -- cycle;
\end{pgfonlayer}

\path[line_solid] (1) to (4);
\path[line_solid] (2) to (3);
\path[line_solid] (5) to (6);
\path[line_solid] (7) to (8);
\path[line_solid] (9) to (12);
\path[line_solid] (10) to (11);

\draw (0,0) node {$M$};

\end{tikzpicture} \hspace{1mm}
  \begin{tikzpicture}[scale=0.5]

\matching{12}

\begin{pgfonlayer}{background}
\draw[draw=none, fill=black!25!white] (5.center) -- (10.center) -- (11.center) -- (12.center) -- cycle;
\end{pgfonlayer}

\path[line_solid] (1) to (4);
\path[line_solid] (2) to (3);
\path[line_solid] (5) to (12);
\path[line_solid] (6) to (9);
\path[line_solid] (7) to (8);
\path[line_solid] (10) to (11);

\end{tikzpicture} \hspace{1mm}
  \begin{tikzpicture}[scale=0.5]

\matching{12}

\begin{pgfonlayer}{background}
\draw[draw=none, fill=black!25!white] (1.center) -- (4.center) -- (5.center) -- (10.center) -- cycle;
\end{pgfonlayer}

\path[line_solid] (1) to (4);
\path[line_solid] (2) to (3);
\path[line_solid] (5) to (10);
\path[line_solid] (6) to (9);
\path[line_solid] (7) to (8);
\path[line_solid] (11) to (12);

\end{tikzpicture} \hspace{1mm}
  \begin{tikzpicture}[scale=0.5]

\matching{12}

\begin{pgfonlayer}{background}
\draw[draw=none, fill=black!25!white] (1.center) -- (6.center) -- (9.center) -- (10.center) -- cycle;
\end{pgfonlayer}
    
\path[line_solid] (1) to (10);
\path[line_solid] (2) to (3);
\path[line_solid] (4) to (5);
\path[line_solid] (6) to (9);
\path[line_solid] (7) to (8);
\path[line_solid] (11) to (12);

\end{tikzpicture} \hspace{1mm}
  \begin{tikzpicture}[scale=0.5]

\matching{12}
    
\begin{pgfonlayer}{background}
\draw[draw=none, fill=black!25!white] (1.center) -- (6.center) -- (11.center) -- (12.center) -- cycle;
\end{pgfonlayer}

\path[line_solid] (1) to (6);
\path[line_solid] (2) to (3);
\path[line_solid] (4) to (5);
\path[line_solid] (7) to (8);
\path[line_solid] (9) to (10);
\path[line_solid] (11) to (12);

\end{tikzpicture} \hspace{1mm}
  \begin{tikzpicture}[scale=0.5]

\matching{12}
    
\begin{pgfonlayer}{background}
\draw[draw=none, fill=black!25!white] (2.center) -- (3.center) -- (6.center) -- (11.center) -- cycle;
\end{pgfonlayer}

\path[line_solid] (1) to (12);
\path[line_solid] (2) to (3);
\path[line_solid] (4) to (5);
\path[line_solid] (6) to (11);
\path[line_solid] (7) to (8);
\path[line_solid] (9) to (10);

\end{tikzpicture} \hspace{1mm}
  \begin{tikzpicture}[scale=0.5]

\matching{12}
    
\path[line_solid] (1) to (12);
\path[line_solid] (2) to (11);
\path[line_solid] (3) to (6);
\path[line_solid] (4) to (5);
\path[line_solid] (7) to (8);
\path[line_solid] (9) to (10);

\draw (0,0) node {$M'$};

\end{tikzpicture} \hspace{1mm}
 \end{center}
 \caption{Definition of path $P$ in  $G_6^\tM$ from $M$ to $M'$.}
 \label{fig:Ppath6}
\end{figure}
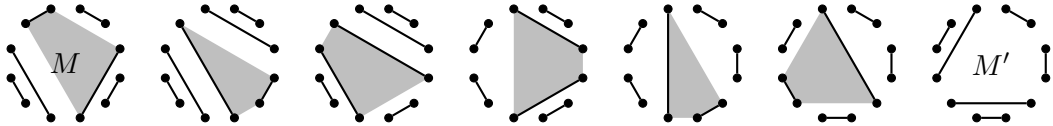

The gray areas in the figure highlight the quadrilaterals involved in each flip from left to right.
Note that $M'$ differs from $M$ by a clockwise rotation by an angle of $\alpha:=2\pi/6$.
It is easy to check that repeating this flip sequence six times, rotating all flips by an angle of $\alpha\cdot i$ for $i=0,1,\dots,5$, yields a cycle $C$ in $G_m^\tM$.
This cycle is highlighted in Figure~\ref{fig:HM6}.
To verify that $C$ is indeed a 2-rainbow cycle, consider the 12 matching edges that appear along the path $P$ and are shown in Figure~\ref{fig:6flipped-edges}.

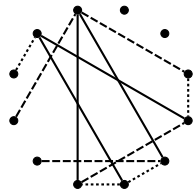
\begin{figure}[htb]
 \begin{center}
% draw the given number of points equidistantly on a circle
% and call them (1),...,(n)
\newcommand{\matching}[1]{
\foreach \i in {1,...,#1} {
  \node[node_black] (\i) at (-\i*360/#1-90+180/#1:1.5) {};
  %\draw (-\i*360/#1-90+180/#1:1.8) node {\scriptsize $\i$};
}
}

\begin{tikzpicture}[scale=0.8]

\matching{12}
    
\path[line_solid] (12) to (5);
\path[line_solid] (1) to (6);
\path[line_solid] (5) to (10);
\path[line_solid] (6) to (11);

\path[line_dashed] (1) to (10);
\path[line_dashed] (2) to (11);
\path[line_dashed] (3) to (6);
\path[line_dashed] (6) to (9);

\path[line_dotted] (12) to (1);
\path[line_dotted] (12) to (11);
\path[line_dotted] (4) to (5);
\path[line_dotted] (9) to (10);

\end{tikzpicture}
\end{center}
\caption{The 12 matching edges that appear along the path $P$ in $G_6^\tM$.}
\label{fig:6flipped-edges}
\end{figure}

In the figure we differentiate the different lengths of the matching edges by three different line styles solid, dashed or dotted.
Note that there are four edges from each length.
It is straightforward to check that rotating this figure by $\alpha\cdot i$ for $i=0,1,\dots,5$ covers each matching edge from $E_m$ exactly twice. Consequently, $C$ is a 2-rainbow cycle in $G_6^\tM$.

For $m=8$, a 2-rainbow cycle in $G_m^\tM$ can be constructed using the path $P$ of length 8 between matchings $M$ and $M'$ depicted in Figure~\ref{fig:Ppath8}.

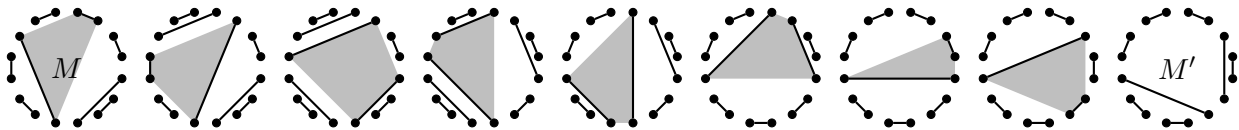
\begin{figure}[h!]
\begin{center}
% draw the given number of points equidistantly on a circle
% and call them (1),...,(n)
\newcommand{\matching}[1]{
\foreach \i in {1,...,#1} {
  \node[node_black] (\i) at (-\i*360/#1-90+180/#1:1.5) {};
  % \draw (-\i*360/#1-90+180/#1:1.8) node {$\i$};
}
}

\begin{tikzpicture}[scale=0.5]

\matching{16}

\begin{pgfonlayer}{background}
\draw[draw=none, fill=black!25!white] (1.center) -- (6.center) -- (9.center) -- (10.center) -- cycle;
\end{pgfonlayer}

\path[line_solid] (1) to (6);
\path[line_solid] (2) to (3);
\path[line_solid] (4) to (5);
\path[line_solid] (7) to (8);
\path[line_solid] (9) to (10);
\path[line_solid] (11) to (12);
\path[line_solid] (13) to (16);
\path[line_solid] (14) to (15);

\draw (0,0) node {$M$};

\end{tikzpicture}
\begin{tikzpicture}[scale=0.5]

\matching{16}

\begin{pgfonlayer}{background}
\draw[draw=none, fill=black!25!white] (1.center) -- (4.center) -- (5.center) -- (10.center) -- cycle;
\end{pgfonlayer}

\path[line_solid] (1) to (10);
\path[line_solid] (2) to (3);
\path[line_solid] (4) to (5);
\path[line_solid] (6) to (9);
\path[line_solid] (7) to (8);
\path[line_solid] (11) to (12);
\path[line_solid] (13) to (16);
\path[line_solid] (14) to (15);

\end{tikzpicture}
\begin{tikzpicture}[scale=0.5]

\matching{16}

\begin{pgfonlayer}{background}
\draw[draw=none, fill=black!25!white] (5.center) -- (10.center) -- (13.center) -- (16.center) -- cycle;
\end{pgfonlayer}

\path[line_solid] (1) to (4);
\path[line_solid] (2) to (3);
\path[line_solid] (5) to (10);
\path[line_solid] (6) to (9);
\path[line_solid] (7) to (8);
\path[line_solid] (11) to (12);
\path[line_solid] (13) to (16);
\path[line_solid] (14) to (15);

\end{tikzpicture}
\begin{tikzpicture}[scale=0.5]

\matching{16}

\begin{pgfonlayer}{background}
\draw[draw=none, fill=black!25!white] (5.center) -- (6.center) -- (9.center) -- (16.center) -- cycle;
\end{pgfonlayer}

\path[line_solid] (1) to (4);
\path[line_solid] (2) to (3);
\path[line_solid] (5) to (16);
\path[line_solid] (6) to (9);
\path[line_solid] (7) to (8);
\path[line_solid] (10) to (13);
\path[line_solid] (11) to (12);
\path[line_solid] (14) to (15);

\end{tikzpicture}
\begin{tikzpicture}[scale=0.5]

\matching{16}

\begin{pgfonlayer}{background}
\draw[draw=none, fill=black!25!white] (1.center) -- (4.center) -- (9.center) -- (16.center) -- cycle;
\end{pgfonlayer}

\path[line_solid] (1) to (4);
\path[line_solid] (2) to (3);
\path[line_solid] (5) to (6);
\path[line_solid] (7) to (8);
\path[line_solid] (9) to (16);
\path[line_solid] (10) to (13);
\path[line_solid] (11) to (12);
\path[line_solid] (14) to (15);

\end{tikzpicture}
\begin{tikzpicture}[scale=0.5]

\matching{16}

\begin{pgfonlayer}{background}
\draw[draw=none, fill=black!25!white] (4.center) -- (9.center) -- (10.center) -- (13.center) -- cycle;
\end{pgfonlayer}

\path[line_solid] (1) to (16);
\path[line_solid] (2) to (3);
\path[line_solid] (4) to (9);
\path[line_solid] (5) to (6);
\path[line_solid] (7) to (8);
\path[line_solid] (10) to (13);
\path[line_solid] (11) to (12);
\path[line_solid] (14) to (15);

\end{tikzpicture}
\begin{tikzpicture}[scale=0.5]

\matching{16}

\begin{pgfonlayer}{background}
\draw[draw=none, fill=black!25!white] (4.center) -- (11.center) -- (12.center) -- (13.center) -- cycle;
\end{pgfonlayer}

\path[line_solid] (1) to (16);
\path[line_solid] (2) to (3);
\path[line_solid] (4) to (13);
\path[line_solid] (5) to (6);
\path[line_solid] (7) to (8);
\path[line_solid] (9) to (10);
\path[line_solid] (11) to (12);
\path[line_solid] (14) to (15);

\end{tikzpicture}
\begin{tikzpicture}[scale=0.5]

\matching{16}

\begin{pgfonlayer}{background}
\draw[draw=none, fill=black!25!white] (4.center) -- (11.center) -- (14.center) -- (15.center) -- cycle;
\end{pgfonlayer}

\path[line_solid] (1) to (16);
\path[line_solid] (2) to (3);
\path[line_solid] (4) to (11);
\path[line_solid] (5) to (6);
\path[line_solid] (7) to (8);
\path[line_solid] (9) to (10);
\path[line_solid] (12) to (13);
\path[line_solid] (14) to (15);

\end{tikzpicture}
\begin{tikzpicture}[scale=0.5]

\matching{16}

\path[line_solid] (1) to (16);
\path[line_solid] (2) to (3);
\path[line_solid] (4) to (15);
\path[line_solid] (5) to (6);
\path[line_solid] (7) to (8);
\path[line_solid] (9) to (10);
\path[line_solid] (11) to (14);
\path[line_solid] (12) to (13);

\draw (0,0) node {$M'$};

\end{tikzpicture}
\end{center}
\caption{Definition of path $P$ between matchings $M$ and $M'$ in $G_8^\tM$.}
\label{fig:Ppath8}
\end{figure}

Note that $M'$ differs from $M$ by a counter-clockwise rotation by an angle of $\alpha:=2\pi/8$.
It is easy to check that repeating this flip sequence eight times, rotating all flips by an angle of $\alpha\cdot i$ for $i=0,1,\dots,7$, yields a cycle $C$ in $G_m^\tM$.
To verify that $C$ is indeed a 2-rainbow cycle, consider the 16 matching edges that appear along the path $P$ shown in Figure~\ref{fig:8flipped-edges}.

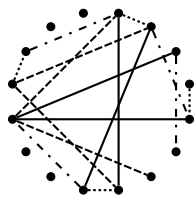
\begin{figure}[h]
 \begin{center}
% draw the given number of points equidistantly on a circle
% and call them (1),...,(n)
\newcommand{\matching}[1]{
\foreach \i in {1,...,#1} {
  \node[node_black] (\i) at (-\i*360/#1-90+180/#1:1.5) {};
  %\draw (-\i*360/#1-90+180/#1:1.8) node {\scriptsize $\i$};
}
}

\begin{tikzpicture}[scale=0.8]

\matching{16}
    
\path[line_solid] (1) to (10);
\path[line_solid] (9) to (16);
\path[line_solid] (4) to (13);
\path[line_solid] (4) to (11);

\path[line_dashed] (5) to (10);
\path[line_dashed] (5) to (16);
\path[line_dashed] (4) to (9);
\path[line_dashed] (4) to (15);

\path[line_dashdotted] (6) to (9);
\path[line_dashdotted] (1) to (4);
\path[line_dashdotted] (10) to (13);
\path[line_dashdotted] (11) to (14);

\path[line_dotted] (5) to (6);
\path[line_dotted] (1) to (16);
\path[line_dotted] (9) to (10);
\path[line_dotted] (12) to (13);

\end{tikzpicture}
\end{center}
\caption{The 16 matching edges that appear along the path $P$ in $G_8^\tM$.}
\label{fig:8flipped-edges}
\end{figure}

The different lengths of the matching edges are visualized by four different line styles solid, dashed, dash dotted, and dotted.
Note that there are four edges from each length.
It is straightforward to check that rotating this figure by $\alpha\cdot i$ for $i=0,1,\dots,7$ covers each matching edge from $E_m$ exactly twice.  Consequently, $C$ is a 2-rainbow cycle in $G_8^\tM$.
\end{proof}

\subsection{Structure of the graph $H_m$}

In this section we prove that the graph $H_m$ has at least $m-1$ connected components (Theorem~\ref{thm:Hm-components} below).

The following definitions and Lemma~\ref{lem:weight-change} are illustrated in Figure~\ref{fig:weight}.
Consider a matching $M\in\cM_m$ and one of its edges $e\in M$, and let $i$ and $j$ be the endpoints of $e$ so that the origin lies to the right of the ray from $i$ to $j$.
We define the \emph{sign of the edge $e$} as
\begin{equation*}
\sgn(e):=\begin{cases}
          +1 & \text{if $i$ is odd}, \\
          -1 & \text{if $i$ is even}.
         \end{cases}
\end{equation*}
Moreover, we define the \emph{weight} of the matching $M$ as
\begin{equation*}
w(M):=\sum_{e \in M} \sgn(e) \cdot \ell(e).
\end{equation*}
Note that rotating a matching by $\pi/m$ changes the weight by a factor of $-1$.

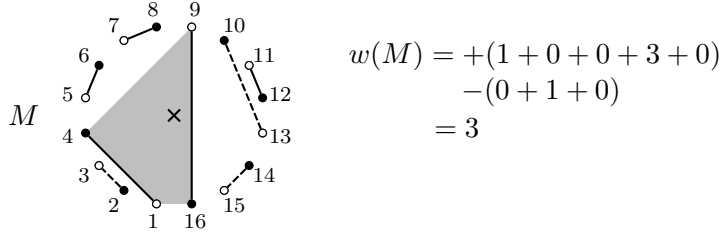
\begin{figure}[htb]
% draw the given number of points equidistantly on a circle
% and call them (1),...,(n)
\newcommand{\matching}[1]{
\foreach \i in {1,...,#1} {
  \node[node_black] (\i) at (-\i*360/#1-90+180/#1:1.5) {};
  \draw (-\i*360/#1-90+180/#1:1.8) node {\scriptsize $\i$};
}
}

\begin{tikzpicture}[scale=0.8]

\matching{16}

\begin{pgfonlayer}{background}
\draw[draw=none, fill=black!25!white] (1.center) -- (4.center) -- (9.center) -- (16.center) -- cycle;
\end{pgfonlayer}

\path[line_solid] (-0.1,-0.1) -- (0.1,0.1);
\path[line_solid] (-0.1,0.1) -- (0.1,-0.1);

\path[line_solid] (1) to (4);
\path[line_dashed] (2) to (3);
\path[line_solid] (5) to (6);
\path[line_solid] (7) to (8);
\path[line_solid] (9) to (16);
\path[line_dashed] (10) to (13);
\path[line_solid] (11) to (12);
\path[line_dashed] (14) to (15);

% paint some points white
\foreach \i in {1,3,5,7,9,11,13,15}
\draw[fill=white] (\i) circle(1.8pt);

\draw (-2.5,0) node {$M$};
\draw (6,1) node {$w(M)=+(1+0+0+3+0)$};
\draw (6.1,.4) node {$-(0+1+0)$};
\draw (4.7,-.2) node {$=3$};
\draw[fill=white] (-7,0) node {};
\end{tikzpicture}
\caption{Illustration of the weight of a matching with $m=8$ edges.
Edges with sign $+1$ are drawn solid, edges with sign $-1$ are drawn dashed.
Applying the flip indicated in the figure changes the weight by $-(1+3)-(2+0)=-6=-(m-2)$.
}
\label{fig:weight}
\end{figure}

The following lemma is an immediate consequence of these definitions.

\begin{lemma}
\label{lem:weight-change}
A centered 4-gon has two edges with positive sign and two edges with negative sign, and the pairs of edges with the same sign are opposite to each other.
Consequently, applying a centered flip to any matching from $\cM_m$ changes its weight by $-(m-2)$ if the two edges with negative sign appear in this flip, or by $+(m-2)$ if the two edges with positive sign appear in this flip.
Flips of these two kinds must alternate along any sequence of centered flips.
\end{lemma}

\begin{proof}
The first part of the lemma is an immediate consequence of the definitions given before.
To see that flips that change the weight by $-(m-2)$ or $+(m-2)$ must alternate along any sequence of centered flips, note that in any matching $M\in\cM_m$, all edges that are visible from the origin have the same sign, and any flip must change this sign.
\end{proof}

The next lemma shows that the weight of a matching lies in a specific interval.

\begin{lemma}
\label{lem:weight-range}
Given any matching $M\in\cM_m$, we have 
\begin{equation*}
w(M) \in [-(m-2), m -2] := \big\{-(m-2), -(m-2) +1, \dots, m -3, m -2 \big\}. 
\end{equation*}
For any integer $c$ in this set, there is a matching $M\in\cM_m$ with $w(M)=c$.
\end{lemma}

The weights of all matchings with $m=6$ edges are shown in Figure~\ref{fig:HM6}.

\begin{proof}
We fix a matching $M\in \cM_m$ throughout the proof.
For every odd $k\in[n]$ we consider the ray $r_k$ from the origin $(0,0)$ through the point $k$; see Figure~\ref{fig:rays}.
This yields a set $C(r_k)$ of points $x=r_k\cap e$ where the ray $r_k$ crosses any matching edge $e\in M$ in its interior.
We define the \emph{sign} of this crossing point $x$ as the sign of the matching edge involved, i.e., $\sgn(x):=\sgn(e)$.
Moreover, we define the \emph{weight} of the ray $r_k$ as the sum of signs of all crossing points from $C(r_k)$ along the ray, i.e., $w(r_k):=\sum_{x\in C(r_k)}\sgn(x)$.
Let $C$ be the union of all these crossing points between rays and matching edges.

Note that the number of rays that cross a fixed edge $e \in M$ is equal to $\ell(e)$, implying that
\begin{equation}
\label{eq:wM}
w(M) = \sum_{x\in C} \sgn(x)=\sum_{k\in[n]\text{ odd}} w(r_k).
\end{equation}

We claim that if we follow any ray $r_k$, then the signs of any two consecutive crossing points $x$ and $y$ along the ray alternate, i.e., $\sgn(x)+\sgn(y)=0$.
To see this let $\{i,j\}$ and $\{p,q\}$ denote the edges from $M$ causing the crossings $x$ and $y$, respectively, such that $i$ and $p$ lie to the left of the ray $r_k$, and $j$ and $q$ lie to the right.
Moreover, let $A$ be the set of points from $[n]$ between $i$ and $p$, and let $B$ be the points between $j$ and $q$; see Figure~\ref{fig:rays}.
Observe that all points from $A$ must be matched within the set, and the same holds for all points within $B$.
This is because if there was a matching edge between $A$ and $B$, then it would cross $r_k$ between $x$ and $y$.
It follows that $|A|$ and $|B|$ are even, and hence the distance between $i$ and $p$, and the distance between $j$ and $k$ are both odd.
This implies that the edges $\{i,j\}$ and $\{p,q\}$ have opposite signs, proving the claim.

\begin{figure}[htb]
\begin{center}
\includegraphics{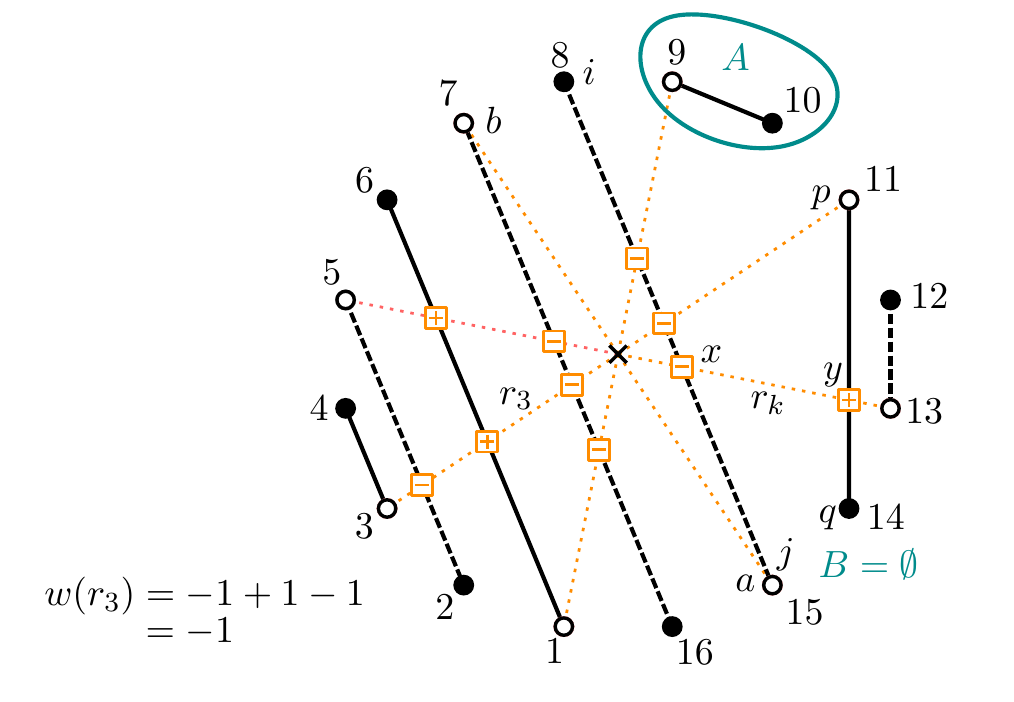}
\end{center}
\caption{Illustration of the proof of Lemma~\ref{lem:weight-range}.} 
\label{fig:rays}
\end{figure}

An immediate consequence of this claim is that for any ray $r_k$ we have $w(r_k)\in \{-1,0,1\}$.
Moreover, there are at least two matching edges visible from the origin and therefore there are two odd points $a$ and $b$ for which $r_a$ and $r_b$ do not cross any matching edges, i.e., $w(r_a) = w(r_b) = 0$.
It follows that two of the summands on the right hand side of \eqref{eq:wM} are 0, and the remaining $m-2$ summands are from the set $\{-1,0,1\}$.
This proves the first part of the lemma.

To prove the second part of the lemma, consider a matching that has exactly two edges whose lengths sum up to $c$ that are both visible from the origin $(0,0)$ and $m-2$ matching edges of length 0.
\end{proof}

Motivated by Lemma~\ref{lem:weight-change} and Lemma~\ref{lem:weight-range}, we partition the set of all matchings $\cM_m$ into sets $\cM_{m,c}$, $c\in[-(m-2),(m-2)]$, where $\cM_{m,c}$ contains all matchings with weight exactly $c$.
Moreover, we define $\cM_{m,c}^+:=\cM_{m,c}\cup \cM_{m,c-(m-2)}$ for $c=0,1,\dots,m-2$.
These two lemmas imply the following structural result about the graph $H_m$.

\begin{theorem}
\label{thm:Hm-components}
For any even $m\geq 2$, the subgraph $H_m$ of $G_m^\tM$ that uses only centered flips has no edges between any two partition classes $\cM_{m,c}^+$, $c=0,1,\dots,m-2$, and therefore at least one connected component in each partition class, in total at least $m-1$ connected components.
\end{theorem}

As Figure~\ref{fig:HM6} shows, the subgraph of $H_m$ induced by a partition class $\cM_{m,c}^+$ is not necessarily connected, i.e., the number of connected components of $H_m$ may exceed $m-1$.
For instance, the subgraph of $H_6$ induced by the partition class $\cM_{6,2}^+$ has four connected components and the total number of components of $H_6$ is eight.

We observed empirically for $m\in\{2,4,6,8\}$ that for any $c\in[-(m-2),m-2]$, the number of matchings in $\cM_{m,c}$ is given by
\begin{equation}
\label{eq:Mmc-size}
|\cM_{m,c}|=\begin{cases}
           2 & \text{if $c=0$}, \\
           N_1(m,|c|+1)/2 & \text{if $|c|\geq 1$},
          \end{cases}
\end{equation}
where $N_r(m,k)$ are the \emph{generalized Narayana numbers}, defined as
\begin{equation*}
N_r(m,k) = \frac{r + 1}{m + 1} \binom{m + 1}{k} \binom{m - r - 1}{k - 1}
\end{equation*}
for integers $r\geq 0$ and $0\leq k\leq m-r$.
The quantity $N_r(m,k)$ counts Dyck paths in the integer lattice $\mathbb{Z}^2$ starting at the origin with $m$ upsteps $(+1,+1)$ and $m-r$ downsteps $(+1,-1)$ with exactly $k$ peaks. The Dyck path property means that such a path never moves below the abscissa.

Unfortunately, we are not able to prove \eqref{eq:Mmc-size} in general.
Proving this relation would allow us to exactly compute the cardinalities of the partition classes $\cM_{m,c}^+$ referred to in Theorem~\ref{thm:Hm-components}.

% \begin{proof}
% The bijection given in \cite{MR2732117} explains part of the behavior we see.
% \end{proof}

% \item
% Consider the moment an edge $\{i,j\}\in E_m$ of maximum length $(m-2)/2$ appears as part of a flippable centered 4-gon $P$.
% It has to disappear immediately after it was added, as every flippable centered 4-gon $Q$ contains this edge.
% Both 4-gons $P$ and $Q$ lie on the same side of the edge $\{i,j\}$ (the side containing the origin).
% The edges in the 4-gons $P$ and $Q$ opposite the edge $\{i,j\}$ are different, and they are not nested (both are visible from $\{i,j\}$).

% \item
% Any matching containing two edges $\{i,j\}$, $\{i-1,j+1\}$ of maximal length $(m-2)/2$ cannot be part of a rainbow cycle.
% These two edges can only appear and disappear together because of \ref{it:max-length}, leading back to the same configuration.

% \item 
% Any centered flip transforms a point-symmetric matchings into a point-symmetric matchings.
% Point-symmetric matchings therefore form trees in $H_m$, implying that a rainbow cycle can only contain matchings that are not point-symmetric.
% \end{enumerate}

\section{Permutations}
\label{sec:perm}

In this section, we consider the set of permutations $\Pi_n$ of $[n]$.
We specify a permutation $\pi\in \Pi_n$ as $\pi=(\pi(1),\pi(2),\dots,\pi(n))$.
The graph $G_n^\tP$ has $\Pi_n$ as its vertex set, and an edge $\{\pi,\rho\}$ between any two permutations $\pi$ and $\rho$ that differ in exactly one transposition between the entries at positions $i$ and $j$; see Figure~\ref{fig:flip}~(c).
We label the edge $\{\pi,\rho\}$ of $G_n^\tP$ with the transposition $\{i,j\}$, and we think of these labels as colors.
A 1-rainbow cycle in $G_n^\tP$ is an undirected cycle along which every transposition appears exactly once, so it has length $\binom{n}{2}$.
In this section we only consider 1-rainbow cycles, and we simply refer to them as rainbow cycles.
Note that the number of vertices of $G_n^\tP$ is $n!$.

The following theorem summarizes the results of this section.

\begin{theorem}
\label{thm:perm}
The flip graph of permutations $G_n^\tP$, $n\geq 2$, has the following properties:
\begin{enumerate}[label=(\roman*)]
\item If $\lfloor n/2\rfloor$ is odd, then $G_n^\tP$ has no rainbow cycle.
\item If $\lfloor n/2\rfloor$ is even, then $G_n^\tP$ has a rainbow cycle.
\end{enumerate}
\end{theorem}

\begin{proof}
The graph $G_n^\tP$ is bipartite, since the parity changes in each step, so a cycle of length $\binom{n}{2}$ cannot exist when this number is odd, which happens exactly when $\lfloor n/2\rfloor$ is odd.
This proves (i).

To prove (ii) we assume that $\lfloor n/2\rfloor$ is even, i.e., $n=4\ell$ or $n=4\ell+1$ for some integer $\ell\geq 1$.
We prove these cases by induction.
As the graph $G_n^\tP$ is vertex-transitive, it suffices to specify a sequence of $\binom{n}{2}$ transpositions that yields a rainbow cycle.
We refer to such a sequence as a \emph{rainbow sequence for $\Pi_n$}.
A rainbow sequence of transpositions can be applied to any vertex in $G_n^\tP$, yielding a rainbow cycle.
To settle the induction base $\ell=1$, consider the rainbow sequence
\begin{equation*}
R_4:=\big(\{1,2\},\{3,4\},\{2,3\},\{1,4\},\{2,4\},\{1,3\}\big).
\end{equation*}
Applying $R_4$ to the permutation $1234$ yields the rainbow cycle $C_4:=(1234,2134,2143,2413,3412,3214)$ in $G_4^\tP$, where we omit brackets and commas in denoting these single-digit permutations.
Note that applying the transposition $\{1,3\}$ to the last permutation in $C_4$ yields the first one.
This rainbow cycle is depicted in Figure~\ref{fig:perm}~(a).

\begin{figure}[htb]
\begin{center}
\includegraphics{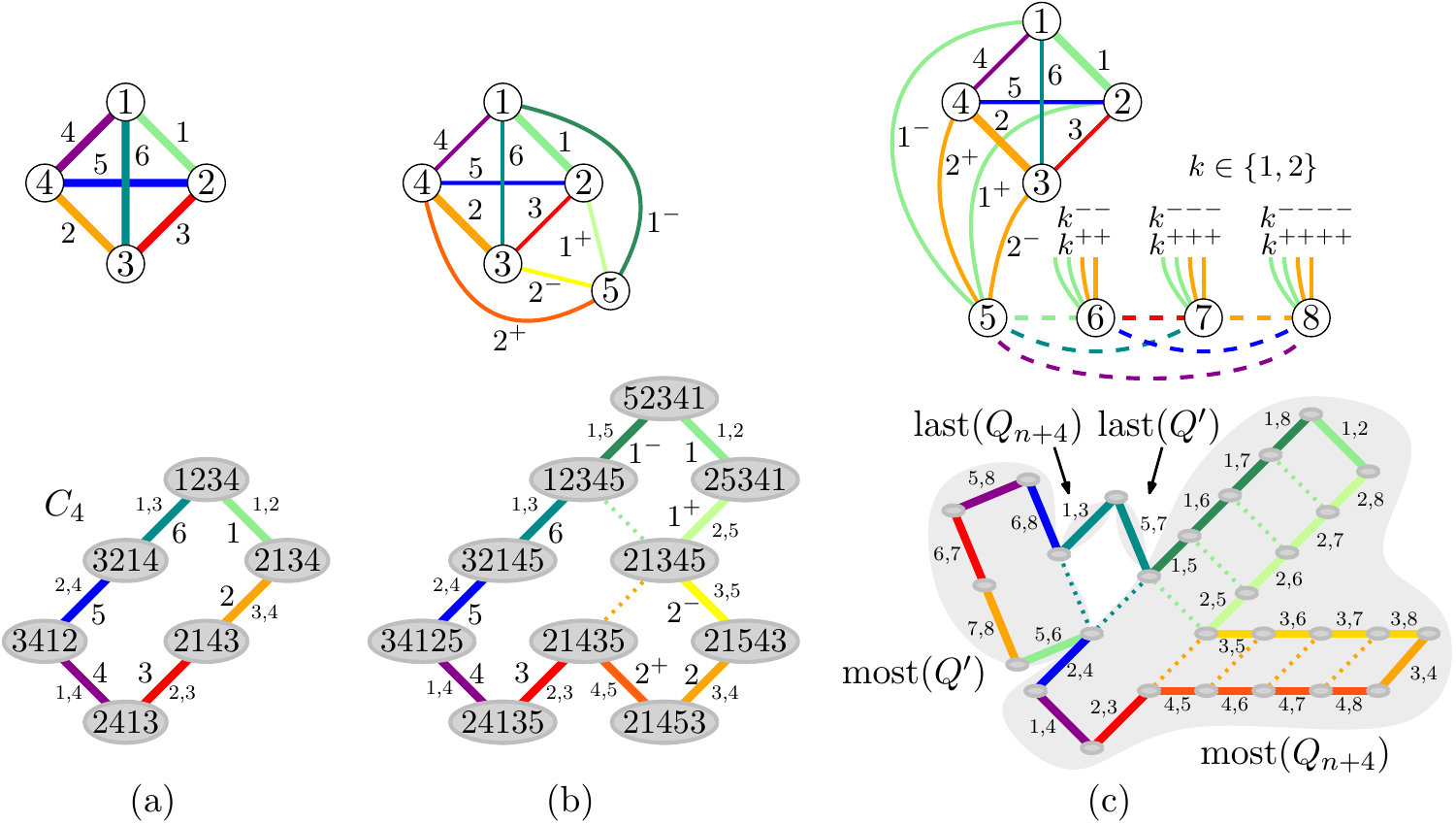}
\caption{Illustration of the proof of Theorem~\ref{thm:perm}:
(a) induction base for $n=4$,\break
(b) induction step $n\rightarrow n+1$,
(c) induction step $n\rightarrow n+4$ (where $n=4\ell$).
The rainbow cycles are shown in the bottom row, the corresponding edge orderings of $K_n$ in the top row.
In the top parts of (b) and (c), the bold edges mark the transpositions $t_i=\{2i-1,2i\}$, $i=1,2,\dots,n/2$, and in the edge orderings the edges labeled with $-$ and $+$ superscripts appear consecutively in the order $k^-,k,k^+$ or $k^{-},k^{--},k^{---},k^{----},k,k^{++++},k^{+++},k^{++},k^{+}$, respectively, for $k\in\{1,2\}$.
Single-digit permutations and transpositions are denoted without brackets and commas.
}
\label{fig:perm}
\end{center}
\end{figure}

We can interpret every transposition in a rainbow sequence for $\Pi_n$ as an edge in $K_n$, yielding an ordering $1,2,\dots,\binom{n}{2}$ of the edges of $K_n$.
In Figure~\ref{fig:perm}, these edge orderings are shown in the top row.
Whether an ordering of the edges of $K_n$ corresponds to a valid rainbow sequence for $\Pi_n$ or not can be decided as follows:
Without loss of generality we start at the identity permutation $(1,2,\dots,n)$, we apply the transpositions given by the edge ordering one after the other, checking that each permutation is encountered at most once and that the final permutation is again the identity permutation.

For the induction step, we assume that we are given a rainbow sequence $R_n$ for $\Pi_n$, $n=4\ell$, and construct rainbow sequences for $\Pi_{n+1}=\Pi_{4\ell+1}$ and for $\Pi_{n+4}=\Pi_{4(\ell+1)}$. 
To this end we consider the effect of replacing a transposition $t=\{i,j\}$, $i<j$, in $R_n$ by the sequence of three transpositions $\that(n+1):=(\{i,n+1\},\{i,j\},\{j,n+1\})$.
Note that the effect of this modification on the entries at positions $i$, $j$ and $n+1$ of the permutation is the same, the only difference is that in the modified sequence also the transpositions $\{i,n+1\}$ and $\{j,n+1\}$ are used.
This simple observation is the key to the following inductive construction.

We first assume that $n=4\ell$, $\ell\geq 1$, and given a rainbow sequence $R_n$ for $\Pi_n$ we construct a rainbow sequence $R_{n+1}$ for $\Pi_{n+1}=\Pi_{4\ell+1}$.
For this we consider the transpositions $t_i:=\{2i-1,2i\}$, $i=1,2,\dots,n/2$, in $R_n$, and to construct $R_{n+1}$ we replace each $t_i$ by the triple $\that_i(n+1)$; see Figure~\ref{fig:perm}~(b).
This yields a sequence $R_{n+1}$ of length $\binom{n}{2}+2\cdot\frac{n}{2}=\binom{n+1}{2}$, and using our previous observation it is easy to check that $R_{n+1}$ is indeed a rainbow sequence for $\Pi_{n+1}$.

We now assume that $n=4\ell$, $\ell\geq 1$, and given a rainbow sequence $R_n$ for $\Pi_n$ we construct a rainbow sequence $R_{n+4}$ for $\Pi_{n+4}=\Pi_{4(\ell+1)}$.
For this we consider the transpositions $t_i:=\{2i-1,2i\}$, $i=1,2,\dots,n/2$, in $R_n$.
We assume without loss of generality that the last transposition in $R_n$ is not one of the $t_i$ (otherwise shift $R_n$ cyclically).
We then define $Q_n:=R_n$ and construct auxiliary sequences $Q_{n+1}$, $Q_{n+2}$, $Q_{n+3}$, $Q_{n+4}$ as follows:
For $k\in\{1,2,3,4\}$, $Q_{n+k}$ is obtained from $Q_{n+k-1}$ by replacing each $t_i$, $i=1,2,\dots,n/2$, in $Q_{n+k-1}$ by the triple $\that_i(n+k)$.
Now let $Q'$ be the sequence of transpositions on the elements $\{n+1,n+2,n+3,n+4\}$ obtained from $R_4$ by adding $n$ to all elements.
The sequence $R_{n+4}$ is constructed by suitably interleaving $Q_{n+4}$ and $Q'$.
Specifically, for any sequence $x=(x_1,x_2,\dots,x_{k+1})$ we define $\most(x):=(x_1,x_2,\dots,x_k)$ and $\last(x):=x_{k+1}$, so $x=(\most(x),\last(x))$.
We then define
\begin{equation*}
R_{n+4}:=\big(\most(Q_{n+4}),\most(Q'),\last(Q_{n+4}),\last(Q')\big) ;
\end{equation*}
see Figure~\ref{fig:perm}~(c).
By construction, the last transposition in $R_n$ is also the last transposition in $Q_{n+4}$.
This yields a sequence $R_{n+4}$ of length $\binom{n}{2}+8\cdot\frac{n}{2}+6=\binom{n+4}{2}$, and using our previous observation it is easy to check that $R_{n+4}$ is indeed a rainbow sequence for $\Pi_{n+4}$.
A straightforward calculation shows that by this interleaving, no permutation is encountered twice along the corresponding cycle.

This completes the proof of (ii).
\end{proof}

\section{Subsets}
\label{sec:comb}

In this section we consider the set of all $k$-element subsets of $[n]$, denoted by $C_{n,k}:=\binom{[n]}{k}$, sometimes called $(n,k)$-combinations.
The graph $G_{n,k}^\tC$ has $C_{n,k}$ as its vertex set, and an edge $\{A,B\}$ between any two sets $A$ and $B$ that differ in exchanging an element $x$ for another element $y$, i.e., $A\setminus B=\{x\}$ and $B\setminus A=\{y\}$; see Figure~\ref{fig:flip}~(d).
We label the edge $\{A,B\}$ of $G_{n,k}^\tC$ with the transposition $A\triangle B=\{x,y\}\in C_{n,2}$, and we think of these labels as colors.
A 1-rainbow cycle in $G_{n,k}^\tC$ is an undirected cycle along which every transposition appears exactly once, so it has length $\binom{n}{2}$.
In this section we only consider 1-rainbow cycles, and we simply refer to them as rainbow cycles.
The number of vertices of $G_{n,k}^\tC$ is clearly $\binom{n}{k}$.
Consequently, a rainbow cycle for $k=2$ is in fact a Hamilton cycle, i.e., a Gray code in the classical sense.
As $G_{n,k}^\tC$ and $G_{n,n-k}^\tC$ are isomorphic, including the edge labels, we will asssume without loss of generality that $k\leq \lfloor n/2\rfloor$.
Also note that for $k=1$, the number of vertices of $G_{n,k}^\tC$ is only $n$, which is strictly smaller than $\binom{n}{2}$ for $n>3$, so we will also assume that $k\geq 2$.

The following theorem summarizes the results of this section.

\begin{theorem}
\label{thm:comb}
Let $n\geq 4$ and $2\leq k\leq \lfloor n/2\rfloor$.
The flip graph of subsets $G_{n,k}^\tC$ has the following properties:
\begin{enumerate}[label=(\roman*)]
\item If $n$ is even, then $G_{n,k}^\tC$ has no rainbow cycle.
\item If $n$ is odd and $k=2$, then $G_{n,2}^\tC$ has a rainbow Hamilton cycle.
\item If $n$ is odd and $k=2$, then $G_{n,2}^\tC$ has two edge-disjoint rainbow Hamilton cycles.
\item If $n$ is odd and $3\leq k<n/3$, then $G_{n,k}^\tC$ has a rainbow cycle.
\end{enumerate}
\end{theorem}

With the help of a computer we found even more than two edge-disjoint rainbow Hamilton cycles in $G_{n,2}^\tC$ for odd $n$; see Table~\ref{tab:hcs}.
Moreover, we firmly believe the $G_{n,k}^\tC$ also has a rainbow cycle for $n/3\leq k\leq \lfloor n/2\rfloor$, but we are not able to prove this.

\subsection{Proof of Theorem~\ref{thm:comb}~(i)}
We start to show the non-existence of rainbow cycles in the case that $n$ is even.
\begin{proof}[Proof of Theorem~\ref{thm:comb} (i)]
Note that, for a fixed element $x\in[n]$, there are $n-1$ transpositions involving $x$.
If $x$ is in a set along a rainbow cycle and such a transposition is applied, then the next set along the cycle does not contain $x$, and vice versa.
In a rainbow cycle we return to the starting set and use each of these transpositions exactly once, so $n-1$ must be even, or equivalently, $n$ must be odd.
\end{proof}

\subsection{Proof of Theorem~\ref{thm:comb}~(ii)}

For the rest of this section we assume that $n$ is odd, i.e., $n=2\ell+1$ for some integer $\ell\geq 2$.
To prove parts (ii)--(iv) of Theorem~\ref{thm:comb}, we construct rainbow cycles using a \emph{rainbow block}.
To introduce this notion, we need some definitions.
For any set $A\subseteq [n]$ we let $\sigma(A)$ denote the set obtained from $A$ by adding 1 to all elements, modulo $n$ with $\{1,2,\dots,n\}$ as residue class representatives.
Moreover, for any pair $\{x,y\}\in C_{n,2}$, we define $\dist(\{x,y\}):=\min\{y-x,x-y\}\in[\ell]$ where the differences are also taken modulo $n$.

We call a sequence $B=(B_1,B_2,\dots,B_\ell)$ of subsets $B_i\in C_{n,k}$ a \emph{rainbow block} if
\begin{equation}
\label{eq:CB}
C(B):=\big(B,\sigma^1(B),\sigma^2(B),\dots,\sigma^{2\ell}(B)\big)
\end{equation}
is a rainbow cycle in $G_{n,k}^\tC$.
Note that the sequence $C(B)$ has the correct length $\ell\cdot (2\ell+1)=\binom{n}{2}$.
By definition, a rainbow cycle built from a rainbow block is highly symmetric.
In the following proofs we will formulate various sufficient conditions guaranteeing that $B$ is a rainbow block, and construct $B$ such that those conditions are satisfied.

\begin{proof}[Proof of Theorem~\ref{thm:comb}~(ii)]
Let $n=2\ell+1$ for some integer $\ell\geq 2$.
We define a sequence $B=(B_1,B_2,\dots,B_\ell)$ of pairs $B_i\in C_{n,2}$ such that the following conditions hold:
\begin{enumerate}[label=(\alph*),topsep=0mm,leftmargin=7mm]
\item $B_i=\{1,b_i\}$ for $i\in[\ell]$ with $3\leq b_i\leq n$ and $b_1=n$,
\item $\{\dist(B_i)\mid i\in[\ell]\}=[\ell]$, and
\item $\{\dist(B_i\triangle B_{i+1}) \mid i\in[\ell-1]\}\cup \{\dist(B_\ell\triangle \sigma(B_1)\}=[\ell]$.
\end{enumerate}

Figure~\ref{fig:block} shows a sequence $B$ satisfying these conditions for $\ell=8$.

\begin{figure}
\centering
\begin{minipage}[c]{.78\textwidth}
\setlength{\arraycolsep}{0.5mm}
$
\begin{array}{l|ccccccccccccccccc|c|l}
B_i & 1 & 2 & 3 & 4 & 5 & 6 & 7 & 8 & 9 & 10 & 11 & 12 & 13 & 14 & 15 & 16 & 17 & \;\dist(B_i) &\; \dist(B_i\triangle B_{i+1})=|d_i| \\  \hline \relax
B_1 & \times & & & & & & & & & & & & & & & & \times & 1 & \downshift{\qquad 3} \\ \relax
B_2 & \times & & \times & & & & & & & & & & & & & & & 2 & \downshift{\qquad 5} \\ \relax
B_3 & \times & & & & & & & & & & & & & & \times & & & 3 & \downshift{$\qquad 7=\ell-1$} \\ \relax
B_4 & \times & & & & \times & & & & & & & & & & & & & 4 & \downshift{\qquad 1} \\ \relax
B_5 & \times & & & & & \times & & & & & & & & & & & & 5 & \downshift{$\qquad 6=\ell-2$} \\ \relax
B_6 & \times & & & & & & & & & & & \times & & & & & & 6 & \downshift{\qquad 4} \\ \relax
B_7 & \times & & & & & & & \times & & & & & & & & & & 7 & \downshift{\qquad 2} \\ \relax
B_8 & \times & & & & & & & & & \times & & & & & & & & 8 & \downshift{\;\multirow[t]{2}{2cm}{$\dist(B_\ell\triangle \sigma(B_1))\newline \hspace*{5pt}=8=\ell$}} \\ \cline{1-19} \relax
\sigma(B_1) & \times & \times & & & & & & & & & & & & & & & & &\\
& & & & & & & & & & & & & & & & & & &\\
\end{array}
$
\end{minipage}
\begin{minipage}[c]{0.19\textwidth}
\includegraphics{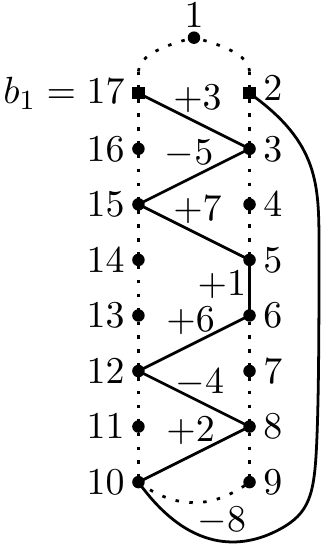}
\end{minipage}
\caption{A rainbow block for $\ell=8$, corresponding to case \eqref{eq:di-meven}.
A cross in row $B_i$ and column $j$ indicates that $j\in B_i$.
On the right hand side, the sequence $(b_1,b_2,\dots,b_\ell,2)$ for this block is depicted as a path.}
\label{fig:block}
\end{figure}

We claim that a sequence $B$ satisfying (a)--(c) is a rainbow block, i.e., the cycle $C=C(B)$ defined in \eqref{eq:CB} is a rainbow cycle.
This can be seen as follows:
By~(a) and \eqref{eq:CB}, any two consecutive sets in $C$ differ in exactly one transposition.
Here we use that 
$B_\ell\triangle \sigma(B_1)=\{1,b_\ell\}\triangle \{1,2\}=\{b_\ell,2\}$ and that $B=\sigma^{2\ell+1}(B)$.
To prove that $C$ is a cycle in $G_{n,2}^\tC$, it remains to argue that all pairs in $C$ are distinct.
Note that for every fixed value of $d\in[\ell]$, there are exactly $n$ different pairs $A\in C_{n,2}$ with $\dist(A)=d$.
Therefore, by~(b) and \eqref{eq:CB}, every pair $A\in C_{n,2}$ appears exactly once in $C$.
We now argue that $C$ is a rainbow cycle, i.e., that every transposition appears exactly once along the cycle $C$. The argument here is very similar.
For every fixed value of $d\in[\ell]$, there are exactly $n$ different transpositions $T\in C_{n,2}$ with $\dist(T)=d$.
Therefore, by (c) and \eqref{eq:CB}, every transposition $T\in C_{n,2}$ appears exactly once along $C$.
This proves that $C$ is indeed a rainbow cycle in $G_{n,2}^\tC$.

It remains to show how to construct a rainbow block satisfying the conditions (a)--(c).
For this it suffices to define values for the elements $b_i$, $i=2,3,\dots,\ell$; recall that $B_i=\{1,b_i\}$ and $b_1=n$.
For even $\ell$ we define
\begin{subequations}
\label{eq:di}
\begin{multline}
\label{eq:di-meven}
  (d_1,d_2,\dots,d_{\ell-1}):= \\
  \begin{cases}
  \big({+}3,-5,+7,\dots,-(\ell-3),+(\ell-1),+1,+(\ell-2),-(\ell-4),\dots,-4,+2\big) & \text{if } \ell\equiv 0\bmod{4} , \\
  \big({+}3,-5,+7,\dots,+(\ell-3),-(\ell-1),-1,-(\ell-2),+(\ell-4),\dots,-4,+2\big) & \text{if } \ell\equiv 2\bmod{4} ,
  \end{cases}
\end{multline}
and for odd $\ell$ we define
\begin{multline}
\label{eq:di-modd}
  (d_1,d_2,\dots,d_{\ell-1}):= \\
  \begin{cases}
  \big({+}3,-5,+7,\dots,+(\ell-2),-\ell,-1,-(\ell-3),+(\ell-5),\dots,+4,-2\big) & \text{if } \ell\equiv 1\bmod{4} , \\
  \big({+}3,-5,+7,\dots,-(\ell-2),+\ell,+1,+(\ell-3),-(\ell-5),\dots,+4,-2\big) & \text{if } \ell\equiv 3\bmod{4} .
  \end{cases}
\end{multline}
\end{subequations}
Using that $b_1=n$, we then define for all $i\in[\ell-1]$
\begin{equation}
\label{eq:bi-rec}
  b_{i+1}:=b_i+d_i=b_1+\sum_{1\leq j\leq i}d_j \enspace \bmod{n}.
\end{equation}

This definition yields a sequence $B=(B_1,B_2,\dots,B_\ell)$ where $B_i=\{1,b_i\}$ for $i\in[\ell]$.
It can be easily verified using \eqref{eq:di} and \eqref{eq:bi-rec} that this sequence satisfies condition~(a), and that $\dist(B_i)=i$ for all $i\in[\ell]$, proving (b).
Moreover, we have
\begin{equation}
\label{eq:bell}
b_\ell=
\begin{cases}
\ell+2 & \text{if $\ell$ is even} , \\
\ell+1 & \text{if $\ell$ is odd} .
\end{cases}
\end{equation}
From these definitions it also follows that $\dist(B_i\triangle B_{i+1})=\dist(\{b_i,b_{i+1}\})=|d_i|$ for all $i\in[\ell-1]$.
For even $\ell$, it therefore follows from \eqref{eq:di-meven} that the set $\{\dist(B_i\triangle B_{i+1})\mid i\in[\ell-1]\}$ contains all numbers $\{1,2,\dots,\ell\}$ except $\ell$.
On the other hand, for odd $\ell$, it follows from \eqref{eq:di-modd} that this set contains all numbers $\{1,2,\dots,\ell\}$ except $\ell-1$.
These missing numbers are contributed by
\begin{equation*}
\dist(B_\ell\triangle \sigma(B_1))=\dist(\{b_\ell,2\}) \eqBy{eq:bell}
\begin{cases}
\ell   & \text{if $\ell$ is even} , \\
\ell-1 & \text{if $\ell$ is odd} ,
\end{cases}
\end{equation*}
so the sequence $B$ indeed satisfies (c).
This proves that $B$ is a rainbow block, so the cycle $C(B)$ defined in \eqref{eq:CB} is a rainbow cycle in $G_{n,2}^\tC$.
\end{proof}

As each pair and each transposition of a rainbow cycle $C$ in $G_{n,2}^\tC$ correspond to an edge of $K_n$, such a rainbow cycle has a nice interpretation as an iterative two-coloring of the edges of $K_n$; see Figure~\ref{fig:mobius}.
One color class are the pairs along the cycle, and the other color class are the transpositions along the cycle.
At each point we consider two consecutive pairs $\{x,y\}$ and $\{x,z\}$ in $C$ and the transposition $\{y,z\}$ between them. 
This means that the edges $\{x,y\}$ and $\{x,z\}$ in $K_n$ corresponding to the pairs share the vertex $x$, and the edge $\{y,z\}$ corresponding to the transposition goes between the other end vertices of these two edges, as depicted on the left hand side of Figure~\ref{fig:mobius}.
The rainbow cycle shown on the right hand side of the figure is $C(B)$ with the rainbow block $B$ as defined in the preceding proof of Theorem~\ref{thm:comb}~(ii) for $n=7$.

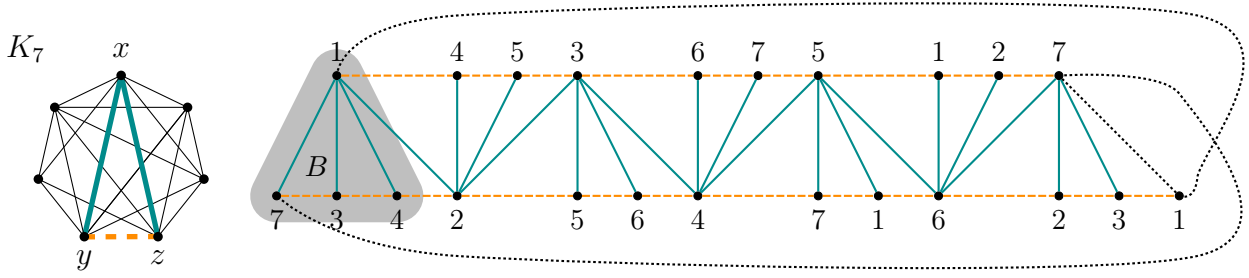
\begin{figure}
\begin{minipage}[c]{.19\textwidth}
 \includegraphics{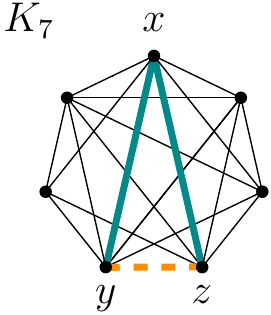}
\end{minipage}
\begin{minipage}[c]{.8\textwidth}
 \begin{tikzpicture}[scale=.8]

\path[draw=black!25!white, fill=black!25!white , line width=20pt, line cap=round, line join=round] (0,2) -- (-1,0) -- (1,0) -- cycle;

\node[anchor=west] at (-0.7,0.5) {$B$};

\node[node_black, label={above:1}] (a1) at (0,2) {};
\node[node_black, label={above:4}] (b1) at (2,2) {};
\node[node_black, label={above:5}] (c1) at (3,2) {};
\node[node_black, label={above:3}] (d1) at (4,2) {};
\node[node_black, label={above:6}] (e1) at (6,2) {};
\node[node_black, label={above:7}] (f1) at (7,2) {};
\node[node_black, label={above:5}] (g1) at (8,2) {};
\node[node_black, label={above:1}] (h1) at (10,2) {};
\node[node_black, label={above:2}] (i1) at (11,2) {};
\node[node_black, label={above:7}] (j1) at (12,2) {};

\node[node_black, label={below:7}] (a2) at (-1,0) {};
\node[node_black, label={below:3}] (b2) at (0,0) {};
\node[node_black, label={below:4}] (c2) at (1,0) {};
\node[node_black, label={below:2}] (d2) at (2,0) {};
\node[node_black, label={below:5}] (e2) at (4,0) {};
\node[node_black, label={below:6}] (f2) at (5,0) {};
\node[node_black, label={below:4}] (g2) at (6,0) {};
\node[node_black, label={below:7}] (h2) at (8,0) {};
\node[node_black, label={below:1}] (i2) at (9,0) {};
\node[node_black, label={below:6}] (j2) at (10,0) {};
\node[node_black, label={below:2}] (k2) at (12,0) {};
\node[node_black, label={below:3}] (l2) at (13,0) {};
\node[node_black, label={below:1}] (m2) at (14,0) {};

\path[line_dashed,dorange] (a1) to (b1);
\path[line_dashed,dorange] (b1) to (c1);
\path[line_dashed,dorange] (c1) to (d1);
\path[line_dashed,dorange] (d1) to (e1);
\path[line_dashed,dorange] (e1) to (f1);
\path[line_dashed,dorange] (f1) to (g1);
\path[line_dashed,dorange] (g1) to (h1);
\path[line_dashed,dorange] (h1) to (i1);
\path[line_dashed,dorange] (i1) to (j1);

\path[line_dashed,dorange] (a2) to (b2);
\path[line_dashed,dorange] (b2) to (c2);
\path[line_dashed,dorange] (c2) to (d2);
\path[line_dashed,dorange] (d2) to (e2);
\path[line_dashed,dorange] (e2) to (f2);
\path[line_dashed,dorange] (f2) to (g2);
\path[line_dashed,dorange] (g2) to (h2);
\path[line_dashed,dorange] (h2) to (i2);
\path[line_dashed,dorange] (i2) to (j2);
\path[line_dashed,dorange] (j2) to (k2);
\path[line_dashed,dorange] (k2) to (l2);
\path[line_dashed,dorange] (l2) to (m2);

\draw[line_dotted] plot [smooth] coordinates {(j1) (14.,1.7) (14.3,-1) (2,-1) (a2)};
\draw[line_dotted] plot [smooth] coordinates {(m2) (14.3,.3) (14.3,3) (2,3) (a1)};

\path[line_solid,dcyan] (a1) to (a2);
\path[line_solid,dcyan] (a1) to (b2);
\path[line_solid,dcyan] (a1) to (c2);
\path[line_solid,dcyan] (a1) to (d2);
\path[line_solid,dcyan] (b1) to (d2);
\path[line_solid,dcyan] (c1) to (d2);
\path[line_solid,dcyan] (d1) to (d2);
\path[line_solid,dcyan] (d1) to (e2);
\path[line_solid,dcyan] (d1) to (f2);
\path[line_solid,dcyan] (d1) to (g2);
\path[line_solid,dcyan] (e1) to (g2);
\path[line_solid,dcyan] (f1) to (g2);
\path[line_solid,dcyan] (g1) to (g2);
\path[line_solid,dcyan] (g1) to (h2);
\path[line_solid,dcyan] (g1) to (i2);
\path[line_solid,dcyan] (g1) to (j2);
\path[line_solid,dcyan] (h1) to (j2);
\path[line_solid,dcyan] (i1) to (j2);
\path[line_solid,dcyan] (j1) to (j2);
\path[line_solid,dcyan] (j1) to (k2);
\path[line_solid,dcyan] (j1) to (l2);
\path[line_dotted] (j1) to (m2);

\end{tikzpicture}
\end{minipage}
\caption{
Interpretation of the rainbow cycle constructed in the proof of Theorem~\ref{thm:comb}~(ii) as an iterative two-coloring of the edges of $K_n$ for $n=7$ ($\ell=3$).
The pairs in the rainbow block $B=(\{1,7\},\{1,3\},\{1,4\})$ are highlighted in gray.
In the drawing on the right, each vertex of $K_n$ appears multiple times; different copies must be identified.
In particular, the leftmost and rightmost edge in the drawing are identified as on a M\"obius strip, indicated by the dotted lines.
Solid edges represent pairs and dashed edges represent transpositions along the rainbow cycle.
The solid edges form a caterpillar, and the dashed edges a Eulerian cycle in $K_n$.}
\label{fig:mobius}
\end{figure} 

\subsection{Proof of Theorem~\ref{thm:comb}~(iii)}

We now extend the method from the proof in the previous section and show that it even yields two edge-disjoint rainbow Hamilton cycles in $G_{n,2}^\tC$.

We call a sequence of numbers $d=(d_1,d_2,\dots,d_\ell)$ with $-\ell\leq d_i\leq \ell$ a \emph{rainbow sequence}, if the sequence $B=(B_1,B_2,\dots,B_\ell)$ with pairs $B_i=\{1,b_i\}$, $i\in[\ell]$, with $b_i$ as defined in \eqref{eq:bi-rec}, satisfies the conditions (a)--(c) in the proof of Theorem~\ref{thm:comb}~(ii) and if $\dist(\{b_\ell,2\})=|d_\ell|$.
The last entry $d_\ell$ of the sequence is determined by the previous entries and by condition~(c).
In particular, we have $\{|d_i| \mid i\in[\ell]\}=[\ell]$.
Recall that the elements $d_i$ of a rainbow sequence are the increments/decrements by which $b_{i+1}$ differs from $b_i$ for $i=1,2,\dots,\ell-1$, and by which $\sigma(1)=2$ differs from $b_\ell$.
In the previous proof we showed that such a rainbow sequence $d$ gives rise to a rainbow cycle $C(B)$.
We let $C(d)$ denote the rainbow Hamilton cycle defined in \eqref{eq:CB} for the rainbow block $B=B(d)$ defined via the rainbow sequence $d$.

\begin{proof}[Proof of Theorem~\ref{thm:comb}~(iii).]
Observe that if $d=(d_1,d_2,\dots,d_\ell)$ is any rainbow sequence, then the reversed sequence $\rev(d):=(d_\ell,d_{\ell-1},\dots,d_1)$ is also a rainbow sequence different from $d$.
In particular, the number of rainbow sequences is always even.

To prove the theorem, we will show that for any rainbow sequence $d$, the rainbow Hamilton cycles $C(d)$ and $C(d')$ with $d':=\rev(d)$ are edge-disjoint cycles in the graph $G_{n,2}^\tC$.
We let $b_i$ and $b_i'$, $i=1,2,\dots,\ell$, denote the values defined in \eqref{eq:bi-rec} for the sequences $d_i$ and $d'_i$, respectively.
As $d'=\rev(d)$ we clearly have
\begin{subequations}
\begin{align}
\begin{alignedat}{2}
d_i &= d'_{\ell-i+1} , \quad & 1 &\leq i \leq \ell , \label{eq:direv} \\
b_i &= n-b'_{\ell-i+2}+2 , \quad & 2 &\leq i \leq \ell .
\end{alignedat}
\end{align}
\end{subequations}

Consider any pair $A\in C_{n,2}$ visited by the cycle $C(d)$, and let $\alpha,\beta\in[n]$ and $i\in[\ell]$ be such that $A=\{\alpha,\beta\}=\sigma^{\alpha-1}(\{1,b_i\})$.
Note that $\alpha$, $\beta$ and $i$ are uniquely determined by \eqref{eq:CB}, and $\alpha$ is not necessarily smaller than $\beta$.
Let $A^-,A^+\in C_{n,2}$ be the pairs that precede and that follow the pair $A$ on the cycle $C(d)$, respectively.
See Figure~\ref{fig:hcs} for an illustration.

\begin{figure}
\begin{center}
\includegraphics[scale=1]{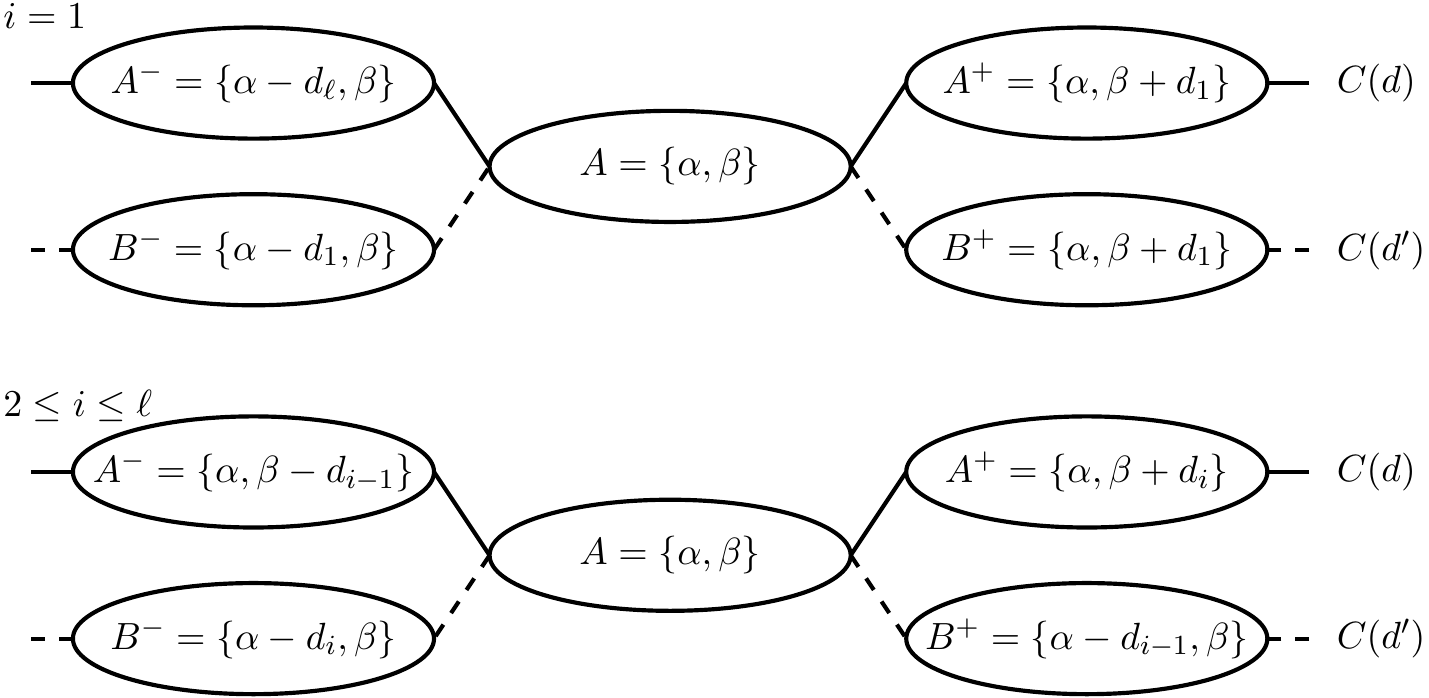}
\caption{Illustration of the proof of Theorem~\ref{thm:comb}~(iii).
The figure shows the predecessors and successors of a pair $A=\{\alpha,\beta\}\in C_{n,2}$ on the cycles $C(d)$ (solid edges) and $C(d')$ (dashed edges) in the cases $i=1$ (top) and $2\leq i\leq \ell$ (bottom).}
\label{fig:hcs}
\end{center}
\end{figure}
\end{proof}

By \eqref{eq:CB} and \eqref{eq:bi-rec} we have
\begin{subequations}\label{eq:A+-}
\begin{align}
A^- &= \begin{cases}
       \{\alpha-d_\ell \bmod{n},\beta\} & \text{if } i=1 , \\
       \{\alpha,\beta-d_{i-1} \bmod{n} \} & \text{if } 2\leq i\leq \ell ,
       \end{cases} \\
A^+ &= \{\alpha,\beta+d_i \bmod{n} \} .
\end{align}
\end{subequations}

We now consider the pairs $B^-$ and $B^+$ that precede and follow the pair $A$ on the cycle $C(d')$, respectively.
By \eqref{eq:CB}, \eqref{eq:bi-rec} and \eqref{eq:direv} we have
\begin{subequations}\label{eq:B+-}
\begin{align}
B^- &= \{\alpha-d'_{\ell-i+1} \bmod{n},\beta\}=\{\alpha-d_i \bmod{n},\beta\} , \\
B^+ &= \begin{cases}
       \{\alpha,\beta+d_\ell \bmod{n}\} & \text{if } i=1 , \\
       \{\alpha+d'_{\ell-i+2} \bmod{n},\beta\}=\{\alpha+d_{i-1} \bmod{n},\beta\} & \text{if } 2\leq i\leq \ell .
       \end{cases}
\end{align}
\end{subequations}
Using that $d_1\neq d_\ell$ it follows immediately from \eqref{eq:A+-} and \eqref{eq:B+-} that the edge sets $\{\{A^-,A\},\{A,A^+\}\}$ and $\{\{B^-,B\},\{B,B^+\}\}$ are disjoint (see Figure~\ref{fig:hcs}), implying that $C(d)$ and $C(d')$ are edge-disjoint cycles, as claimed.
This completes the proof.

With the help of a computer we determined all rainbow sequences $d$ for $\ell=1,2,\dots,7$, as shown in Table~\ref{tab:hcs}, and we computed the maximum number of Hamilton cycles of the form $C(d)$ that are pairwise edge-disjoint.

\begin{table}
\centering
\caption{Number of rainbow sequences $d$ and maximum number of pairwise edge-disjoint rainbow Hamilton cycles of the form $C(d)$ in the graph $G_{n,2}^\tC$, $n=2\ell+1$, for $\ell=1,2,\dots,7$.
For $\ell=6$, rainbow sequences yielding ten edge-disjoint cycles are shown on the right.
Only five sequences are shown explicitly, the other five are obtained by reversal.
For $\ell\in\{1,2,6\}$ (printed in bold), $G_{n,2}^\tC$ has a 2-factorization where all but one 2-factor are a rainbow Hamilton cycle.}
\label{tab:hcs}
\includegraphics{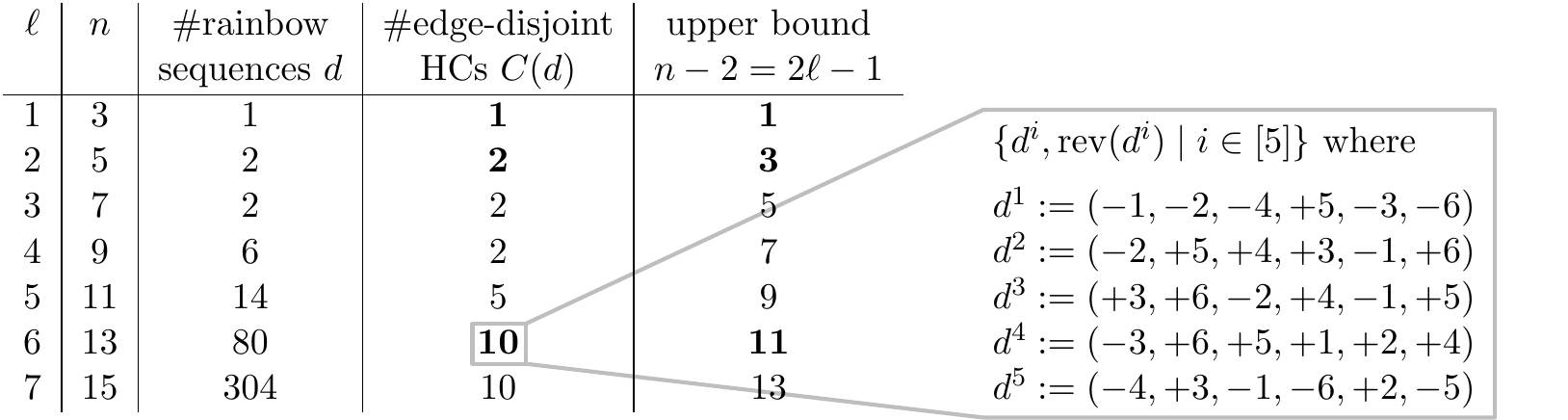}
\end{table}

As the degree of all vertices of $G_{n,2}^\tC$ is $2(n-2)$, the number of edge-disjoint Hamilton cycles in the graph is at most $n-2=2\ell-1$.
If this upper bound is matched, then we obtain a 2-factorization of $G_{n,2}^\tC$ into edge-disjoint rainbow Hamilton cycles.
As the table shows, this happens only in the trivial case $\ell=1$, but it almost happens for $\ell=2$ and $\ell=6$.
In fact, the following theorem shows that $n-3$ edge-disjoint rainbow Hamilton cycles is the maximum we can achieve when considering only cycles of the form $C(d)$ for rainbow sequences $d$.

\begin{theorem}
\label{thm:hcs-ub}
For any $\ell\geq 2$, if $D$ is a set of rainbow sequences such that $C(d)$ and $C(d')$ are edge-disjoint cycles in $G_{n,2}^\tC$, $n=2\ell+1$, for all $d,d'\in D$ with $d\neq d'$, then $|D|\leq n-3$.
\end{theorem}

\begin{proof}
The first pair $B_1$ in a rainbow block is the same $B_1=\{1,b_1\}$, $b_1=n$, for all rainbow sequences $d\in D$, recall condition~(a) from the proof of Theorem~\ref{thm:comb}~(ii).
Moreover, the second pair $B_2$ has the form $\{1,b_2\}$ with $3\leq b_2\leq n$.
By condition~(c) we must have $b_2\neq b_1$, so $3\leq b_2\leq n-1$.
This leaves only $n-3$ different possible values for $b_2$, and the different rainbow sequences $d\in D$ must all have different values $b_2$, otherwise the corresponding cycles $C(d)$ would have the first edge in common.
\end{proof}

Note that even relaxing condition~(a) to $2\leq b_i\leq n$ in the proof of Theorem~\ref{thm:comb}~(ii) does not change the conclusion of Theorem~\ref{thm:hcs-ub}, as the value $b_i=2$ is also forbidden by condition~(b).
However, Theorem~\ref{thm:hcs-ub} does not rule out the existence of a 2-factorization of $G_{n,2}^\tC$ into rainbow Hamilton cycles that are \emph{not} of the form $C(d)$ for some rainbow sequence $d$.

\subsection{Proof of Theorem~\ref{thm:comb}~(iv)}

In this section we describe rainbow cycles in $G_{n,k}^\tC$ for $k\geq 3$.

\begin{proof}[Proof of Theorem~\ref{thm:comb}~(iv).]
Let $n=2\ell+1$ for some integer $\ell\geq 2$.

As before, we construct a rainbow cycle in $G_{n,k}^\tC$ using a rainbow block.
Specifically, we will define a sequence $B=(B_1,B_2,\dots,B_\ell)$ of subsets $B_i\in C_{n,k}$ such that the following conditions hold:
\begin{enumerate}[label=(\alph*),topsep=0mm,leftmargin=7mm]
\item $B_i=[k-1]\cup\{b_i\}$ for $i\in[\ell]$ with $k+1\leq b_i\leq n$ and $b_1=n$,
\item the numbers $b_1,b_2,\dots,b_\ell$ are all distinct, and 
\item $\{\dist(B_i\triangle B_{i+1}) \mid i\in[\ell-1]\}\cup \{\dist(B_\ell\triangle \sigma(B_1))\}=[\ell]$.
\end{enumerate}

We claim that a sequence $B$ satisfying (a)--(c) is a rainbow block, i.e., the cycle $C=C(B)$ defined in \eqref{eq:CB} is a rainbow cycle.
The proof of this fact is very much analogous to the argument in the proof of part~(ii) of the theorem, so we omit it here.

An example of a rainbow block for $\ell=8$ and $k=4$ is shown in Figure~\ref{fig:rblock4}.

\begin{figure}
\centering
\begin{minipage}[c]{.68\textwidth}
\setlength{\arraycolsep}{1.0mm}
$
\begin{array}{l|ccccccccccccccccc|c}
B_i & 1 & 2 & 3 & 4 & 5 & 6 & 7 & 8 & 9 & 10 & 11 & 12 & 13 & 14 & 15 & 16 & 17 & d_i \\  \hline \relax
B_1 & \times & \times & \times & & & & & & & & & & & & & & \times & \downshift{5} \\ \relax
B_2 & \times & \times & \times & & \times & & & & & & & & & & & & & \downshift{1} \\ \relax
B_3 & \times & \times & \times & & & \times & & & & & & & & & & & & \downshift{8} \\ \relax
B_4 & \times & \times & \times & & & & & & & & & & & \times & & & & \downshift{6} \\ \relax
B_5 & \times & \times & \times & & & & & \times & & & & & & & & & & \downshift{4} \\ \relax
B_6 & \times & \times & \times & & & & & & & & & \times & & & & & & \downshift{3} \\ \relax
B_7 & \times & \times & \times & & & & & & \times & & & & & & & & & \downshift{2} \\ \relax
B_8 & \times & \times & \times & & & & & & & & \times & & & & & & & \downshift{7} \\ \cline{1-18} \relax
\sigma(B_1) & \times & \times & \times & \times & & & & & & & & & & & & & &
\end{array}
$
\end{minipage}
\begin{minipage}[c]{0.19\textwidth}
\includegraphics{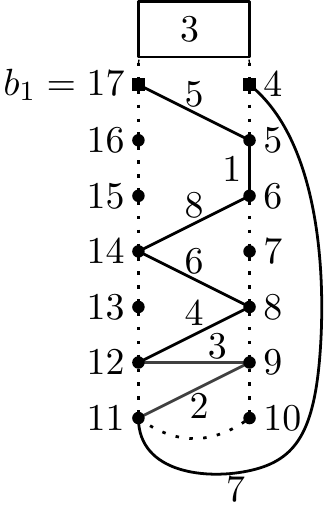}
\end{minipage}
\caption{A rainbow block for $\ell=8$ and $k=4$.
On the right hand side, the sequence $(b_1,b_2,\dots,b_\ell,k)$ for this block is depicted as a path starting with $b_1 = 17$.
The numbers $d_i$ are the edge lengths along the path. 
}
\label{fig:rblock4}
\end{figure}

We interpret a sequence $B$ satisfying these conditions as a path on the vertex set $[n]$ as follows.
Given an edge $\{x,y\}$ of the path, we refer to the quantity $\dist(\{x,y\})\in[\ell]$ as the \emph{length} of this edge.
We say that an edge is \emph{short} if its length is at most~$k-1$, and it is \emph{long} if its length is at least~$k$.

Given a sequence $B$ satisfying the conditions (a)--(c), then the sequence $(b_1,b_2,\dots,b_\ell,k)$ is a simple path of length $\ell$ on the vertex set $[k,n]:=\{k,k+1,\dots,n\}\subseteq [n]$ that starts at the vertex $b_1=n$, ends at the vertex $k$, and that has the property that along the path every edge length from the set~$[\ell]$ appears exactly once; see the right hand side of Figure~\ref{fig:rblock4}.
We refer to such a path as a \emph{rainbow path}.

The following definitions are illustrated in Figure~\ref{fig:zigzag}.

\begin{figure}
\centering
\includegraphics{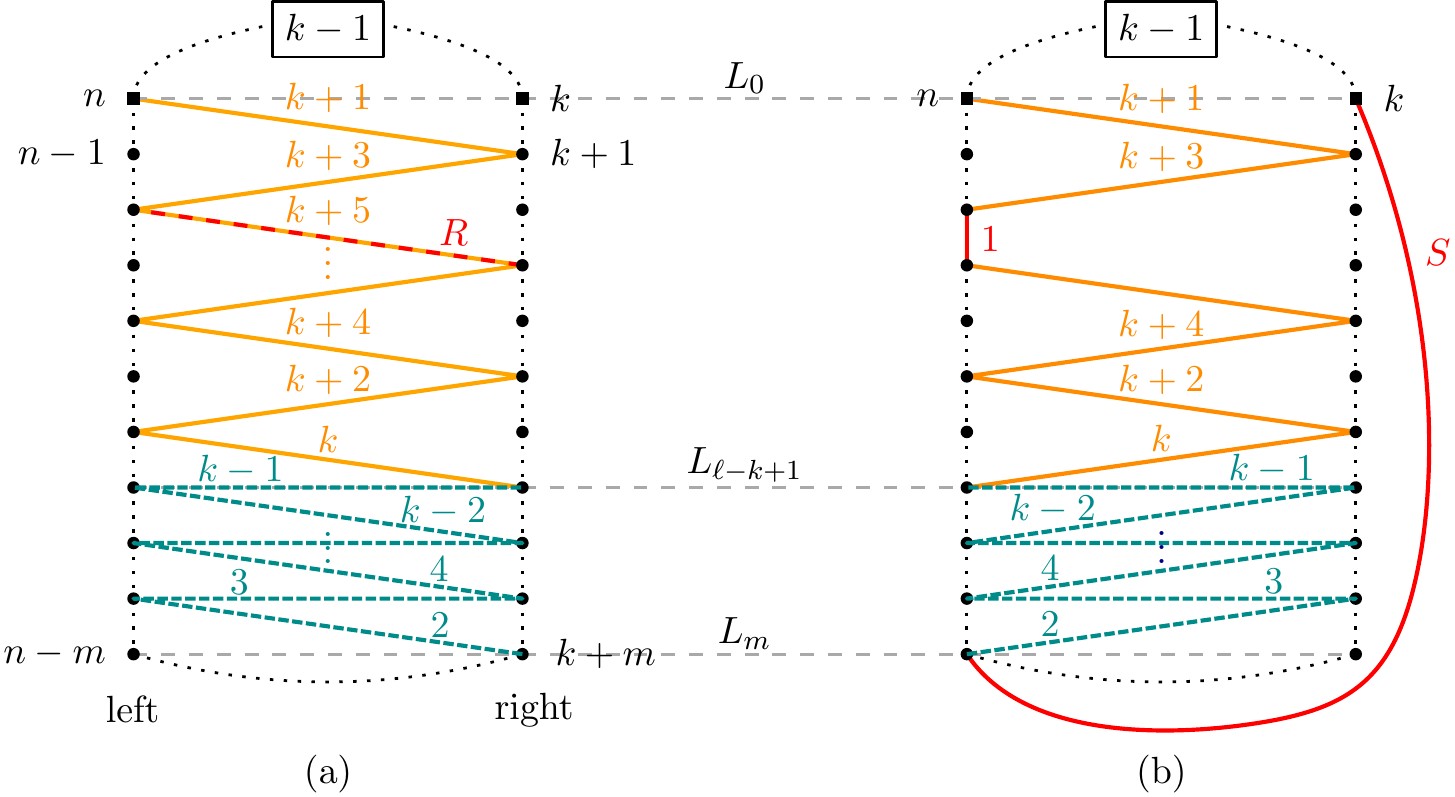}
\caption{(a) A zigzag path for even $\ell$ and even $k$ ($\ell=14$, $k=8$).
The long edges are drawn with bold solid lines, the short edges with bold dashed lines.\newline
(b) A modified zigzag path ($R=k+5=13$ and $S=11$; this is not a rainbow path).}
\label{fig:zigzag}
\end{figure}
We define $L_i:=\{n-i,k+i\}$ and we say that the vertices in $L_i$ belong to \emph{level $i$} where $i=0,1,\dots,m$ and $m:=\ell-\lceil k/2\rceil$.
We refer to the edge $\{n-i,k+i\}$ as a \emph{level edge}, and to any edge of the form $\{i,i+1\}$ as a \emph{cycle edge}.
For any two consecutive levels $L_i$ and $L_{i+1}$, we call the edges $\{n-i,k+i+1\}$ and $\{k+i,n-i-1\}$ \emph{diagonal edges} between these two levels.
Note that these two diagonal edges have the same length.
We refer to the vertices $\{n-i\mid i=0,1,\dots,m\}$ and $\{k+i\mid i=0,1,\dots,m\}$ as \emph{left} and \emph{right vertices}, respectively.
Note that every level contains exactly two vertices, a left and a right vertex.
For odd $k$, there is a unique vertex $(k+m+1)\in [k,n]$ which is not assigned to any level, and which is neither a left nor a right vertex.

\begin{figure}
\centering
% \ell=14, k=8      \ell=14, k=9
% \ell=13, k=6      \ell=15, k=7
\includegraphics{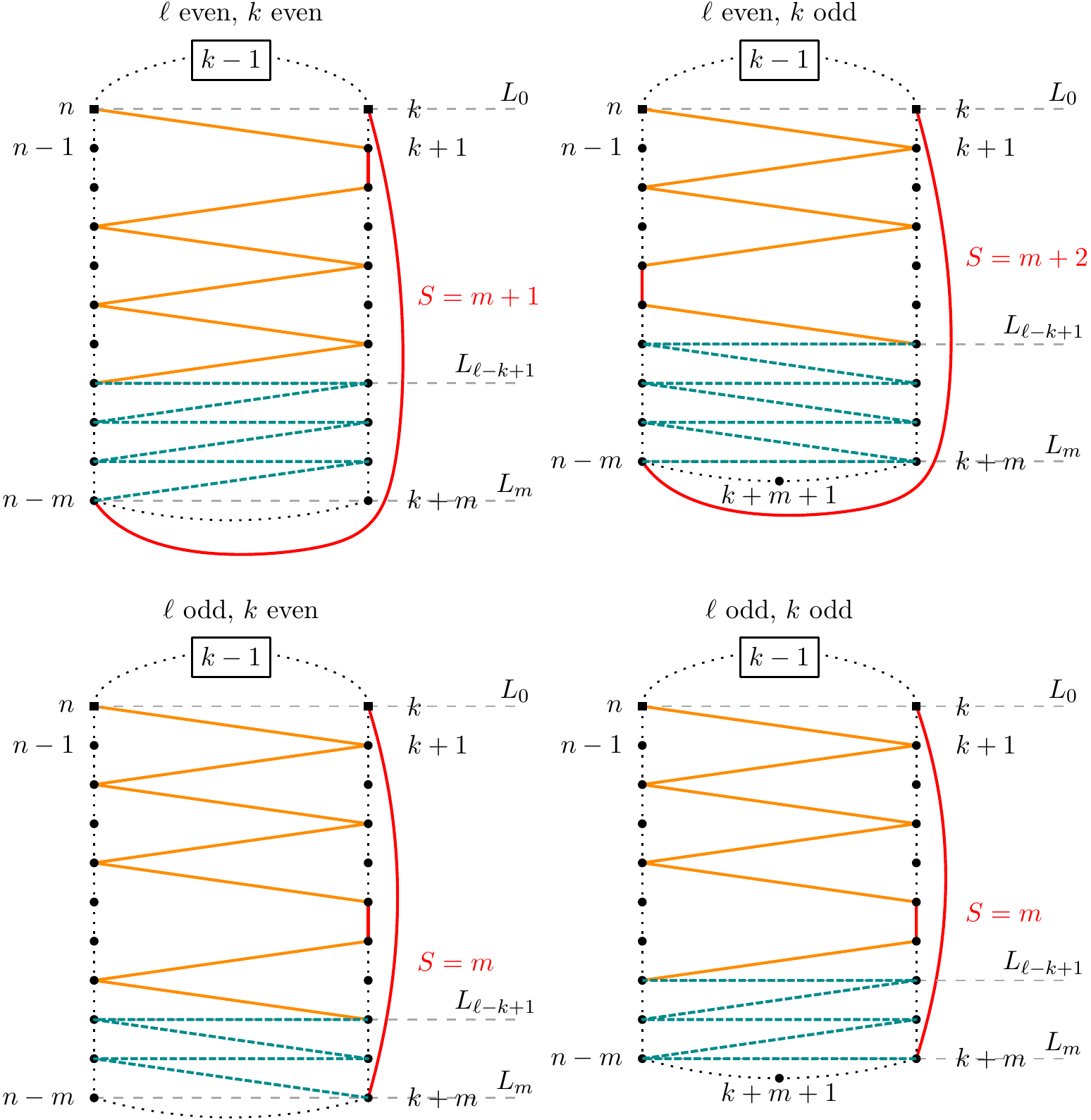}
\caption{Rainbow paths obtained from our construction for all four cases of even/odd $\ell$ and $k$ (specifically, these examples are for $\ell=14$ and $k\in\{8,9\}$, and for $\ell=13$ and $k\in\{6,7\}$).}
\label{fig:zigzag-cases}
\end{figure}

In our figures, we display the vertices in $[k-1]=[n]\setminus [k,n]$ in a box at the top, and the vertices in levels $L_0,L_1,\dots,L_m$ from top to bottom with vertices in the same level drawn at the same vertical position, all left vertices are drawn to the left of all right vertices.
The vertex $(k+m+1)$ for odd $k$ is shown below the level $L_m$.

We construct a rainbow path in two steps; see Figure~\ref{fig:zigzag}.
We first define a \emph{zigzag path}.
This is a path of length $\ell-1$ that starts at the vertex $n$ and ends at a vertex in the last level $L_m$.
The path first alternates between left and right vertices in the levels $L_0,L_1,\dots,L_{\ell-k+1}$ along diagonal edges, and it then alternates between left and right vertices in the levels $L_{\ell-k+1},\dots,L_m$ along level edges and diagonal edges alternatingly, starting with a level edge.

Observe that in the first part the zigzag path uses long edges, namely every length from the set $\{k,k+1,\dots,\ell\}$ exactly once.
Note that as $n=2\ell+1$ is odd, the parity of the edge lengths changes after using the longest or second-longest edge.
In the second part the zigzag path uses short edges, namely every length from the set $\{2,3,\dots,k-1\}$ exactly once.
Note also that a zigzag path ends at a right vertex if $\ell$ is even and at a left vertex if $\ell$ is odd.

We now modify the zigzag path as follows; see Figures~\ref{fig:zigzag} and \ref{fig:zigzag-cases}.
We replace one diagonal edge of some length $R$ by a cycle edge of length 1 between the same two levels, and exchange each vertex after this edge for the other vertex in the same level (in our figures, this corresponds to mirroring the second part of the path after the modification at a vertical line).
In addition, we extend the resulting path with one additional edge, making it a path of length $\ell$, with an edge of length $S$ that leads from the new end vertex in level $L_m$ to the vertex $k$.
It is easy to check that
\begin{equation*}
 S=\begin{cases}
    m & \text{if $\ell$ is odd}, \\
    m+1 & \text{if $\ell$ is even and $k$ is even}, \\
    m+2 & \text{if $\ell$ is even and $k$ is odd}; \\
   \end{cases}
\end{equation*}
see Figure~\ref{fig:zigzag-cases}.
In any case we have $S\geq m$.
Observe that if the edge of length $S$ in a zigzag path is a diagonal, then choosing $R:=S$ in this construction yields a rainbow path, as every edge length from the set $[\ell]$ is used exactly once.

As all long edges along a zigzag path are diagonal edges, requiring that $S\geq k$ is enough to ensure that the edge of length $S$ is a diagonal edge.
This condition is satisfied, as the last inequality in the estimate
\begin{equation*}
S\geq m = \ell-\lceil k/2\rceil \geq \ell-k/2-1/2 \stackrel{!}{>} k-1
\end{equation*}
is equivalent to our assumption $k<n/3=(2\ell+1)/3$.
This completes the proof.
\end{proof}

We remark that there are values of $n$ and $k$ such there is no rainbow cycle of the form $C(B)$ with a rainbow block $B$ satisfying the conditions~(a)--(c) from the previous proof.
For instance, there is no rainbow block of this form for $\ell=k=4$ and $\ell=k=8$.
However, the slightly more general rainbow blocks shown in Figure~\ref{fig:rblock44} and Figure~\ref{fig:rblock88} work in these cases.
These blocks do not satisfy conditions~(a) and (b) stated in the proof, but they satisfy condition~(c).

\begin{figure}[h]
\centering
\setlength{\arraycolsep}{1.0mm}
$
\begin{array}{l|ccccccccc|c}
B_i         & 1 & 2 & 3 & 4 & 5 & 6 & 7 & 8 & 9 & \\ \hline \relax
B_1         & \times & \times & \times &        &        &        &        & & \times & \downshift{4} \\ \relax
B_2         & \times & \times &        &        &        &        & \times & & \times & \downshift{3} \\ \relax
B_3         & \times & \times & \times &        &        &        & \times & &        & \downshift{2} \\ \relax
B_4         & \times & \times & \times &        & \times &        &        & &        & \downshift{1} \\ \cline{1-10} \relax
\sigma(B_1) & \times & \times & \times & \times &        &        &        & &        &
\end{array}
$
\caption{A rainbow block for $\ell=k=4$. }
\label{fig:rblock44}
\end{figure}

\begin{figure}[h]
\centering
\setlength{\arraycolsep}{1.0mm}
$
\begin{array}{l|ccccccccccccccccc|c}
B_i         & 1 & 2 & 3 & 4 & 5 & 6 & 7 & 8 & 9 & 10 & 11 & 12 & 13 & 14 & 15 & 16 & 17 & \\ \hline \relax
B_1         & \times & \times & \times & \times & \times & \times & \times & & & & & & & & & & \times & \downshift{8} \\ \relax
B_2         & \times & \times & \times & \times & \times & \times & & & & & & & & & \times & & \times & \downshift{7} \\ \relax
B_3         & \times & \times & \times & \times & \times & \times & \times & & & & & & & & \times & & & \downshift{6} \\ \relax
B_4         & \times & \times & \times & \times & \times & \times & \times & & \times & & & & & & & & & \downshift{5} \\ \relax
B_5         & \times & \times & \times & \times & \times & \times & \times & & & & & & & \times & & & & \downshift{3} \\ \relax
B_6         & \times & \times & \times & \times & \times & \times & \times & & & & \times & & & & & & & \downshift{2} \\ \relax
B_7         & \times & \times & \times & \times & \times & \times & \times & & & & & & \times & & & & & \downshift{1} \\ \relax
B_8         & \times & \times & \times & \times & \times & \times & \times & & & & & \times & & & & & & \downshift{4} \\ \cline{1-18} \relax
\sigma(B_1) & \times & \times & \times & \times & \times & \times & \times & \times & & & & & & & & & &
\end{array}
$
\caption{A rainbow block for $\ell=k=8$. }
\label{fig:rblock88}
\end{figure}

\section{Open problems}
\label{sec:open}

For all the combinatorial classes considered in this paper, it would be very interesting to exhibit $r$-rainbow cycles for larger values of $r$ (recall Table~\ref{tab:results}), in particular for the flip graphs of permutations and subsets.
Another natural next step is to investigate rainbow cycles in other flip graphs, e.g., for non-crossing partitions of a convex point set or for dissections of a convex polygon (see~\cite{MR2510231}).

We believe that the flip graph of non-crossing perfect matchings $G_m^\tM$ has no 1-rainbow cycle for any $m\geq 5$.
This is open for the even values of $m\geq 12$.
Moreover, the subgraph $H_m$ of $G_m^\tM$ restricted to centered flips (see Figure~\ref{fig:HM6}) is a very natural combinatorial object with many interesting properties that deserve further investigation.
What is the number of connected components of $H_m$ and what is their size?
Which components are trees and which components contain cycles?
As a starting point, it would be very nice to prove the conjectured formula \eqref{eq:Mmc-size} for the number of matchings with a certain weight, and to understand the connection to the generalized Narayana numbers.

We conjecture that the flip graph of subsets $G_{n,k}^\tC$ has a 1-rainbow cycle for all $2\leq k\leq n-2$.
This is open for $n/3\leq k\leq 2n/3$.
In view of Theorem~\ref{thm:comb}~(iii) we ask: does $G_{n,2}^\tC$ have a factorization into $n-2$ edge-disjoint rainbow Hamilton cycles?

\section{Acknowledgements}

We thank Manfred Scheucher for his quick assistance in running computer experiments that helped us to find rainbow cycles in small flip graphs.

\bibliographystyle{alpha}
\bibliography{refs}

% \newpage
\appendix 
\section{Proof of Proposition~\ref{prop:trees-small}}

\begin{proof}[Proof of Proposition~\ref{prop:trees-small}]
For $n=4$ points there are two possible configurations to consider: either all four points are in convex position, or three points are on the convex hull and one is in the interior.
Solutions for both configurations are shown in Figure~\ref{fig:trees-n4r2}.

For $n=5$ points we have to consider three different configurations, and in each case we have to construct $r$-rainbow cycles for each $r\in\{2,3,4\}$.
Even though the arguments in the proofs of Proposition~\ref{prop:trees2r} and \ref{prop:trees2r-1} do not work for $n = 5$, we can still use the described constructions with some minor modifications.
Figure~\ref{fig:trees-50} shows such $r$-rainbow cycles for five points in convex position.
Figure~\ref{fig:trees-51} covers the case that four points are on the convex hull and one point is in the interior.
Figure~\ref{fig:trees-52} covers the case that three points are on the convex hull and two points are in the interior.
\end{proof}

\begin{figure}[h]
\centering
\includegraphics{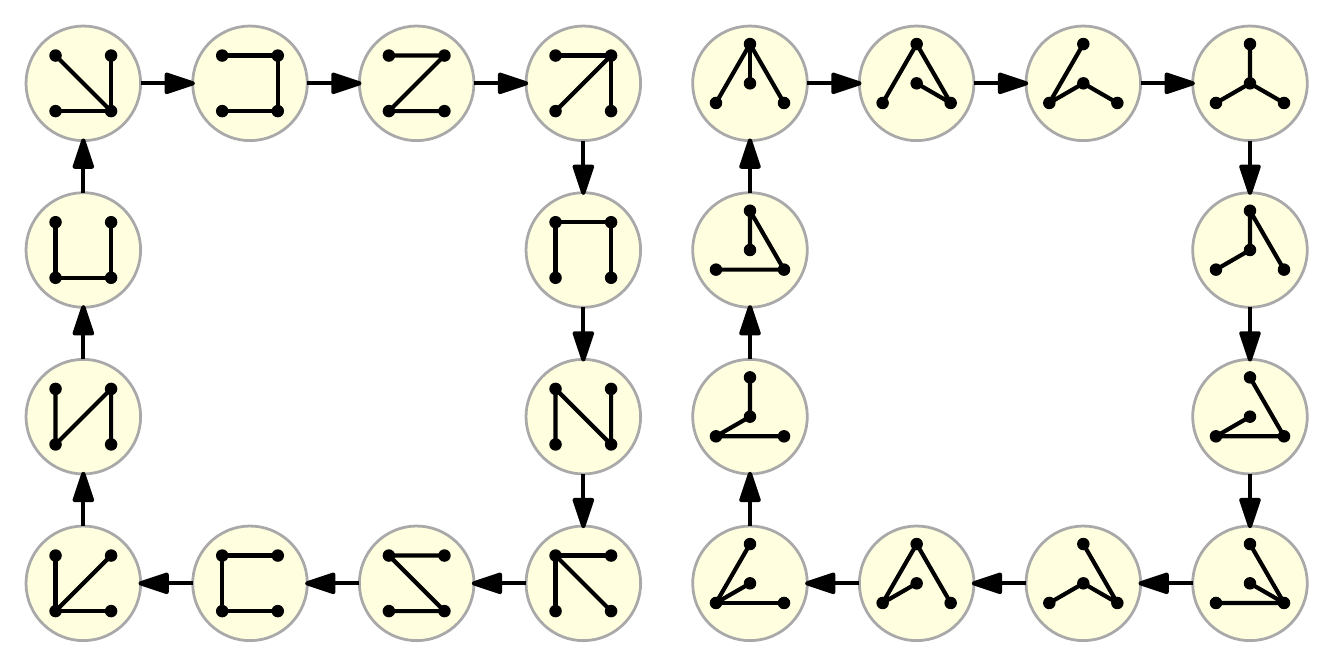}
\caption{The figure shows $2$-rainbow cycles for both configurations of $n=4$ points.}
\label{fig:trees-n4r2}
\end{figure}

\begin{figure}
\centering
\includegraphics{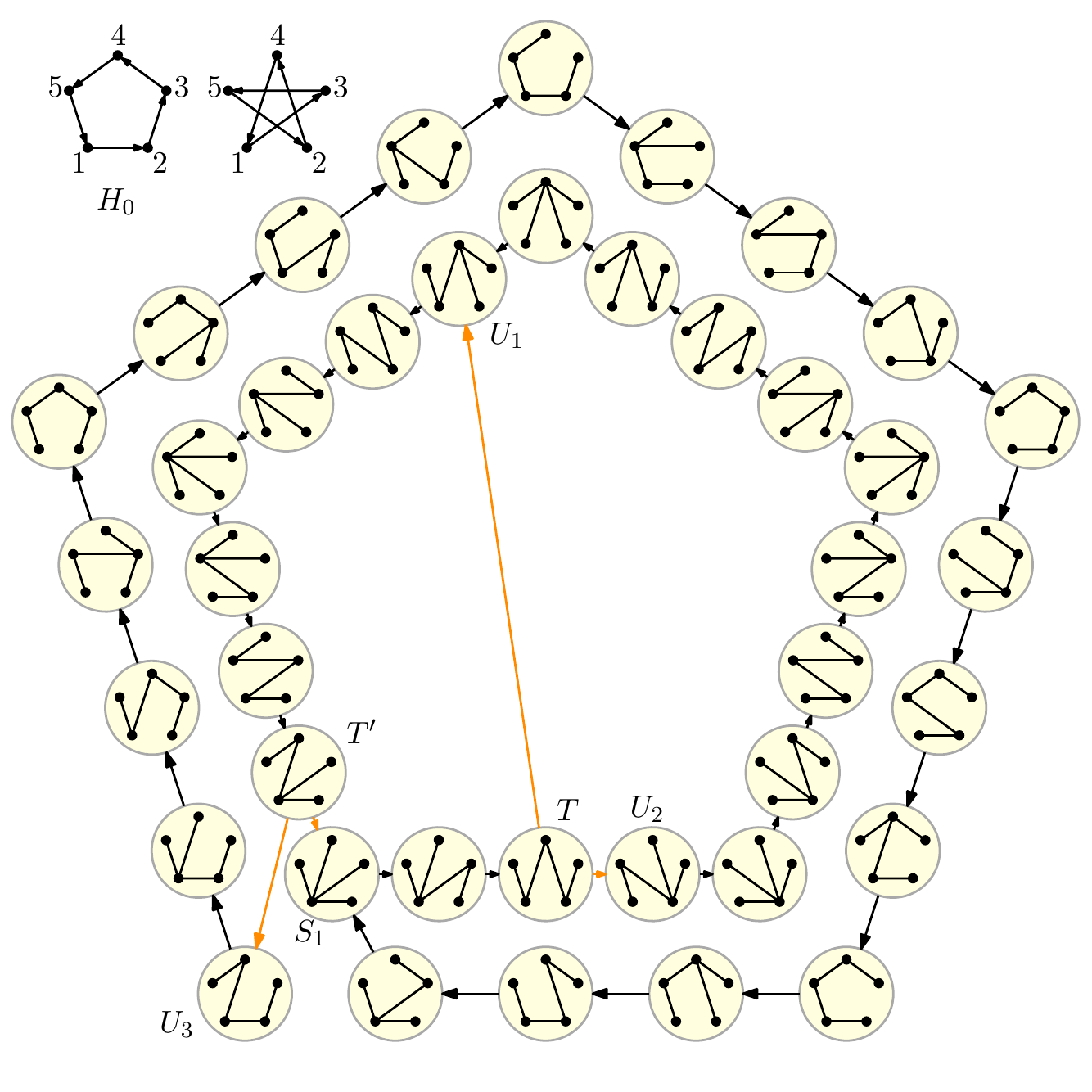}
\caption{Illustration of $1$-, $2$-, $3$-, and $4$-rainbow cycles in $G_X^\tS$, where $X$ is a set of $n=5$ points in convex position.
All four cycles contain $S_1$, but they use different arcs from $T$ and $T'$.
The $1$-rainbow cycle uses the arcs $(T, U_1)$ and $(T', S_1)$.
The $2$-rainbow cycle uses the arcs $(T, U_2)$ and $(T', S_1)$.
The $3$-rainbow cycle uses the arcs $(T, U_1)$ and $(T', U_3)$.
Finally, the $4$-rainbow cycle uses the arcs $(T, U_2)$ and $(T',U_3)$.\\
The Hamilton cycles in $K_X$ and their orientation used in the construction as in the proofs of Proposition~\ref{prop:trees2r} and \ref{prop:trees2r-1} are shown in the upper left corner.}
\label{fig:trees-50}
\end{figure}

\begin{figure}
\centering
\includegraphics{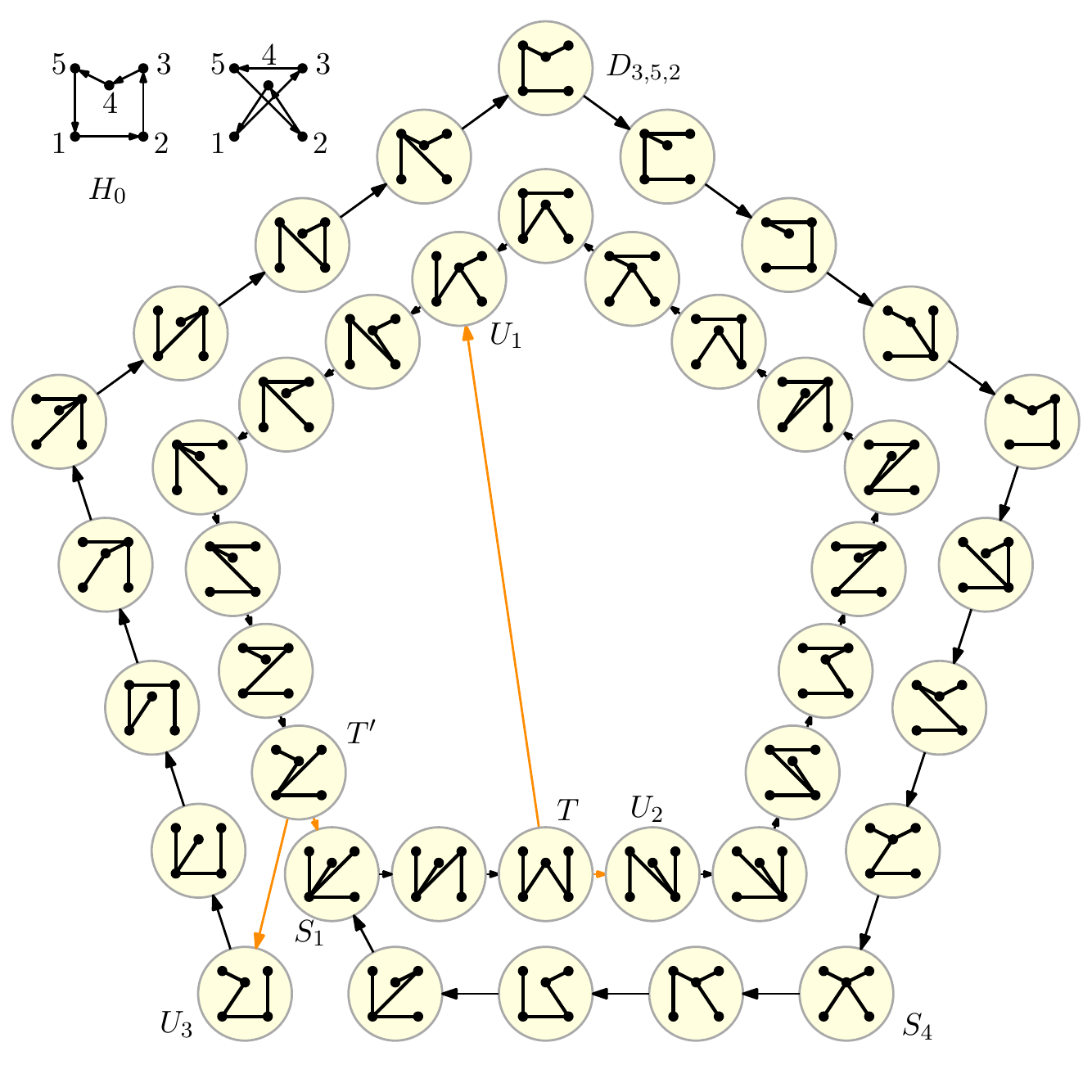}
\caption{
Illustration of $1$-, $2$-, $3$-, and $4$-rainbow cycles in $G_X^\tS$, where $X$ is a set of $n=5$ points with exactly four points on the convex hull.
All four cycles contain $S_1$, but they use different arcs from $T$ and $T'$.
The $1$-rainbow cycle uses the arcs $(T, U_1)$ and $(T', S_1)$.
The $2$-rainbow cycle uses the arcs $(T, U_2)$ and $(T', S_1)$.
The $3$-rainbow cycle uses the arcs $(T, U_1)$ and $(T', U_3)$.
Finally, the $4$-rainbow cycle uses the arcs $(T, U_2)$ and $(T', U_3)$.
Note that in contrast to the construction in the proof of Proposition~\ref{prop:trees2r}, we replaced $D_{2,4,1}$ at the bottom right by $S_4$, since $D_{2,4,1} = D_{3,5,2}$, and $S_4$ does not occur anywhere else.}
\label{fig:trees-51}
\end{figure}

\begin{figure}
\centering
\includegraphics{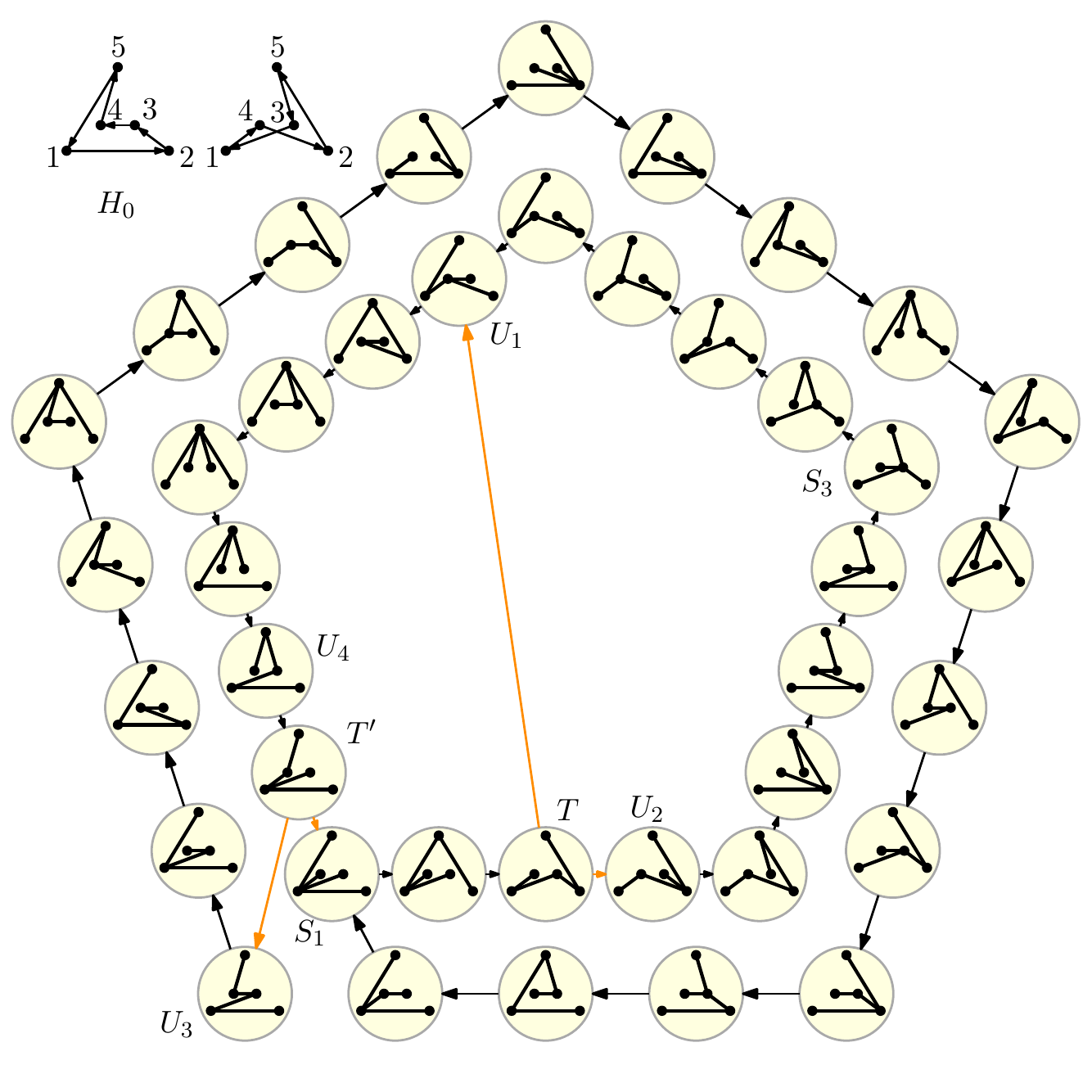}
\caption{Illustration of $1$-, $2$-, $3$-, and $4$-rainbow cycles in $G_X^\tS$, where $X$ is a set of $n=5$ points with exactly three points on the convex hull.
All cycles contain $S_1$, but they use different arcs from $T$ and $T'$.
The $1$-rainbow cycle uses the arcs $(T, U_1)$ and $(T', S_1)$.
The $2$-rainbow cycle uses the arcs $(T, U_2)$ and $(T', S_1)$.
The $3$-rainbow cycle uses the arcs $(T, U_1)$ and $(T', U_3)$.
Finally, the $4$-rainbow cycle uses the arcs $(T, U_2)$ and $(T', U_3)$.
Note that $D_{2,3,4}$ would coincide with the tree $U_4$ and is therefore replaced by $S_3$, which does not occur anywhere else.}
\label{fig:trees-52}
\end{figure}

\end{document}